\documentclass[11pt]{amsart}

\usepackage{epigamath}

\usepackage[english]{babel}

\numberwithin{equation}{section}

\usepackage[shortlabels]{enumitem}
\setlist[enumerate,1]{label={\rm(\arabic*)}, ref={\rm\arabic*}} 
\setlist[enumerate,2]{label={\rm(\alph*)}, ref={\rm\alph*}} 

\usepackage{amsfonts}
\usepackage{tikz-cd}
\usepackage{amsthm}
\usepackage{float}
\usepackage{colortbl}
\usepackage{hhline}
\usepackage{multirow}
\usepackage{rotating}
\usepackage{fancybox}
\usepackage{multicol}
\usepackage{wrapfig}
\usepackage{subcaption}
\usepackage{mathtools}
\usepackage{tikz}
\usetikzlibrary{graphs,decorations.pathmorphing,decorations.markings,babel}
\usepackage[capitalize]{cleveref}
\usepackage{caption} 
\usepackage{adjustbox}

\newtheorem{theo}{Theorem}[section]
\newtheorem*{theo*}{Theorem}
\newtheorem{lem}[theo]{Lemma}
\newtheorem{coro}[theo]{Corollary}
\newtheorem{claim}[theo]{Claim}
\newtheorem{propo}[theo]{Proposition}

\theoremstyle{definition}
\newtheorem{defin}[theo]{Definition}

\theoremstyle{remark}
\newtheorem{rem}[theo]{Remark}
\newtheorem*{notat}{Notation}
\newtheorem*{problem}{Problem}
\newtheorem{ex}[theo]{Example}

\Crefname{theo}{Theorem}{Theorems}
\crefname{rem}{Remark}{Remarks}
\Crefname{lem}{Lemma}{Lemmas}
\Crefname{coro}{Corollary}{Corollaries}
\Crefname{propo}{Proposition}{Propositions}
\Crefname{claim}{Claim}{Claims}
\Crefname{ex}{Example}{Examples}

\definecolor{lg}{gray}{0.9}
\usepackage[first=0,last=9]{lcg}
\newcommand{\II}{\mathrm{II}}
\newcommand{\T}{\mathcal{T}}
\newcommand{\Z}{\mathbb{Z}}
\DeclareMathOperator{\id}{id}
\DeclareMathOperator{\Mon}{Mon}
\DeclareMathOperator{\NS}{NS}
\DeclareMathOperator{\Aut}{Aut}
\DeclareMathOperator{\Bir}{Bir}
\DeclareMathOperator{\Tr}{Tr}
\DeclareMathOperator{\rank}{rank}
\DeclareMathOperator{\val}{val}
\DeclareMathOperator{\lcm}{lcm}
\DeclareMathOperator{\Gal}{Gal}
\DeclareMathOperator{\spin}{spin}
\DeclareMathOperator{\Def}{Def}
\DeclareMathOperator{\ord}{ord}
\DeclareMathOperator{\im}{im}
\DeclareMathOperator{\SO}{SO}

\newcommand{\pex}{\mathrm{pex}}
\newcommand{\prim}{\mathrm{prim}}

\renewcommand{\div}{\operatorname{div}}

\newcommand{\lra}{\longrightarrow}
\def\lowsimeq{\vbox to 0pt{\vss\hbox{$\scriptstyle\simeq$}\vskip-1.5pt}}

\newcommand{\supth}[1]{\ensuremath{#1^{\mathrm{th}}}}

\EpigaVolumeYear{10}{2026} \EpigaArticleNr{20} \ReceivedOn{February 6, 2025}
\InFinalFormOn{February 17, 2026}
\AcceptedOn{May 19, 2026}

\title{Algebraically trivial automorphisms of irreducible holomorphic symplectic manifolds}
\titlemark{Algebraically trivial automorphisms of irreducible holomorphic symplectic manifolds}

\author{Stevell Muller}
\address{Mathematik und Informatik, Universit\"at des Saarlandes, 66123 Saarbr\"ucken, Germany\\
\textit{Current address:} Institut f\"ur Algebraische Geometrie, Leibniz Universit\"at Hannover, 30167 Hannover, Germany}
\email{muller@math.uni-hannover.de}

\authormark{S.~Muller}

\AbstractInEnglish{Given an irreducible holomorphic symplectic (IHS) manifold $X$, we call an automorphism~$f$ of $X$ algebraically trivial if $f$ acts trivially on the N\'eron--Severi group of $X$. In this paper, we extend the approach of Brandhorst--Cattaneo to classify, by means of lattice data, all the pairs $(X, G)$, where $X$ is an IHS manifold of a known deformation type (to this date), and $G$ is a finite group consisting of algebraically trivial automorphisms of $X$ that are nonsymplectic.
In particular, we classify even-order algebraically trivial nonsymplectic automorphisms, with or without trivial discriminant action.
In the case of nontrivial discriminant actions, we show that such automorphisms only exist for finitely many deformation types and orders. }


\acknowledgement{Funded by the Deutsche Forschungsgemeinschaft (DFG, German Research Foundation), Projektnummer 286237555 - TRR 195.}

\begin{document}

%%%%%%%%%%%%%%%%%%%%%%%%%%%%%%%
% Title page
%%%%%%%%%%%%%%%%%%%%%%%%%%%%%%%

%\removeabove{}
%\removebetween{}
%\removebelow{}

\maketitle

\begin{prelims}

\DisplayAbstractInEnglish

\bigskip

\DisplayKeyWords

\medskip

\DisplayMSCclass

\end{prelims}

%%%%%%%%%%%%%%%%%%%%%
% Table of Contents
%%%%%%%%%%%%%%%%%%%%%

\newpage

\setcounter{tocdepth}{1}

\tableofcontents

%%%%%%%%%%%%%%%%%%%%%
% Content begins here
%%%%%%%%%%%%%%%%%%%%%

\section{Introduction}

Let $X$ be an \emph{irreducible holomorphic symplectic} (\emph{IHS}$\mkern2mu$) manifold; \textit{i.e.}, $X$ is a simply connected compact K\"ahler manifold admitting a unique, up to scaling, holomorphic 2-form that is nowhere degenerate. The second integral cohomology of $X$ is equipped with the so-called Beauville--Bogomolov--Fujiki quadratic form $q_X$, turning $H^2(X,\mathbb{Z})$ into an integral nondegenerate $\mathbb{Z}$-lattice. We denote by $\NS(X) \coloneqq  c_1(\operatorname{Pic}(X)) = H^2(X,\mathbb{Z})\cap H^{1,1}(X)$ the primitive sublattice of $H^2(X,\mathbb{Z})$ known as the \emph{N\'eron--Severi lattice} of $X$. In this paper, the IHS manifolds we study are assumed to be projective, and in particular $\NS(X)$ is hyperbolic. There are two well-known orthogonal representations of the group $\Bir(X)$ of birational automorphisms of $X$:
\begin{align*}
    \rho_X\colon\Bir(X)\lra O(H^2(X, \mathbb{Z}))\quad \text{and}\quad \rho_X^{\NS}\colon \Bir(X)\lra O(\NS(X)).
\end{align*}
Their respective kernels are known to be finite (see for instance \cite[Proposition 4.1]{deb22}).

\begin{rem}
  For any IHS manifold $X$, the group $\ker\rho_X\leq \Aut(X)$ consist of automorphisms, and it is finite. However, the group $\ker\rho_X^{\NS}$, which also consists of automorphisms, could be of infinite order. If we assume, though, that $X$ is projective, then the latter group is indeed finite. Note that the converse of the previous statement does not always hold: If  $X$ has Picard rank 0, then it is not projective but $\ker\rho_X^{\NS} = \ker\rho_X$ is finite. 
\end{rem}

The manifold $X$ admits a universal deformation family over a smooth and connected base $\operatorname{Def}(X)$, and any deformation of $X$ over a small disk in $\operatorname{Def}(X)$ is again an IHS manifold. According to \cite[Theorem~2.1]{ht13}, the isomorphism class of $\ker\rho_X$, as a group, is invariant under any such small deformation of~$X$. However, since Hodge structures on second integral cohomology vary along such small deformations, it is not even true that the order of $\ker\rho_X^{\NS}$ is preserved under deformation. In this work, we are interested in the following problem.

\begin{problem}
    Classify all the pairs $(X, G)$, where $X$ is a projective IHS manifold and $G\leq \Aut(X)$ is a finite subgroup satisfying $\ker\rho_X\leq G\leq \ker\rho_X^{\NS}$.
\end{problem}

In order to tackle this problem, we restrict ourselves to studying such pairs $(X, G)$ up to deformation and birational conjugacy. Here, on the one hand, by deformation of the pair $(X, G)$, we mean a deformation $\mathcal{X}\to B$ of $X$ together with a group $\mathcal{G}\leq \Aut(\mathcal{X}/B)$ whose restriction to the central fiber $X$ coincides with~$G$. On the other hand, by birational conjugacy, we mean that we are allowed to change $X$ by any IHS birational model and $G$ by the associated conjugated group. In fact, our methods are lattice-theoretic, and the results we produce cannot, at first, distinguish pairs that are equivalent by the above relations. More generally, our results characterize deformation families of such pairs $(X, G)$: Another pair $(X', G')$ in such a family does not need to satisfy that $G'\leq \ker\rho^{\rm NS}_{X'}$.

As we report later in the introduction, the problem stated above is solved for K3 surfaces, or if the manifold $X$ belongs to one of the known deformation families and the index $[G:\ker\rho_X]$ is trivial or prime. For the rest of the paper, given an IHS manifold $X$ and a subgroup $G\leq \ker\rho_X^{\NS}$, we refer to $G$ as \textit{algebraically trivial}. For a given pair $(X, G)$, where $X$ is a projective IHS manifold and the group $G$ satisfies $\ker\rho_X\leq G\leq \ker\rho_X^{\NS}$, exactly one of the following three cases occurs:
\begin{enumerate}
    \item\label{case1} The group $G = \ker\rho_X$ consists of \emph{symplectic} automorphisms.
    \item\label{case2} The index $[G:\ker\rho_X]$ is prime.
    \item\label{case3} The index $[G:\ker\rho_X]$ is composite.
\end{enumerate}
If $X$ belongs to one of the known deformation families of IHS manifolds (at the time this paper is written), then the pairs $(X, G)$ in case (\ref{case1}) have been completely determined. In fact, in this setting, since $G = \ker\rho_X$ is invariant under deformation, the problem reduces to determining the isomorphism class of $G$ for each known deformation family of IHS manifolds. This has been completely solved in the following list of references: \cite[Propositions 9 and 10]{bea83b}, \cite[Corollary 5]{bnws11} and \cite[Theorems 2.1 and 42.]{mw17}. A strategy to classify pairs $(X, G)$ for the other two cases consists of an application of Torelli-type theorems for IHS manifolds. This is a standard approach that allows us to transform the problem into a lattice-theoretic one.%\medskip

Let $X$ be a projective IHS manifold, and let $\operatorname{T}(X) = \NS(X)^\perp$ be the \emph{transcendental lattice} of $X$. It is, by definition, the smallest primitive sublattice of $H^2(X, \mathbb{Z})$ whose complexification contains $H^{2,0}(X)$. The orthogonal direct sum $\operatorname{T}(X)\oplus\NS(X)$ is a submodule of $H^2(X,\mathbb{Z})$ of finite index. We call a birational automorphism $f\in \Bir(X)$ \emph{symplectic} if $\rho_X(f)$ fixes $\operatorname{T}(X)$ pointwise. With this definition and the previous comment, it follows that the normal subgroup of $\ker\rho_X^{\NS}$ consisting of symplectic algebraically trivial automorphisms coincides with $\ker\rho_X$. Now let  $\ker\rho_X\lneq G\leq \ker\rho_X^{\NS}$ be a nonsymplectic, algebraically trivial, finite subgroup of automorphisms of $X$, and let $g\in G\setminus \ker\rho_X$. In this situation, since $g$ acts nontrivially on $H^2(X,\mathbb{Z})$, the coset $g\ker\rho_X$ generates the cyclic factor group $G/\ker\rho_X$, of finite order $m\coloneqq [G:\ker\rho_X]$; see \cite[Proposition 4.7]{bea83b}. 
Since $G$ is algebraically trivial, so is $g$, and if we set $h \coloneqq  \rho_X(g)$, we obtain that $\NS(X) = \ker(h-\id)$ is the invariant sublattice of $h$. Indeed, according to the proof of \cite[Proposition~4.6]{bea83b}, the restricted isometry $h_{\mid \textnormal{T}(X)}$ has order $m$, and its minimal polynomial is cyclotomic. Consequently, the characteristic polynomial of $h$ is of the form $\Phi_1^\rho(X)\Phi_m^k(X)$, where $\rho \coloneqq  \rank_\mathbb{Z}(\NS(X))$ is the Picard rank of~$X$ and $k$ is a positive integer, and where $\Phi_i$ denotes the $\supth{i}$ cyclotomic polynomial for any positive integer $i\geq 1$.%\smallskip

Via the Torelli-type theorems for IHS manifolds, one can thus treat cases (\ref{case2}) and (\ref{case3}) as before by studying and classifying conjugacy classes of isometries of certain even $\mathbb{Z}$-lattices with minimal polynomial $\Phi_1\Phi_m$ for some $m\geq 2$. The case where $m=2$ has been completed for the known examples of IHS manifolds: See \cite{nik83} for K3 surfaces, \cite{bea11,jou16,ccc20} for the deformation types $\textnormal{K3}^{[n]}$, \cite{tar15,mtw18} for the deformation types $\textnormal{Kum}_n$, and \cite{gro22,bg24} for the sporadic deformation types OG6 and OG10, respectively.

In the case where $m$ is an odd prime number, the case of K3 surfaces has been completely covered by \cite{ast11} with further developments by \cite{mo98,bra19,bh23,bay24a}. In higher dimension, an important part of the work has been carried out by \cite{bcms16,bcs16,cc19} for the $\textnormal{K3}^{[n]}$-types ($n\geq 2$), by \cite{bnws11,tar15,mtw18} for the $\textnormal{Kum}_n$-types ($n\geq 2$), and by \cite{gro22}, who covers the case of OG6-type IHS manifolds.
All of the previous have been recovered by \cite{bc22}, who also treated the case of OG10-type IHS manifolds; see also  \cite{bg24} for this last case. This gives a full solution in case (\ref{case2}) above for the known deformation families of IHS manifolds.

Before this paper, few results were known for $m\geq 3$ composite, except in the case of K3 surfaces. Indeed, the problem of classifying pairs $(X, G)$ as before when $X$ is a K3 surface has been completely solved, as reported in \cite{kon92,tak12,bay24b}, for instance. More generally, nonsymplectic automorphisms of K3 surfaces have been extensively studied: See for instance \cite{sch10,as15,ast16,as18,ags21,acv22} for 2-power orders, \cite{tak10,acv20} for 3-power orders, \cite{dillies} for order 6, \cite{bcl24} for orders divisible by~7, and \cite{bra19,acv22} for orders $n$ with $\varphi(n)\geq 12$ and $\varphi(n)=8, 10$, respectively.\footnote{This list of references is not exhaustive, but it shows that the problem of classifying birational automorphisms on IHS manifolds is still an active field of research, already rich in standard literature.} We refer to \cite{bh23} for a complete classification, up to deformation and conjugacy, of pairs $(X, G)$, where $X$ is a K3 surface and $G\leq \Aut(X)$ is a finite (saturated) group of automorphisms.%\smallskip

The purpose of this work is to complete the classification of pairs $(X, G)$, for $X$ an IHS manifold from one of the known deformation families and $\ker\rho_X\lneq G\leq \ker\rho_X^{\NS}$ satisfying that $[G:\ker\rho_X]$ is composite. 
Following our previous discussion, a way one could use to completely characterize a pair $(X, G)$ up to deformation and birational conjugacy is by determining enough lattice invariants of the conjugacy class of the cyclic group $H\coloneqq \rho_X(G)$. The problem would then reduce to enumerating all possible combinations of said invariants. In our case, some of these lattice invariants are the following:
\begin{enumerate}
    \item The order $m=[G:\ker\rho_X]$ of $H$;
    \item The isometry class of the \emph{invariant lattice} $H^2(X, \mathbb{Z})^H$, consisting of vectors fixed by all isometries in $H$;
    \item The isometry class of the \emph{coinvariant lattice} $H^2(X, \mathbb{Z})_H$, which is the orthogonal complement of the invariant lattice $H^2(X, \mathbb{Z})^H$;
    \item The order for the action of $H$ on the \emph{discriminant group} of $H^2(X, \mathbb{Z})$---for the known deformation families of IHS manifolds, it is known that this order is 1 or 2 (see \Cref{lem:disc action}).
\end{enumerate}

Our two first main results give restrictions on the index $m=[G:\ker\rho_X]$ depending on whether the cyclic group $H\leq O(H^2(X, \mathbb{Z}))$ acts trivially on the discriminant group of $H^2(X, \mathbb{Z})$.

\begin{theo}[see \Cref{mainth1} and \Cref{tab u4,tab u5,tab k3n,tab og10}]\label{thm:12}
    Let $\mathcal{T}$ be a known deformation family of IHS manifolds, and let $m\geq 2$ be a composite positive integer. Then, there exists a pair $(X, G)$ where $X$ belongs to $\mathcal{T}$, the finite group $\ker\rho_X\leq G$ is algebraically trivial with $[G:\ker\rho_X]=m$ and such that $\rho_X(G)$ acts trivially on the discriminant group of $H^2(X, \mathbb{Z})$ if and only if the pair $(\mathcal{T}, m)$ satisfies one of the following conditions:
    \begin{enumerate}
        \item $\mathcal{T}=\operatorname{Kum}_n$ for any $n\geq 2$ and $m\in\{4,6,12\}$,
        \item $\mathcal{T}=\operatorname{Kum}_n$ for some $n\geq 2$ satisfying $n+1=x^2+xy+y^2$ for some $x,y\in\mathbb{Z}$ coprime integers and $m = 9$,
        \item $\mathcal{T} = \operatorname{K3}^{[n]}$ for any $n\geq 2$ and $m\in\{4,6,8,9,12,14,16,18,21,22,25,27,28,33,36,42,44,66\}$,
        \item $\mathcal{T}=\textnormal{OG6}$ and $m\in\{4,6,9,12\}$,
        \item $\mathcal{T}=\textnormal{OG10}$ and $m\in\{4,6,8,9,12,14,16,18,21,22,25,27,28,33,36,42,44,66\}$.
    \end{enumerate}
\end{theo}

In most cases, an example of a pair $(X, G)$ as in \Cref{thm:12} can be obtained via some known geometric constructions: For instance, start with a K3 surface $S$ together with a nontrivial subgroup $G\leq\Aut(S)$ of algebraically trivial automorphisms, and  consider the Hilbert scheme of $n$ points on $S$ for some $n\geq 2$. We refer to \Cref{subsec induced descr}, where we discuss such \emph{induced} actions. 

For nontrivial action on discriminant groups, we observe the following finiteness result.

\begin{theo}[see \Cref{nontriv: all,propo: discard nontriv act}]\label{thm:13}
    Let $\mathcal{T}$ be a known deformation family of IHS manifolds, and let $m\geq 2$ be a composite positive integer. Then, there exists a pair $(X, G)$ where $X$ belongs to $\mathcal{T}$, the finite group $\ker\rho_X\leq G$ is algebraically trivial with $[G:\ker\rho_X]=m$ and such that $\rho_X(G)$ acts nontrivially on the discriminant group of $H^2(X, \mathbb{Z})$  if and only if the pair $(\mathcal{T}, m)$ appears in the following table:
   {\small\setlength{\tabcolsep}{4pt}
\renewcommand\arraystretch{2}
\begin{table}[H]

    \centering
    \begin{tabular}{c|cccccccc}
        $\T$&\cellcolor{lightgray!40!white}$\textnormal{K3}^{[3]}$&$\textnormal{K3}^{[4]}$&\cellcolor{lightgray!40!white}$\textnormal{K3}^{[6]}$& $\textnormal{K3}^{[9]}$&\cellcolor{lightgray!40!white}$\textnormal{K3}^{[14]}$& $\textnormal{K3}^{[24]}$&$\cellcolor{lightgray!40!white}\textnormal{OG6}$&$\textnormal{OG10}$\\
         \hline
         $m$&\cellcolor{lightgray!40!white}$4,\,8,\,16,\,32$&$6,\,18$&\cellcolor{lightgray!40!white}$10,\,50$&$14$&\cellcolor{lightgray!40!white}$26$&$46$&\cellcolor{lightgray!40!white}$4,\,8$&$6,\,18$\\
    \end{tabular}
    
\end{table}}
\end{theo}

This result shows a clear difference with the previous case.
It is reasonable to expect that such a finiteness result highlights the existence of interesting geometric examples of algebraically trivial automorphisms. As we show in the examples in \Cref{subsec examples}, geometric examples of pairs $(X, G)$ as in \Cref{thm:13} sometimes come along with special involutions that are intrinsic to the construction of $X$. Together with the previous theorem, we also prove the following classification result.

\begin{theo}[see \Cref{mainth2}]
    Let $X$ be a projective IHS manifold of known deformation type $\T$, let $m\geq 2$ be a composite integer, and let $\ker\rho_X\lneq G\leq\Aut(X)$ be algebraically trivial with $[G:\ker\rho_X]=m$. Let us set $\Lambda\coloneqq  H^2(X, \Z)$ and suppose that the cyclic group $H \coloneqq  \rho_X(G)$ acts nontrivially on the discriminant group of $\Lambda$. Then, up to deformation and birational conjugacy, the pair $(X, G)$ is uniquely determined by $(\T, m, \Lambda^H, \Lambda_H)$, except in the case $(\T, m) = (\textnormal{K3}^{[24]}, 46)$, where there are three such pairs. The corresponding tuples $(\T, m, \Lambda^H, \Lambda_H)$ are given in \Cref{disc non triv}.
\end{theo}

In \cite{bc22}, Brandhorst and Cattaneo describe an approach to studying nonsymplectic automorphisms of odd prime order in a uniform way for each known deformation type of IHS manifolds, by means of classifying isometries of some specific even unimodular $\Z$-lattices.
The authors suspected that their methods could be applied to classifying some automorphisms of composite order.
In this work, we generalize and extend their approach in order to prove all of the results previously stated.

\begin{rem}
    Note that the classification results presented in this paper could be obtained in a computer-assisted way. We aim to keep the theoretical presentation as it is, to show the complexity of classifying birational automorphisms of IHS manifolds of composite order. The case of algebraically trivial nonsymplectic automorphisms is probably the easiest case to study, and as we demonstrate in this work, it is already quite technical. 
\end{rem}

\subsection*{Plan of the paper}
In \Cref{sec:prelim}, we start by recalling some basic definitions and results about IHS manifolds, integer lattices and hermitian lattices. 
Then we introduce the lattice approach of \cite{bc22}, and we fix some notation for the rest of the paper. 
In particular, in \Cref{sec: lattice approach}, we define a framework for the study of isometries of some even unimodular $\Z$-lattices, and we explain what we need to achieve for our classification of algebraically trivial nonsymplectic automorphisms of IHS manifolds. 
\Cref{sec: existence isometries} deals with the first part of our lattice problem. 
Namely, we give necessary conditions for the existence of some isometries of unimodular $\Z$-lattices with specified local invariants. 
We also recall how the transfer construction allows one to construct isometries of integer lattices with cyclotomic minimal polynomial, using the theory of hermitian lattices.
The work of the previous sections culminates in \Cref{sec: corr th}, where we give necessary and sufficient conditions for the existence of isometries of unimodular lattices with given local invariants.
Moreover, as in \cite{bc22}, we give constructive proofs so that one can determine the type of such isometries.
We also give a way of classifying them and how to count the number of conjugacy classes of such.
Finally, in \Cref{sec: geom}, we apply the lattice theory of this paper to get classifications of conjugacy classes of isometries of the unimodular $\Z$-lattices relevant to (the known types of) IHS manifolds.
We offer in the appendices some further results related to cyclotomic fields that we use to prove many statements in the paper. \Cref{app:C} contains  tables of results for some of the lattice-theoretic results proved in \Cref{sec: geom}.

\section*{Acknowledgments}
The author is very grateful to Simon Brandhorst for introducing him to this project, for sharing his insights regarding some of the proofs and for carefully reviewing early versions of this work. The author would like to thank Chiara Camere for the exchanges on moduli spaces of symmetric IHS manifolds, Annalisa Grossi for telling him about MRS models and discussing \Cref{ex: anna} together, and Markus Kirschmer for the explanation about his algorithms on quadratic defects. Finally, the author would like to thank the anonymous referees for their useful comments.

\section{Preliminaries}\label{sec:prelim}
Let us start by introducing all preliminary notions and results we need for the rest of the paper. Most of the content of this section is recollected from cited literature. We also prove crucial statements that are not known by the author to be available anywhere else.

\subsection{Integer lattices}

\subsubsection*{Definitions and notation}
Throughout the paper, we denote by $(L, b)$ any $\mathbb{Z}$-lattice, where $b$ is the associated nondegenerate $\mathbb{Z}$-bilinear form. 
If there is no confusion possible, we sometimes drop $b$ from the notation, and for all $x\in L$, we set $x^2 \coloneqq  b(x,x)$.
For $a_1,\ldots, a_k\in\mathbb{Q}^\times$, we define $\langle a_1, \ldots, a_k\rangle$ to be the $\mathbb{Z}$-lattice whose Gram matrix in a given basis is diagonal with entries $a_1, \ldots, a_k$. 
In this paper, ADE root lattices are supposed to be negative definite, and we often denote a $\mathbb{Z}$-lattice by the Gram matrix associated to one of its bases. \looseness=-1

\begin{ex}
    We denote by $U \coloneqq  \begin{psmallmatrix}
    0&1\\1&0
\end{psmallmatrix}$ the integer \emph{hyperbolic plane lattice}.
\end{ex}

For a $\mathbb{Z}$-lattice $(L, b)$, we define the \emph{scale} $s(L,b)$ and the \emph{norm} $n(L,b)$ of $(L,b)$ to be the following $\mathbb{Z}$-ideals:
\[
s(L,b) \coloneqq  \sum_{x,y\in L}b(x,y)\mathbb{Z}  \quad \text{and}\quad n(L, b) \coloneqq  \sum_{x\in L}b(x,x)\mathbb{Z} \subseteq s(L,b).
\]
The $\mathbb{Z}$-lattice $(L,b)$ is called \emph{integral} if $s(L, b)\subseteq \mathbb{Z}$, and \emph{even} if furthermore $n(L, b)\subseteq 2\mathbb{Z}$. 
For an integral $\mathbb{Z}$-lattice $(L, b)$, we denote by $(L^\vee, b)$ the \emph{dual lattice} of $(L, b)$ and by $D_L\coloneqq  L^\vee/L$ its discriminant group.
It is equipped with a \emph{torsion bilinear form}
\[
b_L\colon D_L\times D_L\lra \mathbb{Q}/\mathbb{Z},\quad(x+L, y+L)\longmapsto b(x,y)+\mathbb{Z}.
\]
When $(L, b)$ is even, one also defines a \emph{torsion quadratic form} $q_L$ on $D_L$ given by 
\[
q_L\colon D_L\lra \mathbb{Q}/2\mathbb{Z},\quad x+L\longmapsto b(x,x)+2\mathbb{Z}. 
\]
The pair $(D_L, q_L)$ is called the \emph{discriminant form} of $(L, b)$; for simplicity, in this paper, we omit $q_L$ from the notation whenever possible. 

\begin{rem}
    As for $\mathbb{Z}$-lattices, we often denote $(D_L, q_L)$ by a matrix $M$.
    For a given set of generators $(x_1+L, \ldots, x_r+L)$ for $D_L$, the diagonal entries of $M$ refer to representatives of $q_L(x_i+L)\in \mathbb{Q}/2\mathbb{Z}$, and the off-diagonal entries of $M$ give representatives of $b_L(x_i+L, x_j+L)\in\mathbb{Q}/\mathbb{Z}$ for all $i\neq j$.
\end{rem}

For any $a\in \mathbb{Z}_{\neq 0}$, we define $(L, b)(a) \coloneqq (L, a\cdot b)$ to be the associated rescaled lattice. Similarly, if $(L, b)$ is even, for any submodule $H\leq (D_L, q_L)$, we let $H(-1)$ to be the same as $H$ as finite abelian group but equipped with the opposite form $(-q_L)_{\mid H}$.%\smallskip

For an integer $n\geq 1$, we say that $L$ is {\em $n$-elementary} if $nD_L = \{0\}$ or, equivalently, $nL^\vee\subseteq L$. 
In the case where $L = L^\vee$, we refer to $L$ as \emph{unimodular}. 
We denote by $O(L, b)$ the \emph{isometry group} of $(L, b)$. 
If $(L, b)$ is even, any isometry $f\in O(L, b)$ induces a group isomorphism $D_f\colon D_L\to D_L$ that is an isometry with respect to the form $q_L$. We call $f$ \emph{stable} if $D_f = \id_{D_L}$ and \emph{nonstable} otherwise.

For any $f\in O(L, b)$, we define $O(L, b, f)$ to be the centralizer of $f$ in $O(L, b)$, and similarly we let $O(D_L, q_L, D_f)$  be the centralizer of $D_f$ in the group of isometries of $(D_L, q_L)$.

\begin{notat}
  In this paper, direct sums of lattices and torsion quadratic forms are supposed to be \emph{orthogonal direct sums}---we denote them using the symbol $\oplus$.
\end{notat}

Finally, we recall that given a $\mathbb{Z}$-lattice $(L, b)$ and a vector $v\in L$, we define the \emph{divisibility} $\div(v, L)$ of $v$ in~$L$ to be the positive generator of the ideal $b(v, L)$. If $L$ is integral, then $d\coloneqq \div(v, L)$ is the largest positive integer such that $v/d\in L^\vee$.

\subsubsection*{Genera of integer lattices}\label{subsec: genera zlat}
In this paper, we mainly work with even $\Z$-lattices. Most of the notions we introduce in the next paragraphs have their equivalent for general $\mathbb{Z}$-lattices. 
We refer to \cite[Chapter 15]{splg} for more details about genera.

\begin{notat}
    For a prime number $p$ and a $\mathbb{Z}$-lattice $(L, b)$, we denote by $(L_p, b_p) \coloneqq  (L, b)\otimes_\mathbb{Z} \mathbb{Z}_p$ the corresponding $\mathbb{Z}_p$-lattice, where $\mathbb{Z}_p$ denotes the ring of $p$-adic integers. Similarly, we let $(L_\mathbb{R}, b_{\mathbb{R}}) \coloneqq  (L, b)\otimes_{\mathbb{Z}}\mathbb{R}$.
\end{notat}

Note that the definitions of the previous section extend naturally for $\mathbb{Z}_p$-lattices.

\begin{defin}
    Two $\mathbb{Z}$-lattices $S$ and $T$ are said to be in the same \emph{genus} if $S_p\simeq T_p$ for all $p$ prime and $S_\mathbb{R}\simeq T_\mathbb{R}$. 
\end{defin}

Let $L$ be an even $\mathbb{Z}$-lattice.
For any prime number $p$, the isometry class of $L_p$ is determined by a finite list of invariants (and in a unique way for $p$ odd). Moreover, the isometry class of $L_\mathbb{R}$ is determined by its signatures $(l_+, l_-)$.
Hence, one can describe a genus of $\mathbb{Z}$-lattices by a pair of real signatures and a finite set of numerical invariants determining isometry classes of non-unimodular $\Z_p$-lattices, for some prime number~$p$ (see for instance \cite[Chapter VI]{mm09}).
One can effectively decide whether there exists a $\mathbb{Z}$-lattice in such a genus; see \cite[Theorem 1.10.1]{nik79b}.

\begin{rem}
    Any (nonempty) genus of $\mathbb{Z}$-lattices can be represented by a symbol summarizing these invariants; see \cite[Chapter 15]{splg}.
\end{rem}

For a prime number $p$, the isometry class of an even $\mathbb{Z}_p$-lattice also determines the isometry class of a $p$-adic torsion quadratic form $q_p\colon A\to \mathbb{Q}_p/2\mathbb{Z}_p$, where $A$ is a finite abelian $p$-group. Any such $p$-adic torsion quadratic form decomposes as an orthogonal direct sum of elementary constituents of rank at most equal to 2, which have been completely classified (see for instance \cite[Proposition 1.8.1]{nik79b}). Note that such a decomposition is not unique, but there is a canonical choice of an orthogonal decomposition, referred to as the \emph{normal form}, as described in \cite[Chapter IV]{mm09}.%\smallskip

Now, according to \cite[Corollary 1.16.3]{nik79b}, any genus of even $\mathbb{Z}$-lattices is uniquely determined by a pair of real signatures $(l_+, l_-)$ and an isometry class of torsion quadratic forms $A\to\mathbb{Q}/2\mathbb{Z}$.
Indeed, given an even $\mathbb{Z}$-lattice $L$ with discriminant form $A\to \mathbb{Q}/2\mathbb{Z}$, the associated \emph{$p$-primary part} $A_p\to \mathbb{Q}_p/2\mathbb{Z}_p$, where $A_p$ is the $p$-Sylow subgroup of $A$, together with the real signatures and the determinant of $L$ determine the isometry class of $L_p$ for all prime numbers $p$.
In particular, knowing the genus symbol of a given even $\mathbb{Z}$-lattice $L$ allows one to completely describe the torsion quadratic form on $D_L$ from the associated $p$-adic torsion quadratic forms, and hence from the genus symbol.

\begin{propo}\label{propo: existence 24 elem}
    Let $l_+, l_-\geq 0$, let $n\in\mathbb{N}$, let $\delta\in\{1,3,5,7\}$, and let $\epsilon_1,\epsilon_2\in\{\pm 1\}$.
    The genus with symbol $\II_{(l_+, l_-)}2_\II^{\epsilon_1n}4^{\epsilon_2}_\delta$ is nonempty if and only if all of the following hold:
    \begin{enumerate}
        \item $n$ is even and $\epsilon_2 = \left(\frac{\delta}{2}\right)$. 
        \item $l_++l_-\geq n+1$, with equality only if $\epsilon_1 = \epsilon_2$.
        \item $l_+-l_-\equiv \delta+2-2\epsilon_1\mod 8$.
        \item If $n=0$, then $\epsilon_1 = +1$.
    \end{enumerate}
\end{propo}

\begin{proof}
    This is a direct application of \cite[Theorem 15.11]{splg}.
\end{proof}

\subsubsection*{Lattices with isometry}
For $i \geq 1$, we denote by $\Phi_i(X)\in\mathbb{Z}[X]$ the $\supth{i}$ cyclotomic polynomial, and we let~$\varphi$ denote the Euler totient function.

%\smallskip
For a given $\mathbb{Z}$-lattice $L$ and an isometry $f\in O(L)$, we call the pair $(L, f)$ a {\em lattice with isometry}.
Two such lattices with isometry $(L_1, f_1)$ and $(L_2, f_2)$ are called \emph{isomorphic} if there exists an isometry $\psi\colon L_1\to L_2$ such that $f_2\psi = \psi f_1$. 
Let $\mu(X)\in\mathbb{Q}[X]$ be a monic polynomial. 
We call \emph{$\mu$-lattice} any lattice with isometry $(L, f)$ such that $\mu(f) = 0$ as $\mathbb{Z}$-module endomorphisms. 
We moreover say that $(L, f)$ is a \emph{$\mu^\ast$-lattice} if $\mu$ is the minimal polynomial of $f$.%\smallskip

Given a lattice with isometry $(L, f)$, and given a polynomial $\mu(X)\in \mathbb{Q}[X]$, we call the $\mathbb{Z}$-lattice $L^{\mu(f)} \coloneqq  \ker\left(\mu(f)\right)$ the \emph{$\mu$-kernel sublattice} of $(L, f)$.
Now, let $L$ be a $\mathbb{Z}$-lattice, and let $f \in O(L)$ be of finite order $m$. 
We denote by
\[
\chi_f(X) = \prod_{n\mid m}\Phi_{n}^{d_n(f)}(X)\in\mathbb{Z}[X]\quad \text{and}\quad m_f(X) = \prod_{n\mid m,\; d_n(f)\neq 0}\Phi_{n}(X)\in\mathbb{Z}[X]
\]
the characteristic and minimal polynomials of $f$, respectively, where $d_n(f)\in\mathbb{N}_0$ for all $n|m$. 

\begin{rem}
    For all $n\,|\, m $,
the corresponding kernel sublattice $L^{\Phi_n(f)} =\ker\left(\Phi_{n}(f)\right)$ is of $\mathbb{Z}$-rank $d_n(f)\varphi(n)$. Moreover, if $d_n(f)\neq 0$, the $\Z$-lattice $L^{\Phi_n(f)}$ equipped with the restriction $f_n$ of $f$ to $L^{\Phi_n(f)}$ defines a $\Phi_{n}^\ast$-lattice.
\end{rem}

If $n = 1$, we usually denote by $L^f \coloneqq  L^{f-\id}$ the so-called \emph{invariant sublattice of\, $(L, f)$}. 
The orthogonal complement $L_f \coloneqq  (L^f)^\perp_L$ of $L^f$ in $L$ is called the \emph{coinvariant sublattice of\, $(L, f)$}.
In a similar vain, if $m$ is even and $n = 2$, then $L^{f+\id} = L^{-f}$ is referred to as the \emph{$(-1)$-sublattice of\, $(L, f)$}. 
For any $n\geq 3$, the minimal polynomial of $f_n$ is $\Phi_n$, so we can see $L^{\Phi_n(f)}$ as a $\mathbb{Z}[\zeta_n]$-module, where $\zeta_n$ is a primitive $\supth{n}$ root of unity: The multiplication by $\zeta_n$ is given by the action of $f_n$. 
In this way, we can define the structure of a {\em hermitian $\mathbb{Z}[\zeta_n]$-lattice} on $L^{\Phi_n(f)}$ (see \Cref{subsec trace}). 
These kernel sublattices for finite-order isometries allow us to define an invariant of a lattice with isometry of finite order called the \emph{type}.

\begin{defin}[\textit{cf.} {\cite[Definition 4.18]{bh23}}]
    Let $(L, f)$ be a lattice with isometry of finite order $m\geq 1$. 
    For a positive divisor $n$ of $m$, let $H_n$ be the hermitian structure of $(L^{\Phi_n(f)}, f_n)$ (see \Cref{subsec trace}), and let $A_n$ be the $(X^n-1)$-kernel sublattice of $(L, f)$. 
    The finite collection of pairs of genera $\{(g(H_n), g(A_n))\}_{n\mid m}$ is called the \emph{type} of $(L, f)$.
\end{defin}

\begin{rem}
    Two isomorphic lattices with isometry share the same type, but a type might be represented by several distinct isomorphism classes of lattices with isometry. Note, moreover, that the notion of type is not a generalization of the notion of genus of $\Z$-lattices. In particular, having  two lattices with isometry $(L, f)$ and $(M, g)$ of the same type does not imply that for all prime numbers $p$, there exists an isometry $\psi\colon L_p\to M_p$ such that $\psi\circ f_p = g_p\circ\psi$.
\end{rem}

\subsection{Equivariant primitive extensions}\label{eq glue}
We recall some definitions and results about equivariant primitive extensions. 
The reader is expected to know about the general theory of primitive embeddings of even lattices following the work of Nikulin \cite{nik79b}.
%\smallskip

Let $(S, s)$ and $(T, t)$ be two even lattices with isometry, where $s\in O(S)$ and $t\in O(T)$. 
Given an abelian group isomorphism $\gamma\colon H_S \to H_T$ between respective subgroups $H_S\leq D_S$ and $H_T\leq D_T$ of the discriminant groups of $S$ and $T$, we call $\gamma$ a \emph{glue map} if  $q_T(\gamma(x)+T) = -q_S(x+S)$ for all $x+S\in H_S$. 
In such a case, we call $H_S$ and $H_T$ the \emph{glue domains} of $\gamma$. 
Such a glue map $\gamma$ is called \emph{$(s, t)$-equivariant} if $H_S$ and $H_T$ are, respectively, $D_s$-stable and $D_t$-stable, and if it satisfies the {\em equivariant gluing condition} 
\begin{equation}\label{eq:egc}\tag{EGC}
    \gamma\circ (D_{s})_{\mid H_S} = (D_{t})_{\mid H_T}\circ \gamma.
\end{equation}
Note that $\gamma$ is $(s, t)$-equivariant if and only if $s\oplus t$ extends along the primitive extension $S\oplus T\leq L_{\gamma}$ to an isometry $f_\gamma\in O(L_{\gamma})$, where $L_\gamma$ is such that $L_\gamma/(S\oplus T)$ is the graph of $\gamma$ in $D_S\oplus D_T$; see \cite[Proposition~1.4.1]{nik79b}.
We  call $(L_\gamma, f_\gamma)$ an \emph{equivariant primitive extension of\, $(S, s)$ and $(T, t)$}. 

\begin{defin}\label{defin: isom equiv primext}
    Let $(S_1, s_1)\oplus (T_1, t_1)\leq (L_1, f_1)$ and $(S_2, s_2)\oplus (T_2, t_2)\leq (L_2, f_2)$ be two equivariant primitive extensions. They are said to be \emph{isomorphic} if there exists an isomorphism $\psi\colon (L_1, f_1)\to (L_2, f_2)$ that restricts to isomorphisms $\psi_S\colon (S_1,s_2)\to (S_2, s_2)$ and $\psi_T\colon (T_1, t_1)\to (T_2, t_2)$.
\end{defin}

\begin{rem}\label{rem: existence equiv gluing}
    Let $(S_1, s_1)\oplus (T_1, t_1)\leq (L_1, f_1)$ and $(S_2, s_2)\oplus (T_2, t_2)\leq (L_2, f_2)$ be two equivariant primitive extensions as in \Cref{defin: isom equiv primext}. Let moreover $\gamma_1\colon H_{S_1} \rightarrow H_{T_1}$  and $\gamma_2\colon H_{S_2} \rightarrow H_{T_2}$  be the associated equivariant glue maps. 
    By \cite[Corollary 1.5.2]{nik79b}, the two previous equivariant primitive extensions are isomorphic if and only if there exist two isomorphisms $\psi_S\colon (S_1, s_1)\to (S_2, s_2)$ and $\psi_T\colon (T_1, t_1)\to (T_2, t_2)$ such that $\psi_S$ induces an isometry $\overline{\psi_S}\colon H_{S_1} \to H_{S_2}$, $\psi_T$ induces an isometry $\overline{\psi_T}\colon H_{T_1} \to H_{T_2}$ and the following square commutes:
\[\begin{tikzcd}
    H_{S_1}\arrow[r, "\gamma_1"]\arrow[d, "\overline{\psi_S}"']&H_{T_1}\arrow[d, "\overline{\psi_T}"]\\
    H_{S_2}\arrow[r, "\gamma_2"']&H_{T_2}\rlap{.}
\end{tikzcd}\]
\end{rem} 

Following the ideas of \Cref{rem: existence equiv gluing}, it is then also possible  to derive a way of classifying equivariant primitive extensions of two given lattices with isometry.

\begin{propo}[\textit{cf.} {\cite[Proposition 2.2]{bh23}}]\label{nikulin classification equivariant}
    Let $(S, s)$ and $(T, t)$ be two even lattices with isometry.
    Then the double cosets
    \[
    \overline{O(T, t)}\backslash \{\gamma\colon H_S\to H_T\colon \gamma\text{ is an $(s, t)$-equivariant glue map}\}\slash  \overline{O(S, s)}
    \]
    are in bijection with the isomorphism classes of equivariant primitive extensions $(S, s)\oplus (T, t)\to (L_{\gamma}, f)$ such that $f_{\mid S} = s$ and $f_{\mid T} = t$. 
    Here by $\overline{O(T, t)}$ we mean the image of the representation $O(T, t)\to O(D_T, D_t)$ $($and similarly for $\overline{O(S, s)})$.
\end{propo}

\begin{ex}\label{first ex}
A simple example of an equivariant primitive extension is the following. 
Let $L$ be an even $\mathbb{Z}$-lattice, and let $f\in O(L)$. Suppose that there exists a factorisation $m_f(X)  = p_1(x)p_2(X)$ of the minimal polynomial of $f$, with $p_1(X), p_2(X)\in \mathbb{Q}[X]$ coprime. 
The primitive sublattices $L^{p_1(f)}$ and $L^{p_2(f)}$ are orthogonal in $L$. 
If we denote by $f_1$ and $f_2$ the respective restrictions of $f$ to $L^{p_1(f)}$ and $L^{p_2(f)}$, then \((L^{p_1(f)}, f_1) \oplus (L^{p_2(f)}, f_2)\leq(L, f)\)
is an equivariant primitive extension. 
\end{ex}

If we let $\gamma$ be the associated $(f_1, f_2)$-equivariant glue map from \Cref{first ex}, then from $p_1$ and $p_2$ one gets information on the primes dividing the order of the glue domains of $\gamma$.
In fact, let us demonstrate this in what follows.

\begin{defin}
    Let $p(X), q(X)\in \Z[X]$ be two polynomials that are coprime in $\mathbb{Q}[X]$. Then we define the \emph{reduced resultant} of $p$ and $q$, denoted by $\operatorname{rres}(p, q)$, to be the positive generator of the $\Z$-ideal
    \[(p(X)\Z[X]+q(X)\Z[X])\cap \Z.\]
\end{defin}

From its definition, one sees that the reduced resultant of two polynomials $p(X),q(X)\in \Z[X]$  that are coprime in $\mathbb{Q}[X]$ divides their \emph{resultant}
\[
\operatorname{res}(p,q) \coloneqq  \prod_{\substack{(u,v)\in \mathbb{C}^2\\p(u)=q(v)=0}}(u-v)\in \Z.
\]
Moreover, $\operatorname{res}(p,q)$ and $\operatorname{rres}(p,q)$ share the same set of prime divisors.

\begin{rem}
    If $p(X),q(X)\in \Z[X]$ have a common factor of positive degree, then we can define the resultant and reduced resultant similarly, and we obtain that $\operatorname{rres}(p,q) = \operatorname{res}(p,q) = 0$.
\end{rem}

\begin{propo}\label{reduce res}
    Let $(S, s)\oplus (T,t)\leq (L,f)$ be an equivariant primitive extension, and let $d \coloneqq  \operatorname{rres}(m_s(X), m_t(X))$ be the reduced resultant of the minimal polynomials of $s$ and $t$, respectively. Then
    \[ dL\leq S\oplus T.\]
\end{propo}

\begin{proof}
    If $m_s$ and $m_t$ have a common factor of positive degree, then $d=0$ and the result obviously follows. Let us now assume  that $m_s$ and $m_t$ are coprime in $\mathbb{Q}[X]$. In particular, we have that $m_f = m_sm_t$, and we can see $S$ and $T$ as the respective kernel sublattices $\ker m_s(f)$ and $\ker m_t(f)$. Therefore, since $S$ and $T$ are both primitive in $L$, we observe that 
    \[m_s(f)(L)\leq T\quad  \text{and}\quad m_t(f)(L) \leq S.\]
     By the definition of $d$, there exist $u(X), v(X)\in \Z[X]$ such that 
    \[d = u(X)m_s(X) + v(X)m_t(X).\]
    Since $d$ is constant, seen as an element of $\Z[X]$, we have that 
    \[d\cdot\id_L = u(f)m_s(f) + v(f)m_t(f),\]
    meaning that
    \[ dL = (u(f)m_s(f)+v(f)m_t(f))(L)\leq u(f)m_s(f)(L) + v(f)m_t(f)(L)\leq S\oplus T.\]
    Hence the result follows.
\end{proof}

\begin{lem}[\textit{cf.} {\cite[Lemma 2]{fil02}}]\label{resultant cyclo}
    Let $1<n<m$ be positive integers. 
    Then
    \[\operatorname{rres}(\Phi_n, \Phi_m) = \operatorname{rres}(\Phi_m, \Phi_n) = \begin{cases}
    \,p&\text{if }\, \frac{m}{n} \text{ is a power of the prime number $p$,}\\
    \,1 &\text{else.}\end{cases}\]
\end{lem}

We conclude with the following.

\begin{coro}\label{good result pelem}
    Let $p$ be a prime number, and let $n$ be a positive integer. Let $(L, f)$ be a lattice with isometry such that $f$ has order $p$ and $L$ is $n$-elementary. Let us denote by $F \coloneqq  L^f$ and $C\coloneqq  L_f$ the associated invariant and coinvariant sublattices of\, $(L,f)$. Then $F$ and $C$ are $pn$-elementary.
\end{coro}

\begin{proof}
    According to \Cref{reduce res,resultant cyclo}, we have that $pL\leq F\oplus C$. In particular, since $nL^\vee\leq L$, we obtain
    \[pnL^\vee\leq pL\leq F\oplus C.\]
    But now, using the fact that $F\leq L$ is a primitive sublattice, we have that the morphism $\pi\colon L^\vee\to F^\vee$ is surjective. Therefore, by applying $\pi$ to the inclusion $pnL^\vee\leq F\oplus C$, we obtain that 
    \[pnF^\vee\leq F\]
    and $F$ is $pn$-elementary (see \cite[Proposition 4.10]{bh23}). Similar arguments apply to $C$.
\end{proof}

\subsection{Hermitian lattices}\label{prelim hermlat}
We now introduce hermitian lattices, which play a major role in this work. For a more general picture on the subject, we refer to Kirschmer's  habilitation thesis  \cite{kir16}.

\subsubsection*{Definitions}
Let $K$ be a number field, and let $E$ be a degree 2 extension of $K$.
We have that $\Gal(E/K)$ has order 2 and is generated by an involution $\iota$.
For any place $\mathfrak{q}$ of $K$, we define by $E_\mathfrak{q} \coloneqq  E\otimes_KK_\mathfrak{q}$ the $\mathfrak{q}$-adic completion. 
A finite place $\mathfrak{q}$ of $K$ is said to be {\em good} if $\mathfrak{q}$ does not ramify in $E/K$; otherwise, we call it {\em bad}.
Let us denote by $\mathcal{O}_E$ and $\mathcal{O}_K$ respective maximal orders of $E$ and $K$.

A \emph{hermitian $\mathcal{O}_E$-lattice} $(L, h)$ consists of a finitely generated projective $\mathcal{O}_E$-module $L$ together with a nondegenerate binary form 
\[h\colon (L\otimes_{\mathcal{O}_E} E)\times (L\otimes_{\mathcal{O}_E} E)\lra E\]
that is $E$-linear on the first variable and that satisfies $h(x,y) = \iota(h(y, x))$ for all $x,y\in L\otimes_{\mathcal{O}_E} E$ (we say that $h$ is $\iota$-sesquilinear). 
Many of the notions defined for $\mathbb{Z}$-lattices exist for hermitian lattices (see \cite[Section~2]{kir16}), 
with their equivalent in terms of fractional ideals of $\mathcal{O}_E$ and $\mathcal{O}_K$. 
For instance, we have the following:
\begin{enumerate}
    \item\label{notion1} The \emph{scale} $\mathfrak{s}(L, h)$ of $(L, h)$ is the fractional $\mathcal{O}_E$-ideal generated by the elements of the form $h(x,y)$ for $x,y\in L$.
    \item\label{notion2} The \emph{norm} $\mathfrak{n}(L, h)$ of $(L, h)$ is the fractional $\mathcal{O}_K$-ideal generated by the elements of the form $h(x,x)$ for $x\in L$. 
    It satisfies $\mathfrak{n}(L, h)\mathcal{O}_E\subseteq \mathfrak{s}(L, h)$;
    \item\label{notion3} The \emph{dual lattice} of $(L, h)$ is the hermitian $\mathcal{O}_E$-lattice with underlying module
    \[ L^{\#} \coloneqq  \left\{x\in L\otimes_{\mathcal{O}_E} E \colon h(x, L)\subseteq \mathcal{O}_E\right\}\]
    equipped with the form $h$.
    \item\label{notion4} If there exists a fractional $\mathcal{O}_E$-ideal $\mathfrak{A}$ such that $\mathfrak{A}L^{\#} = L$,
    then $L$ is said to be \emph{$\mathfrak{A}$-modular}, and \emph{unimodular} if $\mathfrak{A} = \mathcal{O}_E$.
    \item\label{notion5} If $a_1, \ldots, a_k\in K$, we denote again by $\langle a_1,\ldots, a_k\rangle$ the free hermitian $\mathcal{O}_E$-lattice whose Gram matrix in a given basis is diagonal with entries $a_1, \ldots, a_k$.
\end{enumerate}
If there is no ambiguity, we drop $h$ and $\mathcal{O}_E$ from the notation. 
The previous definitions of hermitian lattices and items (\ref{notion1})--(\ref{notion5}) above apply if we work over the completions $E_\mathfrak{q}/K_{\mathfrak{q}}$ for any finite place $\mathfrak{q}$ in $K$. 

\begin{rem}
     For any finite place $\mathfrak{q}$, the maximal order $\mathcal{O}_{E_\mathfrak{q}}$ is a principal ideal domain: Hence, as in the case of $\mathbb{Z}$-lattices, we can define bases for hermitian $\mathcal{O}_{E_\mathfrak{q}}$-lattices and thus associate Gram matrices to them. In particular, it makes sense to define the \emph{determinant} $\det(L)\in K_\mathfrak{p}^\times/N^{E_\mathfrak{q}}_{K_\mathfrak{q}}(\mathcal{O}_{E_\mathfrak{q}}^\times)$ of a hermitian $\mathcal{O}_{E_\mathfrak{q}}$-lattice $L$.
\end{rem}

\begin{notat}
    For any element $e\in E_\mathfrak{q}^\times$ in a completion of $E$, we denote by $H(e)$ the hermitian $\mathcal{O}_{E_\mathfrak{q}}$-lattice with Gram matrix
$\begin{psmallmatrix}0&e\\\iota(e)&0\end{psmallmatrix}$. We moreover denote by $H(0)$ the hermitian form on $E_\mathfrak{q}$ with Gram matrix $\begin{psmallmatrix}0&1\\1&0\end{psmallmatrix}$. 
\end{notat}

\begin{defin}
    A hermitian $\mathcal{O}_{E_\mathfrak{q}}$-lattice $L$ is called \emph{hyperbolic} if $L\otimes_{\mathcal{O}_{E_\mathfrak{q}}} E_\mathfrak{q}$ is isometric to $H(0)^{\oplus r}$ for some $r\geq 1$.
\end{defin}

According to \cite[Proposition 3.3.5]{kir16}, if a finite place $\mathfrak{q}$ of $K$ is good, then up to isomorphism there is a unique unimodular $\mathcal{O}_{E_\mathfrak{q}}$-lattice for any given rank.%\smallskip

For any hermitian $\mathcal{O}_E$-lattice $(L, h)$ and for any finite place $\mathfrak{q}$ of $K$, the hermitian $\mathcal{O}_{E_{\mathfrak{q}}}$-lattice $(L_{\mathfrak{q}}, h_\mathfrak{q}) \coloneqq  (L, h)\otimes_{\mathcal{O}_K} \mathcal{O}_{K_{\mathfrak{q}}}$ admits an orthogonal Jordan decomposition whose direct summands are modular; see \cite[Theorem 3.3.3]{kir16}. Such a Jordan decomposition might not be unique, and the local isometry class of a hermitian $\mathcal{O}_E$-lattice $L$ at a finite place $\mathfrak{q}$ of $K$ is determined by an equivalence class of Jordan decompositions of $L_\mathfrak{q}$.

\begin{defin}
We say that two hermitian $\mathcal{O}_E$-lattices $(L, h)$ and $(L', h')$ are in the same \emph{genus} if  
\((L_\mathfrak{q}, h_\mathfrak{q}) \simeq (L'_\mathfrak{q}, h'_\mathfrak{q})\) for all places $\mathfrak{q}$ of $K$.    
\end{defin}

Note again that two isometric hermitian $\mathcal{O}_E$-lattices are in the same genus, but the converse does not always hold. 
Let us denote by $\Omega_\infty(K)$ the set of real infinite places of $K$. For a real place $\mathfrak{q}\in \Omega_\infty(K)$, we denote by $n(\mathfrak{q})$ the  number of negative entries on the diagonal of the Gram matrix of $(L, h)\otimes_{\mathcal{O}_K} \mathcal{O}_{K_{\mathfrak{q}}}$ associated to any orthogonal basis. 
The collection $\{n(\mathfrak{q})\}_{\mathfrak{q}\in \Omega_\infty(K)}$ defines the \emph{signatures } of $(L, h)$ and, together with the rank of $L$, they uniquely determine the local isometry class of $(L, h)$ at each real infinite place of $K$.

\subsubsection*{The cyclotomic trace equivalence}\label{subsec trace}
Let $m \geq 3$ be an integer.
We let $E \coloneqq  \mathbb{Q}(\zeta_m)$ be the $\supth{m}$ cyclotomic field, with $\zeta_m\in\mathbb{C}$ a primitive $\supth{m}$ root of unity. 
The field $E$ is equipped with a $\mathbb{Q}$-linear involution $\iota\colon E\to E$ whose fixed field $K$ is generated by $\zeta_m+\zeta_m^{-1}$ over $\mathbb{Q}$. 
Note that the extension $E/K$ is CM. 
We moreover have that $\mathcal{O}_E \coloneqq  \mathbb{Z}[\zeta_m]$ and $\mathcal{O}_K \coloneqq  \mathbb{Z}[\zeta_m+\zeta_m^{-1}]$ are respective maximal orders in $E$ and $K$; see \cite{was97}. %\smallskip

Let $(L, b, f)$ be a $\Phi_m$-lattice, and let us write 
\(S_m \coloneqq  \left\{1\leq i\leq \lfloor m/2\rfloor\colon \gcd(m,i) = 1\right\}.\)
This set has order $\#S_m = \varphi(m)/2 = s$, where $s$ is the number of infinite places of $K$; in particular, there is a bijection between the set $S_m$ and $\Omega_\infty(K)$, the set of (real) infinite places of $K$. 
For $i\in S_m$, we denote by $\mathfrak{q}_i\in \Omega_\infty(K)$ the infinite place of $K$ whose associated $\mathbb{Q}$-embedding in $\mathbb{R}$ sends $\zeta_m+\zeta_m^{-1}$ to $\zeta_m^i+\zeta_m^{-i}$.%\smallskip 

The isometry $f$ has minimal polynomial $\Phi_m$, so there is an action of $\mathcal{O}_E$ on $(L, b)$ given by
\( \zeta_m \cdot x \coloneqq  f(x)\) for all $x\in L$.
Therefore, we can define the form %\vspace*{-1ex}
\begin{equation}\label{hermitian structure eq}
h\colon \left(L\otimes_{\mathcal{O}_E}E\right)\times \left(L\otimes_{\mathcal{O}_E}E\right)\lra E,\quad (x,y)\longmapsto \frac{1}{m}\sum_{0\leq i\leq m-1}b\left(x, f^i(y)\right)\zeta_m^i,%\vspace*{-1ex}
\end{equation}
which is $\iota$-sesquilinear and nondegenerate.
The $\mathcal{O}_E$-module $L$ equipped with $h$ defines a hermitian $\mathcal{O}_E$-lattice, which we call the {\em hermitian structure} of $(L, b, f)$.  

\begin{rem}
    When $f = \pm \id_L$, we let $E \coloneqq  \mathbb{Q}$ and $\iota = \id_E$. Together with \Cref{hermitian structure eq}, we have that $h = b$ and the hermitian structure of $(L, b, f)$ is $(L, b)$ itself.
\end{rem}

Conversely, given a hermitian $\mathcal{O}_E$-lattice $(L, h)$, we define the form
\[ b\colon (L\otimes_\mathbb{Z}\mathbb{Q})\times (L\otimes_\mathbb{Z}\mathbb{Q})\lra \mathbb{Q},\quad (x,y)\longmapsto \Tr^E_\mathbb{Q}(h(x,y)),\]
which is symmetric, bilinear and nondegenerate. 
This turns $(L, b)$ into a $\mathbb{Z}$-lattice.
The multiplication by $\zeta_m$ given by the $\mathcal{O}_E$-module structure on $L$ defines an isometry $f$ of $(L, b)$ of minimal polynomial $\Phi_m$.
Therefore, $(L,b, f)$ is a $\Phi_m$-lattice that we call the \emph{trace lattice} of $(L, h)$.
The two previous constructions are inverse to each other: We talk about \emph{trace equivalence} or \emph{transfer}. We have that 
\begin{equation}\label{ranks trace constr}
    \rank_{\mathbb{Z}}(L) = \varphi(m)\cdot\rank_{\mathcal{O}_E}(L), 
\end{equation}
and two hermitian $\mathcal{O}_E$-lattices are isometric if and only if their respective trace lattices are isomorphic as lattices with isometry. Note moreover that, through the trace equivalence, the centralizer $O(L, b, f)$ coincides with the \emph{unitary group} $U(L, h)$ consisting of $\mathcal{O}_E$-module isomorphisms of $L$ respecting the form $h$.
For a hermitian $\mathcal{O}_E$-lattice $(L, h)$ with trace lattice $(L, b, f)$, we have that
\begin{equation}\label{dual trace constr}
L^{\vee} = \mathfrak{D}_{E/\mathbb{Q}}^{-1}L^{\#},
\end{equation}
and $(L, b)$ is integral if and only if $\mathfrak{s}(L, h)\subseteq \mathfrak{D}_{E/\mathbb{Q}}^{-1}$. Here $\mathfrak{D}_{E/\mathbb{Q}}$ denotes the \emph{different ideal} of the finite extension $E/\mathbb{Q}$; we recall that it is defined as the inverse of the fractional ideal
\[
\left(\mathfrak{D}_{E/\mathbb{Q}}^{-1}=\right)\ \left\{a\in E \;\colon \forall y\in\mathcal{O}_E,\, \Tr_{E/\mathbb{Q}}(xy)\in \mathbb{Z}\right\}.
\]
Moreover, if $(L, b)$ is integral, then $(L, b)$ is even if and only if $\mathfrak{n}(L, h)\subseteq \mathfrak{D}_{K/\mathbb{Q}}^{-1}$ (see \cite[Lemma 6.6]{bh23}). 
More generally, we prove the following lemma.

\begin{lem}\label{trace of norm}
    Let $(L, b, f)$ be the trace lattice of a hermitian $\mathcal{O}_E$-lattice $(L,h)$. 
    Then, 
    \[s(L, b) = \Tr^E_\mathbb{Q}(\mathfrak{s}(L, h))\quad \text{and}\quad n(L, b) = 2\Tr^K_\mathbb{Q}(\mathfrak{n}(L, h)).\]
\end{lem}

\begin{proof}
    The result for the scale is trivial since by definition $b = \Tr^E_\mathbb{Q}\circ h$.
    Now, we remark that for all $x\in L$, we have
    $b(x,x) = \Tr^E_\mathbb{Q}(h(x,x)) = 2\Tr^K_\mathbb{Q}(h(x,x))\in 2\Tr^K_\mathbb{Q}(\mathfrak{n}(L, h)).$
    For the other inclusion, we follow a similar strategy to that in \cite[Corollary 6.7]{bh23}. 
    Note that since 
    \begin{equation}\label{padictracerep}
        \left(L_p, b_p\right) \simeq \bigoplus_{\mathfrak{p} \mid  p\mathcal{O}_K} \left(L_\mathfrak{p}, \Tr^{E_\mathfrak{p}}_{\mathbb{Q}_p}\circ h_\mathfrak{p}\right)
    \end{equation}
    as $\mathbb{Z}_p$-lattices, it suffices to show that for all prime numbers $p$ and every prime ideal $\mathfrak{p}$ of $\mathcal{O}_K$ dividing $p\mathcal{O}_K$, we have the inclusion
    \[2\Tr^{K_\mathfrak{p}}_{\mathbb{Q}_p}\left(\mathfrak{n}\left(L_\mathfrak{p}, h_\mathfrak{p}\right)\right)\subseteq n\left(L_p, b_p\right).\]
    The latter would hold if one can prove that for all $a\in \mathcal{O}_{K_\mathfrak{p}}$ and for all $x\in L_\mathfrak{p}$,
    \[2\Tr^{K_\mathfrak{p}}_{\mathbb{Q}_p}\left(ah_\mathfrak{p}(x,x)\right) = \Tr^{E_\mathfrak{p}}_{\mathbb{Q}_p}\left(ah_\mathfrak{p}(x,x)\right)\in n\left(L_p, b_p\right).\]
    Let $B$ be the $\mathbb{Z}$-module consisting of elements $w\in K_\mathfrak{p}$ such that $\Tr^{E_\mathfrak{p}}_{\mathbb{Q}_p}(wh_\mathfrak{p}(x,x))\in n(L_p, b_p)$ for all $x\in L_\mathfrak{p}$. We want to show that $B$ contains $\mathcal{O}_{K_\mathfrak{p}}$. In order to do so, we apply \cite[Lemma 6.5]{bh23}, which tells us that if~$B$ contains 1, $\Tr^{E_\mathfrak{p}}_{K_{\mathfrak{p}}}(\mathcal{O}_{E_\mathfrak{p}})$ and $N^{E_\mathfrak{p}}_{K_\mathfrak{p}}(\mathcal{O}_{E_\mathfrak{p}})B$, then $B$ contains $\mathcal{O}_{K_\mathfrak{p}}$.  
    By the decomposition in \Cref{padictracerep}, one can already show that $B$ contains $1$, and $B$ is nonempty. Now let $w\in B$ and let $\lambda\in\mathcal{O}_{E_\mathfrak{p}}$. One computes, for all $x\in L_\mathfrak{p}$, 
    \[\Tr^{E_\mathfrak{p}}_{\mathbb{Q}_p}\left(\lambda\iota(\lambda)wh_\mathfrak{p}(x,x)\right) = \Tr^{E_\mathfrak{p}}_{\mathbb{Q}_p}\left(wh_\mathfrak{p}(\lambda x,\lambda x)\right) \in n\left(L_p, b_p\right).\]
    Hence, $N^{E_\mathfrak{p}}_{K_\mathfrak{p}}(\mathcal{O}_{E_\mathfrak{p}})B\subseteq B$. Finally, for all $\lambda\in \mathcal{O}_{E_\mathfrak{p}}$ and for all $x\in L_\mathfrak{p}$, we observe that 
    \[\Tr^{E_\mathfrak{p}}_{\mathbb{Q}_p}\left((\lambda+\iota(\lambda))h_\mathfrak{p}(x,x)\right) = \Tr^{E_\mathfrak{p}}_{\mathbb{Q}_p}\left(h_\mathfrak{p}((\lambda+1) x,(\lambda+1) x)\right) - \Tr^{E_\mathfrak{p}}_{\mathbb{Q}_p}\left(h_\mathfrak{p}(\lambda x,\lambda x)\right) - \Tr^{E_\mathfrak{p}}_{\mathbb{Q}_p}\left(h_\mathfrak{p}(x,x)\right)\in n\left(L_p, b_p\right).
    \]
    This implies that $\Tr^{E_\mathfrak{p}}_{K_{\mathfrak{p}}}(\mathcal{O}_{E_\mathfrak{p}})\subseteq B$, and we can conclude.
\end{proof}

For all $i\in S_m$, we denote by $(k_i^+, k_i^-)\in \mathbb{N}^2$ the signatures of the real quadratic space
\((K_i, b_{\mathbb{R}}) \coloneqq  \ker(f_{\mathbb{R}}+f_{\mathbb{R}}^{-1}-\zeta_m^i-\zeta_m^{-i}).\) 
If $(l_+, l_-)$ denotes the signatures of $(L, b)$, one has that $l_{\pm} = \sum_{i\in S_m} k_i^{\pm}.$
Moreover, for all $i\in S_m$, we have that $(K_i, b_{\mathbb{R}}) = \Tr^{E_{\mathfrak{q}_i}}_{K_{\mathfrak{q}_i}}(L_{\mathfrak{q}_i})$ and $k_i^- = 2n(\mathfrak{q}_i)$. 
Hence for all $i\in S_m$, both signatures of $K_i$ are even and
\begin{equation}\label{sign trace constr}
    l_- = 2\sum_{i\in S_m}n(\mathfrak{q}_i).
\end{equation}

\subsection{Irreducible holomorphic symplectic manifolds}
An \emph{irreducible holomorphic symplectic} (IHS) manifold is a simply connected, compact, K\"ahler manifold $X$ such that $H^0(X, \Omega^2_X)$ is $1$-dimensional generated by a nowhere-degenerate holomorphic 2-form $\sigma_X$. %\smallskip

IHS manifolds have even complex dimension, and for a given dimension, the known examples of IHS manifolds are classified up to deformation. 
The 2-dimensional IHS manifolds are K3 surfaces, and for each even dimension $2n > 2$, we know at least two deformation types: the $\textnormal{K3}^{[n]}$-type (see \cite[Section~6]{bea83a}) and the $\textnormal{Kum}_n$-type (see \cite[Section~7]{bea83a}). 
Moreover, thanks to the work of O'Grady, we know two further deformation types in dimension 6 and 10, respectively, denoted by $\textnormal{OG6}$ (see \cite{og03}) and $\textnormal{OG10}$ (see \cite{og99}). 
At the time this paper is written, no other deformation types have been discovered apart from the one previously cited; we refer to them as {\em known deformation types} in this work. 

\begin{notat}
For an IHS manifold $X$ and a known deformation type $\mathcal{T}$, we write $X \sim \mathcal{T}$ if $X$ is of deformation type $\mathcal{T}$. 
An example for each known deformation type can be obtained by constructing some moduli spaces of semistable sheaves on some projective K3 or abelian surfaces.
\end{notat}

For a birational automorphism $f$ of an IHS manifold $X$, the action of $f$ on cohomology respects the Hodge structure on $H^2(X, \mathbb{C})$. 
Those automorphisms that restrict to the identity on $H^{2, 0}(X)$ are called {\em symplectic}; those that do not are called {\em nonsymplectic}.
A main feature for any IHS manifold $X$ is that its integral second cohomology group comes equipped with a nondegenerate quadratic form $q_X$ of real signatures $(3, b_2(X)-3)$, called the Beauville--Bogomolov--Fujiki (BBF) quadratic form. 
Moreover, any $f\in\Bir(X)$ induces an isometry of $(H^2(X, \mathbb{Z}), q_X)$  that is in particular a \emph{monodromy operator}; see \cite[Theorem 1.3]{mar11}. 
These operators generate the \emph{monodromy group} $\Mon^2(X)$, and the ones whose extension to $\mathbb{C}$ respects the Hodge structure on $H^2(X, \mathbb{C})$ are the so-called {\em Hodge monodromies}. 
It follows that the image of the orthogonal representation 
\[\rho_X\colon \Bir(X) \lra O\left(H^2(X,\mathbb{Z}), q_X\right)\]
consists of Hodge monodromies. 
The homomorphism $\rho_X$ has finite kernel; see \cite[Proposition 9.1]{huy99}.
Moreover, by a consequence of {\em Markman--Verbitsky's strong Torelli--type theorems}
(see \cite[Theorem 1.3]{mar11} and \cite[Theorem 1.17]{ver13}), we know that the image of $\rho_X$ consists of Hodge monodromies preserving the \emph{fundamental exceptional chamber} $\mathcal{FE}_X$ of $X$. For the reader's convenience, we recall the definition of $\mathcal{FE}_X$. The set of vectors $\{v\in H^{1,1}(X, \mathbb{R})\;\colon q_X(v,v) > 0\}$ has two connected components, and exactly one of them contains a K\"ahler class of $X$; we call it the \emph{positive cone} $C_X$ of $X$. Such a cone $C_X$ admits two \emph{wall-and-chamber decompositions}. The walls of these decompositions are described by two sets of vectors $\mathcal{W}^{\pex}(X)\subseteq \mathcal{W}(X)$, respectively called \emph{stably prime exceptional divisors} and {\em MBM-classes} of $X$; see \cite[Definition~1.13]{av15}. The fundamental exceptional chamber $\mathcal{FE}_X$ of $X$ is the chamber containing a K\"ahler class for the decomposition of $C_X$ with respect to $\mathcal{W}^{\pex}(X)$. Similarly, the \emph{K\"ahler cone} $\mathcal{K}_X$ of $X$ is given by the chamber containing a K\"ahler class, for the decomposition of $C_X$ with respect to $\mathcal{W}(X)$. Note that $\mathcal{K}_X\subseteq \mathcal{FE}_X$.
%\smallskip

For an IHS manifold $X$ of one of the known deformation families (at the time this paper is written), the monodromy group $\Mon^2(X)$, the kernel of $\rho_X$ and the BBF form $q_X$ have been explicitly determined (see \cite[Lemma 4.2]{mar10}, \cite[Theorem 2.3]{mon16}, \cite[Theorem 1.4]{mr21}, \cite[Theorem 5.4]{ono22} for the monodromy groups; \cite[Propositions 9 and 10]{bea83b}, \cite[Corollary 5]{bnws11}, \cite[Theorems 2.1 and 4.2]{mw17} for the kernel of $\rho_X$; and \cite[Propositions 6 and 8]{bea83a}, \cite[Corollary 3.5.13]{rap07}, \cite[Theorem 4.3]{rap08} for the computation of the BBF forms).
This information allows one to study symmetries for a given deformation type by working with the isometries of the associated BBF form; such isometries are called {\em effective} if they can be realized as symmetries of an IHS manifold and \emph{nonsymplectic} if they can be realized by a nonsymplectic symmetry.%\smallskip 

The common strategy to classify such effective isometries is to first bring  all the previous definitions to a single even $\Z$-lattice by means of markings. 
A {\em marked pair} $(X, \eta)$ consists of an IHS manifold $X$ of a given deformation type $\mathcal{T}$, and a {\em marking} $\eta\colon H^2(X, \mathbb{Z})\to \Lambda_\mathcal{T}$ that is an isometry to a reference $\Z$-lattice attached to the deformation type.
These pairs are classified up to equivalence in a coarse moduli space $\mathcal{M}_\mathcal{T}$.
Two pairs lie in the same connected component of $\mathcal{M}_\mathcal{T}$ if their markings differ by a parallel transport operator, and they define inseparable points if and only if they represent birational IHS manifolds; see \cite[Theorem 4.6]{huy99}.
The associated period map
\[\mathcal{P}_\T\colon\mathcal{M}_\mathcal{T} \lra \Omega_\mathcal{T} \coloneqq  \left\{\mathbb{C}\omega\in\mathbb{P}(\Lambda_\mathcal{T}\otimes_\mathbb{Z} \mathbb{C})\colon(\omega, \omega) = 0,\; (\omega, \bar{\omega}) > 0\right\}, \quad (X, \eta)\longmapsto \mathbb{C}\eta\left(\sigma_X\right) \]
is a local homeomorphism. 
It is surjective onto the period domain $\Omega_\T$ when restricted to any connected component of $\mathcal{M}_\T$; see \cite[Theorem 8.1]{huy99}. 
The overall strategy to classify automorphisms of IHS manifolds consists in using the monodromy group, the MBM-classes and the BBF form of a given deformation type to classify the groups of effective Hodge monodromies. 
The surjectivity of the period map and the strong Torelli theorem then allow us  to draw a conclusion about the existence of an IHS manifold of that deformation type with the given classified symmetries, up to the kernel of the associated orthogonal representation.

\begin{rem}
    Up to fixing a connected component of $\mathcal{M}_\mathcal{T}$, one can \emph{numerically} characterize the set of MBM-classes on any $X\sim\mathcal{T}$. By this, we mean the following. Given a deformation type $\T$ and a connected component $\mathcal{M}^\circ_{\mathcal{T}}$ of the moduli space of marked pairs, there exist two sets of vectors
    \[\mathcal{W}_\circ^{\pex}(\Lambda_\T)\subseteq \mathcal{W}_\circ(\Lambda_\T)\subseteq \Lambda_\T\]
    that satisfies the following: For any marked pair $[(X, \eta)]\in \mathcal{M}^\circ_{\mathcal{T}}$, one has
    \[\mathcal{W}^{\pex}(X) = \eta^{-1}\left(\mathcal{W}_\circ^{\pex}(\Lambda_\T)\right)\cap H^{1,1}(X, \mathbb{R})\quad \text{and}\quad \mathcal{W}(X) = \eta^{-1}\left(\mathcal{W}_\circ(\Lambda_\T)\right)\cap H^{1,1}(X, \mathbb{R}).\]
\end{rem}

\subsection{A moduli classification}\label{mod class part}
In \cite{bc22}, the authors classify pairs $(X, G)$ consisting of a (projective) IHS manifold $X$ of known deformation type and a finite subgroup $G\leq \Aut(X)$ of automorphisms whose action on $H^2(X, \mathbb{Z})$ is cyclic of odd prime order, generated by the image of a nonsymplectic automorphism. For the geometric part of \cite{bc22}, most of the proofs rely on the fact that for isometries of odd prime order $p$, the associated minimal polynomial always has at most two factors, and primitive $\supth{p}$ roots of unity are not real. In the general situation, this still holds as long as $m\geq 3$ and the minimal polynomial is of the form $\Phi_1\Phi_m$. In what follows, we recall the results of Brandhorst--Cattaneo.%\smallskip

For the rest of this section, let us fix $\T$, a deformation type of IHS manifolds. 

\begin{defin}\label{def eq and bir conj}
    Let $X$ and $X'$ be two projective IHS manifolds of the same deformation type, and let $G\leq\Aut(X)$ and $G'\leq\Aut(X')$ be finite.  
    We call the pairs $(X, G)$ and $(X', G')$ \emph{birational conjugate} if there exists a birational isomorphism $f\colon X\dashrightarrow X'$ such that $fGf^{-1} = G'$.
\end{defin}

Let $X$ be an IHS manifold, and let $\mathcal{X}\to \Def(X)$ be the universal deformation family of $X$. In \cite[Section~3.1]{bc22}, the authors show that for any finite subgroup $G\leq \Aut(X)$, there exists a universal deformation family of the pair $(X, G)$ whose base is $\Def(X)^G$. This motivates the following definition.

\begin{defin}\label{def eq ohne bir conj}
    Let $X$ and $X'$ be two projective IHS manifolds of the same deformation type, and let $G\leq\Aut(X)$ and $G'\leq\Aut(X')$ be finite. We call the pairs $(X, G)$ and $(X', G')$ \emph{deformation equivalent} if there exist a connected family $\pi\colon\mathcal{X}\to B$ of IHS manifolds, a group $\mathcal{G}\leq \Aut(\mathcal{X}/B)$ and two points $b,b'\in B$ such that $(\mathcal{X}, \mathcal{G})_b = (X, G)$ and $(\mathcal{X}, \mathcal{G})_{b'} = (X', G')$. 
\end{defin}
Here, we say  $\pi\colon \mathcal{X}\to B$ is a connected family of IHS manifolds if the base $B$ is connected, $\pi$ is a smooth proper holomorphic map, and the fiber over every closed point in $B$ is an IHS manifold.%\smallskip

Let $\mathcal{M}^\circ$ be a connected component of the moduli space $\mathcal{M}$ of marked IHS manifolds of deformation type~$\T$. Let $\Lambda$ be the abstract $\Z$-lattice associated to the deformation type $\T$, and let $\Mon_\circ^2(\Lambda) \coloneqq  \eta\Mon^2(X)\eta^{-1}$ for some $[(X, \eta)]\in\mathcal{M}^\circ$. We denote again by $\mathcal{W}_\circ^{\pex}(\Lambda)\subseteq \mathcal{W}_\circ(\Lambda)$ the sets of numerical stably prime exceptional divisors and MBM-classes of $\Lambda$ associated to $\mathcal{M}^\circ$, and we define $C^\circ_\Lambda\subseteq \Lambda\otimes_\Z\mathbb{R}$ to be such that for all $[(X, \eta)]\in \mathcal{M}^\circ$, we have an equality $C_X = \eta^{-1}(C^\circ_\Lambda)\cap H^{1,1}(X,\mathbb{R})$. 

\begin{rem}
    Note that the definitions of $\Mon_\circ^2(\Lambda),\, \mathcal{W}_\circ^{\pex}(\Lambda),\, \mathcal{W}_\circ(\Lambda)$ and $C^\circ_\Lambda$ do not depend on a choice of $[(X, \eta)]\in \mathcal{M}^\circ$. Hence, these objects are well defined as long as we fix a connected component of $\mathcal{M}$.
\end{rem}

Now let  $H\leq \Mon^2_\circ(\Lambda)$ be cyclic of finite order $m\geq 3$ generated by $h$ with minimal polynomial $\Phi_1\Phi_m$. Suppose that $h$ is nonsymplectic; in particular, the coinvariant sublattice $\Lambda_h = \Lambda^{\Phi_m(h)}$ has signatures $(2, \ast)$; see \cite[Proposition 6]{bea83b}. We fix a primitive $\supth{m}$ root of unity $\zeta_m$ such that 
\[\ker\left(h+h^{-1}-\zeta_m-\zeta_m^{-1}\right)\]
has real signatures $(2, \ast)$ (see \Cref{subsec trace}), and we define a character $\chi\colon H\to \mathbb{C}^\times, \; \chi(h) = \zeta_m$. The choice of~$\chi$ gives rise to a moduli space $\mathcal{M}^\chi_H$ parametrizing $H$-marked IHS manifolds $(X, \eta, G)$ of deformation type~$\T$ such that for all $g\in G$ and $\sigma_X\in H^{2,0}(X)$, the following holds:
\[\chi\left(\eta\rho_X(g)\eta^{-1}\right)\sigma_X = (g^\ast)^{-1}\sigma_X.\]
We say that a triple $(X, \eta, G)$ is \emph{$H$-marked} if  $X\sim \T$, $\ker\rho_X\leq G\leq \Aut(X)$ and $\eta\rho_X(G)\eta^{-1}=H$. According to \cite[Proposition 3.8]{bc22}, the forgetful map
\[\phi\colon \mathcal{M}^\chi_H\lra \mathcal{M},\quad (X, \eta, G)\longmapsto (X, \eta)\]
is a closed embedding. In particular, the period map
\[\mathcal{P}^\chi\colon \mathcal{M}^\chi_H\lra \Omega^\chi \coloneqq  \left\{[\omega]\in \mathbb{P}(\Lambda\otimes_\mathbb{Z}\mathbb{C})\;\colon \omega^2 = 0,\; \omega\cdot\overline{\omega} > 0\;\text{and}\; h(\omega) = \chi(h)\omega,\, \forall h\in H\right\}\]
is a local isomorphism. In what follows, we set $\mathcal{M}^{\chi, \circ}_H \coloneqq  \phi^{-1}(\mathcal{M}^\circ\cap \phi(\mathcal{M}^\chi_H))$.%\smallskip 

We denote by $N \coloneqq  \Lambda^H$ the corresponding invariant sublattice and by $M = N^\perp$ its orthogonal complement in $\Lambda$. For later purposes, we let $C^\circ_N \coloneqq  C^\circ_\Lambda\cap (N\otimes_\Z\mathbb{R})$. We call \emph{K\"ahler-type chamber} in $C^\circ_N$ any connected component of $C^\circ_N\setminus \bigcup_{v\in \mathcal{W}_\circ(\Lambda)}v^\perp$.

\begin{defin}[\textit{cf.} {\cite[Definition 3.10]{bc22}}]
    Let $[(X, \eta, G)]\in \mathcal{M}^{\chi, \circ}_H$ be an $H$-marked triple.
    \begin{enumerate}
        \item We say $(X, \eta, G)$ is \emph{$(H, N)$-polarized} if $C^\circ_N\cap \eta(\mathcal{K}_X)\neq \emptyset$.
    \end{enumerate}
    Now suppose $(X, \eta, G)$ is $(H, N)$-polarized.
    \begin{enumerate}\setcounter{enumi}{1}
        \item For a K\"ahler-type chamber $K(N)$ in $C^\circ_N$ preserved by $H$, we say $(X, \eta, G)$ is \emph{$K(N)$-general} if we have $\eta(\mathcal{K}_X)\cap (N\otimes_\Z\mathbb{R}) = K(N)$.
        \item We say $(X, \eta, G)$ is \emph{$H$-general} if it is $K(N)$-general for some K\"ahler-type chamber $K(N)$ in $C^\circ_N$ preserved by $H$.
    \end{enumerate}
\end{defin}

For an $H$-marked triple $(X, \eta, G)$, we write $C_X^G \coloneqq  C_X\cap H^2(X, \mathbb{R})^{\rho_X(G)}$ and $\mathcal{FE}_X^G \coloneqq  \mathcal{FE}_X\cap H^2(X,\mathbb{R})^{\rho_X(G)}$. Recall that the fundamental exceptional chamber $\mathcal{FE}_X$ of $X$ is a chamber of 
\[C_X\setminus\bigcup_{v\in\mathcal{W}^{\pex}(X)}v^\perp,\]
where $\mathcal{W}^{\pex}(X)\subseteq \NS(X)$ is the set of stably prime exceptional divisors of $X$. The following is a straightforward generalization of \cite[Lemma 3.12]{bc22}.

\begin{lem}\label{lem fund chamb}
    If an $(H, N)$-polarized triple $(X, G, \eta)$ is $H$-general, then $\mathcal{FE}_X^G$ is a chamber of 
    \[C_X^G\setminus\bigcup_{v\in\mathcal{W}^{\pex}(X)}v^\perp.\]
\end{lem}
\begin{proof}
   By the generality assumption,  we observe that
    \[C_X^G\setminus\bigcup_{v\in\mathcal{W}^{\pex}(X)}v^\perp =C_X^G\setminus\bigcup_{v\in\mathcal{W}^{\pex}(X)\cap H^2(X, \Z)^{\rho_X(G)}}v^\perp.\]
    Hence, the stably prime exceptional divisors of $X$ are fixed under the action of $G$, so we can conclude (see also \cite[Lemma 5.2]{bcs19}).
\end{proof}

Let us now fix a K\"ahler-type chamber $K(N)\subseteq C^\circ_N$. We define
\begin{align*}
    \Delta_\circ &\coloneqq  \bigcup_{v\in \mathcal{W}_\circ(M)}\mathbb{P}(v^\perp)\subseteq \mathbb{P}(\Lambda\otimes_\Z\mathbb{C}),\\
    \Delta_\circ'&\coloneqq  \bigcup_{v\in \mathcal{W}_\circ'}\mathbb{P}(v^\perp)\subseteq \mathbb{P}(\Lambda\otimes_\Z\mathbb{C}),
\end{align*}
where $\mathcal{W}_\circ(M) \coloneqq  \mathcal{W}_\circ(\Lambda)\cap M$ and 
\[\mathcal{W}_\circ' \coloneqq  \{v\in \mathcal{W}_\circ(\Lambda)\;\colon \exists(v_N, v_M)\in N^\vee\times M^\vee,\; v = v_N+v_M\;\text{and}\; v_N^2<0,\; v_M^2<0\}.\]
In \cite[Proposition 3.11(1)]{bc22},
the authors show that the restriction of the period map $\mathcal{P}^\chi$ to $\mathcal{M}^{\chi, \circ}_H$ surjects onto $\Omega^\chi\setminus \Delta_\circ$. Moreover, if one denotes by $\mathcal{M}^{\chi, \circ}_{K(N)}$ the subset of $\mathcal{M}^{\chi, \circ}_H$ consisting of $K(N)$-general $(H, N)$-polarized triples, then $\mathcal{P}^\chi$ induces a bijection
\[\mathcal{P}^\chi_{K(N)}\colon \mathcal{M}^{\chi, \circ}_{K(N)}\lra \Omega^\chi\setminus (\Delta_\circ\cup \Delta_\circ')\]
(see \cite[Proposition 3.11(2)]{bc22} and \cite[Theorem 5.6]{bcs19}). Moreover, the following holds.

\begin{propo}
    Let $w\in\Omega^\chi\setminus(\Delta_\circ\cup \Delta_\circ')$, and let $[(X_1, \eta_1, G_1)],\,[(X_2, \eta_2, G_2)]\in \mathcal{M}^{\chi, \circ}_H$ be two $K(N)$-general elements in the fiber of\, $\mathcal{P}^\chi\!$ over the period point $w$. Then $(X_1, G_1)$ and $(X_2, G_2)$ are birational conjugate.
\end{propo}

\begin{proof}
The proof is similar to that of \cite[Proposition 3.11, 3.]{bc22}.\qedhere
\end{proof}

We can now state the following theorem, which is a generalization of 
\cite[Theorem 3.13]{bc22}. It relates a classification of pairs $(X, G)$ as before, up to deformation and birational conjugacy, with a classification of conjugacy classes of finite cyclic subgroups $H\leq \Mon_\circ^2(\Lambda)$.

\begin{theo}\label{main thhhh}
    Fix a deformation type $\mathcal{T}$ of IHS manifolds and a connected component $\mathcal{M}^\circ$ of the moduli space of marked IHS manifolds of deformation type $\T$. Let $h_1,\ldots,h_n$ be a complete set of representatives for the conjugacy classes of elements in $\Mon_\circ^2(\Lambda)$ of order $m\geq 3$ whose minimal polynomial equals $\Phi_1\Phi_m$. Suppose that the representatives $h_i$ are chosen such that the real quadratic space $\ker(h_i+h_i^{-1}-\zeta_m-\zeta_m^{-1})\leq \Lambda\otimes_\mathbb{Z}\mathbb{R}$ has signatures $(2, \ast)$ for all $i$. Let $H_i = \langle h_i \rangle \leq \Mon_\circ^2(\Lambda)$, and let $\chi_i\colon H_i\to \mathbb{C}^{\times}$ be the character defined by $\chi_i(h_i) \coloneqq  \zeta_m$. For each $H_i$, choose a point $(X_i, G_i, \eta_i)\in \mathcal{M}^{\chi_i, \circ}_{H_i}$. Then, the pairs $(X_1, G_1), \ldots, (X_n, G_n)$ form a complete set of representatives of pairs $(X, G)$ up to deformation and birational conjugacy, where $X\sim\mathcal{T}$, the group $G \leq \Aut(X)$ is nonsymplectic with symplectic subgroup $G_s = \ker (\rho_X)$, and $\rho_X(G)$ is cyclic of order $m$ generated by an isometry with minimal polynomial $\Phi_1\Phi_m$.
\end{theo}

\begin{proof}
    The monodromy conjugacy class of $H_i = \langle h_i \rangle$ is invariant under deformation and birational conjugacy; therefore, the minimal polynomial of a generator stays unchanged too. Finally, since $\chi_i$ is not real, the period domains are connected, and the proof follows similarly to that of \cite[Theorem 3.13]{bc22}. We omit the details.
\end{proof}

\begin{rem}\label{rem perm}
    For $H_i = \langle h_i\rangle$ as in Theorem~\ref{main thhhh}, all the generators $h_i'$ of $H_i$ have order $m$ and minimal polynomial $\Phi_1\Phi_m$. However, the signatures of the real quadratic space $\ker(h_i+h_i^{-1}-\zeta_m-\zeta_m^{-1})$ are $(2, \ast)$, while those of $\ker(h_i'+h_i'^{-1}-\zeta_m-\zeta_m^{-1})$ are $(0, \ast)$ for any other generator $h_i'\neq h_i,\ h_i^{-1}$.
    Since we aim to classify cyclic groups of nonsymplectic isometries, this extra condition avoids duplication as we could have two representatives of different conjugacy classes generating cyclic subgroup of $O(\Lambda)$ that are conjugate. This actually boils down to considering different genera of hermitian $\mathbb{Z}[\zeta_m]$-lattices, up to an action of the Galois group $\Gal(\mathbb{Q}(\zeta_m+\zeta_m^{-1})/\mathbb{Q})$ on the set of places of $\mathbb{Q}(\zeta_m+\zeta_m^{-1})$.
\end{rem}

By \cite[Lemma 3.7 and Theorem 3.9]{cam16}, the component $\mathcal{M}^{\chi, \circ}_{H}$ contains a Hausdorff subspace consisting of those polarized IHS manifolds such that $\eta(\NS(X)) = N$, \textit{i.e.}, for which the corresponding group $G$ is algebraically trivial. According to \cite[Theorem 3.9]{cam16}, the image of this subspace under the period map is dense and connected in the associated period domain $\Omega^\chi$.
In particular, for any effective finite subgroup $H\leq \Mon_\circ^2(\Lambda)$ that is cyclic of order $m\geq 3$, generated by $h$, with minimal polynomial $\Phi_1\Phi_m$, and for any $H$-marked triple $(X, \eta, G)$, the deformation family of the pair $(X, G)$ contains a pair $(X', G')$ with $G'$ algebraically trivial.
Therefore, the representatives $(X_i, G_i)$ from \Cref{main thhhh} can be chosen such that the groups $G_i$ are algebraically trivial. The aim for the rest of the paper is to provide the tools to determine the isometries $h_i$ from \Cref{main thhhh}, and give a proof of
\Cref{mainth2}.

\section{Lattice-theoretic approach}\label{sec: lattice approach}
In order to make \Cref{main thhhh} explicit on the $\Z$-lattice side, we need to study the existence of isometries of the known BBF forms with minimal polynomial $\Phi_1\Phi_m$, and classify them up to conjugacy in the respective monodromy groups. For each known deformation type $\T$ of IHS manifolds, with reference $\Z$-lattice $\Lambda_\T$, there exists an even unimodular $\Z$-lattice $M_\T$ in which $\Lambda_\T$ embeds primitively with positive-definite complement. In \cite{bc22}, Brandhorst and Cattaneo use this fact to transport the $\Z$-lattice classification of isometries for each $\T$ to a study of isometries with given minimal polynomial in certain unimodular $\Z$-lattices. The important point here is that this allows one to unify the study into a given common problem, mainly performing a classification at the level of unimodular $\Z$-lattices.

\subsection{From monodromies to isometries of unimodular \texorpdfstring{$\Z$}{Z}-lattices}\label{monodrom to unimod}
In this subsection, we review the strategy presented in \cite[Sections~2 and~4.2]{bc22} to classify monodromies of odd prime order, and we fix notation. 
But before that, we make a small remark on nonstable Hodge monodromies for the known BBF forms. 
Indeed, we recall the following.

\begin{lem}\label{lem:disc action}
    Let $\T$ be a known deformation type of IHS manifolds. Then, the orthogonal representation of $\Mon^2(\Lambda_\T)$ on $D_{\Lambda_\T}$ has order $2$, except for $\T = \textnormal{K3}, \textnormal{K3}^{[2]}$, where it has order $1$.
\end{lem}

\begin{proof}
    This follows from the monodromy computations of \cite[Lemma 4.2]{mar10}, \cite[Theorem 2.3]{mon16}, \cite[Theorem 1.4]{mr21} and \cite[Theorem 5.4]{ono22}.  
\end{proof}

We see that in general, if one goes beyond the case of odd order, one has to split the work into two parts: the case where the action on $D_{\Lambda_\T}$ is trivial, and the nontrivial case. 
We show in \Cref{subsec: type study} that for the known types of IHS manifolds, nonstable Hodge monodromies with minimal polynomial $\Phi_1\Phi_m$ only exist for finitely many deformation types and orders.%%\smallskip

We introduce some notation of \cite[Section~4]{bc22} that we use throughout the paper.
Let $\T$ be a known deformation type of IHS manifolds, let $\Lambda \coloneqq  \Lambda_\T$, and let $M \coloneqq  M_\T$ be the corresponding (\emph{extended}\,) \emph{Mukai lattice} (see \Cref{tab:extdata}).
We set $V \coloneqq  \Lambda^\perp_{M}$. 
For $\T = \textnormal{K3}^{[n]}, \textnormal{Kum}_n$, we let $h_V \coloneqq  -\id_{V}$. For $\T = \textnormal{OG6}, \textnormal{OG10}$, we let $h_V$ be represented by the matrix $\begin{psmallmatrix}0&1\\1&0\end{psmallmatrix}$ on a basis of $V$ such that
\begin{align*}
    &V = \begin{pmatrix}
        2&0\\0&2
    \end{pmatrix}\text{ if $\T=\textnormal{OG6}$} \quad\text{and}\quad V = \begin{pmatrix}
        2&-1\\-1&2
    \end{pmatrix}\text{ if $\T=\textnormal{OG10}$}.
\end{align*}
Note that in all cases, $h_V$ is nonstable, except for $\T =  \textnormal{K3},\,\textnormal{K3}^{[2]}$, where the discriminant group has no nontrivial automorphisms. 
 For $\T = \textnormal{K3}^{[n]}, \textnormal{Kum}_n$, we let $v\coloneqq  0$ be the zero vector, and for $\T = \textnormal{OG6}, \textnormal{OG10}$, we let $v$ be the sum the basis vectors of $V$, for the given bases. In all cases, note that $v$ generates the invariant sublattice of $(V, h_V)$. Finally, we let $ O(M, V, v)$ be the joint stabilizer. The following lemma is known from \cite{bc22}; we give a proof for completeness.

\begin{lem}\label{lem: well defined extension}
    There exists a well-defined restriction map $O(M, V, v)\to O(\Lambda)$ that admits a section 
    \[ \gamma\colon \Mon^2(\Lambda)\lra O(M, V, v), \quad h\longmapsto \widehat{\chi}(h)\oplus h,\]
    where $\widehat{\chi}(h) = \id_V$ if\, $h$ is stable, and $\widehat{\chi}(h) = h_V$ otherwise.
\end{lem}

\begin{proof}
    By the definition of $O(M, V, v)$, since $\Lambda$ is the orthogonal complement of $V$ in $M$, we have that 
    \[ r\colon O(M, V, v)\lra O(\Lambda),\quad  g\longmapsto g_{\mid \Lambda}\]
    is well defined. Moreover, by \Cref{lem:disc action}, we know that any $h\in \Mon^2(\Lambda)$ acts with order at most 2 on~$D_\Lambda$. In fact, for $\T=\textnormal{K3}^{[n]},\,\textnormal{Kum}_n,\, \textnormal{OG10}$, we have that $D_h = \pm \id_{D_\Lambda}$, and if $\T = \textnormal{OG6}$, then $D_h$ is either the identity or the map exchanging two generators. It follows from the definition of $\widehat{\chi}(h)$ that in both cases, $D_h$ trivial or not, the isometries $D_h\in O(D_\Lambda)$ and $D_{\widehat{\chi}(h)}\in O(D_V)$ agree along the glue map $D_\Lambda\to D_V$. Hence $\widehat{\chi}(h)\oplus h$ extends along the primitive extension $V\oplus \Lambda \leq M$ to an integral isometry preserving $V$ and fixing~$v$.
\end{proof}

\begin{notat}
For a sublattice $N\leq M$, we define $S(N) \coloneqq  O(N)$ if $\T = \textnormal{K3}^{[n]}, \textnormal{OG6}, \textnormal{OG10}$ and $S(N) \coloneqq  \SO(N)$ if $\T = \textnormal{Kum}_n$. Moreover, we let $S(M, V, v) \coloneqq  S(M)\cap O(M, V, v)$.
\end{notat}

According to \cite[Lemma 4.6]{bc22}, the image of $\gamma$ is given by the kernel $G$ of
\[ \vartheta\cdot \chi_V\colon S(M, V, v)\lra \{\pm 1\},\]
where $\vartheta$ is the character induced by the real spinor norm morphism $\spin_\mathbb{R}$ on $O(M\otimes_\Z\mathbb{R})$, and $\chi_V$ is the natural character induced by the composite morphism $S(M, V, v) \to O(V) \to O(D_V)$. 

\begin{rem}
    Analogously to \cite{bc22}, we choose the following convention regarding real spinor norms. For any anisotropic vector $v\in(L, b)\otimes_\mathbb{Z}\mathbb{R}$, the reflection $\tau_v$ defined by $v$ has real spinor norm 
    \[\spin_\mathbb{R}(\tau_v) \coloneqq  -b_\mathbb{R}(v,v)(\mathbb{R}^\times)^2.\]
    In particular, if $(L, b)$ has real signatures $(l_+, l_-)$, then $\spin_\mathbb{R}(-\id_L) = (-1)^{l_+}(\mathbb{R}^\times)^2$.
\end{rem}

\begin{rem}\label{rem: K compl}
    For a known deformation type $\T\neq \textnormal{K3}$, the $\Z$-lattice $K \coloneqq  v^\perp_V$ has rank 1. 
    If $h\in\Mon^2(\Lambda)$ is stable, then the image of $K$ in $M$ is contained in $M^{g}$, where $g\coloneqq \gamma(h)$. 
    Otherwise, since $K$ is the $(-1)$-sublattice of $(V, h_V)$, we have that $K$ is contained in $M^{-g}$. In particular, if $m_h(1) \neq 0$, where $m_h$ is the minimal polynomial of $h$, then $M^{-g}$ has rank 1 and it is equal to $K$.
    The isometry class of $K$ can be read from \Cref{tab:extdata}.
\end{rem}

\begin{lem}\label{lem: character for K}
    Let $g\in G$. Then $\chi_V(g) = +1$ if and only if\, $K\leq M^{g}$, and $\chi_V(g) = -1$ if and only if\, $K\leq M^{-g}$.
\end{lem}

\begin{proof}
   By the definition of $G$ and $\chi_V$, we have that $\chi_V(g) = +1$ if and only if $g = \id_V\oplus h$ for some $h\in \Mon^2(\Lambda)$ with $D_h$ trivial. Otherwise, since $K$ is the $(-1)$-kernel sublattice of $(V, h_V)$, we have that $\chi_V(g) = -1$ if and only if $g = h_V\oplus h$ for some $h\in\Mon^2(\Lambda)$ with $D_h$ nontrivial. We conclude with \Cref{rem: K compl}.
\end{proof}

{
\renewcommand\arraystretch{1.5}
    \begin{table}[!ht]   
    \caption{Known deformation types and unimodular data}\label{tab:extdata}
    \centering
    \begin{tabular}{ccccccc}
        $\T$&$\Lambda$&$\Mon^2(\Lambda)$&$M$&$V$&$v^2$&$K$ \\
         \hline
         \rowcolor{lightgray!40!white}$\textnormal{Kun}_n$,\; $n\geq 2$&$U^{\oplus 3}\oplus A_1(n+1)$&$\mathcal{N}^+(\Lambda)$&$U^{\oplus 4}$&$\langle 2n+2\rangle$&$0$&$\langle 2n+2\rangle$\\
        OG6&$U^{\oplus 3}\oplus A_1^{\oplus 2}$&$O^+(\Lambda)$&$U^{\oplus 5}$&$\langle 2\rangle^{\oplus 2}$&$4$&$\langle 4\rangle$\\
        \rowcolor{lightgray!40!white}K3&$U^{\oplus 3}\oplus E_8^{\oplus 2}$&$O^+(\Lambda)$&$U^{\oplus 3}\oplus E_8^{\oplus 2}$&$\{0\}$&0&$\{0\}$\\
        $\textnormal{K3}^{[n]}$,\; $n\geq 2$&$U^{\oplus 3}\oplus E_8^{\oplus 2}\oplus A_1(n-1)$&$\mathcal{W}^+(\Lambda)$&$U^{\oplus 4}\oplus E_8^{\oplus 2}$&$\langle 2n-2\rangle$&$0$&$\langle 2n-2\rangle$\\
        \rowcolor{lightgray!40!white}OG10&$U^{\oplus 3}\oplus E_8^{\oplus 2}\oplus A_2$&$O^+(\Lambda)$&$U^{\oplus 5}\oplus E_8^{\oplus 2}$&$A_2(-1)$&$2$&$\langle 6\rangle$\\
         \hline
    \end{tabular}
\end{table}}

\begin{notat}
Following \cite[Section~4]{ms17}, given an even $\Z$-lattice $L$, we set 
\begin{align*}
    \mathcal{W}^+(L) &\coloneqq  \left\{f\in O^+(L)\;\colon D_f = \pm\id_{D_L}\right\},\\
    \mathcal{N}^+(L) &\coloneqq  \left\{f\in O^+(L)\;\colon\det(f)D_f = \id_{D_L}\right\}.
\end{align*}
\end{notat}

\subsection{Type study}\label{subsec: type study}
According to \Cref{rem: K compl}, for a fixed known deformation type $\T$ of IHS manifolds, and for $h\in\Mon^2(\Lambda_\T)$ with minimal polynomial $\Phi_1\Phi_m$ ($m\geq 3$),  either $D_h$ is trivial and $m_{\gamma(h)} = \Phi_1\Phi_m$ too, or $D_h$ has order 2 and, by definition, $m_{\gamma(h)} = \Phi_1\Phi_2\Phi_m$, where the $(-1)$-sublattice of $(M_\T, \gamma(h))$ has rank 1. 
In this subsection, we study the types of such unimodular lattices with isometry in order to have more information on the genera of their kernel sublattices.%\smallskip

The first case to consider is a generalization of the odd prime-order case as studied in \cite[Section~2]{bc22}.

\begin{propo}\label{propo: case 1 minpoly}
Let $(M, g)$ be a unimodular $\Phi_1\Phi_m^\ast$-lattice, where $m\geq 2$.
All of the following hold:
\begin{enumerate}
    \item\label{p:c1m-1} The isometries $D_{g_1}$ and $D_{g_m}$ are trivial.
    \item\label{p:c1m-2} If $m$ is composite, then both $M^g$ and $M_g$ are unimodular.
    \item\label{p:c1m-3} If $m = p^k$ is a prime power, then both $M^g$ and $M_g$ are $p$-elementary.
\end{enumerate}
\end{propo}

\begin{proof}
See for instance \cite[Proposition 4.8]{bay24b}.
\end{proof}

Similarly to \Cref{propo: case 1 minpoly}, we want to obtain information about the kernel sublattices of a unimodular $\Phi_1\Phi_2\Phi_m^\ast$-lattice for $m\geq 3$ even.  

\begin{propo}\label{propo: case 2 minpoly}
   Let $(M, g)$ be a unimodular $\Phi_1\Phi_2\Phi_m^\ast$-lattice, where $m\geq 3$ is even. 
   One of the following holds:
    \begin{enumerate}
        \item\label{p:c2m-1} If $m = 2^k$ is a power of\, $2$, then $M^{\Phi_m(g)}$ is $2$-elementary, and $M^g$ and $M^{-g}$ are both $4$-elementary.
        \item\label{p:c2m-2} If $m = 2p^k$ is twice an odd prime power, then $M^{\Phi_m(g)}$ is $p$-elementary, $M^{-g}$ is $2p$-elementary and $M^g$ is $2$-elementary.
        \item\label{p:c2m-3} Otherwise, $M^{\Phi_m(g)}$ is unimodular, and $M^g$ and $M^{-g}$ are both $2$-elementary.
    \end{enumerate}
\end{propo}

\begin{proof}
  For what follows, we set $C \coloneqq  M^{\Phi_m(g)}$ and $F \coloneqq  (C)^\perp_M = M^{g^2-1}$. Note that $(M, g^2)$ is a $\Phi_1\Phi_{m/2}^\ast$-lattice, and $F$ and $C$ are the respective invariant and coinvariant sublattices---we can therefore obtain information on the last two by applying \Cref{propo: case 1 minpoly} to $g^2$.

(\ref{p:c2m-1})~ Since $m=2^k \geq 4$, \Cref{propo: case 1 minpoly}(\ref{p:c1m-1}) gives that $F$ and $C$ are both $2$-elementary.
  Moreover, $g_{\mid F}$ has order~2, so \Cref{good result pelem} tells us that both $M^g$ and $M^{-g}$ are $4$-elementary.

(\ref{p:c2m-2})~ We now have  that $g^2$ has odd order $p^k$, so \Cref{propo: case 1 minpoly} tells us that $F$ and $C$ are both $p$-elementary. 
        Thus, according to \Cref{good result pelem}, we have that $M^g$ and $M^{-g}$ are both $2p$-elementary. Note moreover that since $g$ has minimal polynomial $\Phi_1\Phi_2\Phi_m$, we have that $M^g = M^{g^{p^k}}$; since $g^{p^k}$ has order 2, we deduce that $M^g$ is actually 2-elementary (see \Cref{propo: case 1 minpoly}(\ref{p:c1m-2})).

(\ref{p:c2m-3})~ In this case, the order of $g^2$ is not a prime power, and by \Cref{propo: case 1 minpoly}(\ref{p:c1m-3}), we have that both $F$ and~$C$ are unimodular. By \Cref{propo: case 1 minpoly}(\ref{p:c1m-2}), we conclude that $M^g$ and $M^{-g}$ are $2$-elementary.
\end{proof}

By \Cref{propo: case 2 minpoly}, there are strong restrictions on the local invariants of the $\Phi_2$-kernel sublattice of any $\Phi_1\Phi_2\Phi_m^\ast$-lattice $(M, g)$. 
In particular, in the case where $M\simeq M_\T$ for a known deformation type $\T$ of IHS manifolds and $g$ is induced by an isometry of $\Lambda_\T$ with nontrivial discriminant action, we have seen in \Cref{rem: K compl} that the kernel sublattice $M^{-g}$ has rank 1 and is uniquely determined.
We can therefore already draw a conclusion about the possible orders such an isometry could have.

\begin{propo}\label{nontriv: all}
    Let $X$ be a projective IHS manifold of known deformation type $\T$. Let $f\in\Bir(X)$ be such that the minimal polynomial of\, $h \coloneqq  \rho_X(f)$ is $\Phi_1\Phi_m$ for some positive integer $m\geq 3$, and suppose that $D_h$ is nontrivial.
    Then the pair $(\T, m)$ appears in \Cref{tab:nontrivact}.
\end{propo}

{\small\setlength{\tabcolsep}{4pt}
\renewcommand\arraystretch{2}
\begin{table}[H]

    \caption{Deformation types and orders for nonstable $\Phi_1\Phi_m$-isometries}\label{tab:nontrivact}\vspace*{0.3cm}
    \centering
    \begin{tabular}{c|cccccccc}
        $\T$&$\cellcolor{lightgray!40!white}\textnormal{OG10}$&$\textnormal{OG6}$&\cellcolor{lightgray!40!white}$\textnormal{K3}^{[3]}$&$\textnormal{K3}^{[4]}$&\cellcolor{lightgray!40!white}$\textnormal{K3}^{[6]}$& $\textnormal{K3}^{[p+1]}$, \ $7\leq p\leq 23 $& \cellcolor{lightgray!40!white}$\textnormal{Kum}_{2}$& $\textnormal{Kum}_{p-1}$,\ $5\leq p\leq 7 $\\
         \hline
         $m$&\cellcolor{lightgray!40!white}$6,\,18,\,54$&$4,\,8$&\cellcolor{lightgray!40!white}$4,\,8,\,16,\,32$&$6,\,18,\,54$&\cellcolor{lightgray!40!white}$10,\,50$&$2p$&\cellcolor{lightgray!40!white}$6,\,18$&$2p$\\
    \end{tabular}
    
\end{table}}

\begin{proof}
    We follow the notation of \Cref{monodrom to unimod}.
    \begin{enumerate}
        \item First suppose that $\T = \textnormal{OG10}$. Then $g \coloneqq  \gamma(h)\in G$ has minimal polynomial $\Phi_1\Phi_2\Phi_m$ (see \Cref{rem: K compl}), and we have that $M^{-g}\simeq \langle 6\rangle$ (see \Cref{tab:extdata}). 
        By \Cref{propo: case 2 minpoly}, $M^{-g}$ is $6$-elementary if and only if $m = 2\cdot 3^k$ for some $k\geq 1$. 
        Note that in this case $3\varphi(m) = m$. Since $H^2(X, \mathbb{Z})$ has rank 24 and $H^2(X, \mathbb{Z})^h$ has rank at least 1, we have $2\leq \varphi(m)\leq 22$ and therefore $6\leq m\leq 66$. 
        \item Now let $\T = \textnormal{K3}^{[n]}$ for some $n\geq 2$---the following can be adapted to prove the statement for $\T = \textnormal{Kum}_n$.  
        According to \Cref{rem: K compl} and \Cref{tab:extdata}, we have that $g \coloneqq  \gamma(h)$ has minimal polynomial $\Phi_1\Phi_2\Phi_m$ and that  $M^{-g} = \langle 2n-2\rangle$ is either unimodular or $4$-elementary, or its discriminant group is of the form $(\mathbb{Z}/2\mathbb{Z})^{\oplus\alpha}\oplus (\mathbb{Z}/p\mathbb{Z})^{\oplus\beta}$ for some $\alpha,\beta \geq 0$. 
        The unimodular case is clearly not possible.
        The $4$-elementary case occurs only when $m$ is a power of 2, and in that situation we know that $M^{g}$ and $M^{-g}$ are both $4$-elementary. 
        In particular, $n-1 = 2$ and $4\leq m\leq 32$ since $\Lambda_{\textnormal{K3}^{[n]}}$ has rank~23.
        For the remaining case, since the discriminant group of $M^{-g}$ is $\mathbb{Z}/(2n-2)\mathbb{Z}$ with $n\geq 3$, then necessarily $\alpha = \beta = 1$. 
        In particular, $n-1 = p$ for an odd prime number $p$ and $m\in \{2p^k\;\colon k\geq 1\}$.
        \item Finally, let us assume that $\T = \textnormal{OG6}$. In a same way as before, $g \coloneqq  \gamma(h)$ has minimal polynomial $\Phi_1\Phi_2\Phi_m$, but this time $M^{-g}\simeq \langle 4\rangle$ is 4-elementary. According to \Cref{propo: case 2 minpoly}, the previous is possible only if $m$ is a power of 2. 
        Since $H^2(X, \mathbb{Z})^h$ has rank at least 1 and $\Lambda_{\textnormal{OG6}}$ has rank 8, we see that $4\leq m \leq 8$ as $\varphi(m) = m/2 < 8$.\qedhere \end{enumerate}
\end{proof}

\begin{rem}\label{rem order og10}
    We actually show later that some of the cases in \Cref{tab:nontrivact} cannot occur if we moreover use the assumption that $f$ is nonsymplectic (see \Cref{propo: discard nontriv act}). 
\end{rem}

\section{Constructing isometries using hermitian lattices}\label{sec: existence isometries} In the next two sections, we review existence conditions for even unimodular $\Z$-lattices equipped with an isometry of given minimal polynomial. 
We focus in particular on the case where the minimal polynomial is $\Phi_1\Phi_m$, or $\Phi_1\Phi_2\Phi_m$ and the $\Phi_2$-kernel sublattice is isometric to the $\mathbb{Z}$-lattice $\langle2p\rangle$ for some prime number~$p$ (see \Cref{subsec: type study}). 
Once this is done, we explain how one can classify such isometries up to conjugacy, and give a way of estimating the number of such conjugacy classes.

\begin{rem}
    Necessary and sufficient conditions for the existence of a unimodular $\Phi_1\Phi_m$-lattice $(M, g)$ of given signatures are given in \cite[Lemma 4.2 and Theorem 4.5]{bay24b}. In this section, we complement these existence conditions in order to classify these lattices with isometry up to isomorphism.
\end{rem}

\subsection{Local information} Let $p$ be a prime number, let $m = p^k\geq 3$ be a power of $p$, and let $\zeta \coloneqq  \zeta_m$ be a primitive $\supth{m}$ root of unity. 
Let us  moreover set $E \coloneqq  \mathbb{Q}(\zeta)$, $K \coloneqq  \mathbb{Q}(\zeta + \zeta^{-1})$, $\pi \coloneqq  1-\zeta$, $\mathfrak{P} \coloneqq  \pi\mathcal{O}_E$ and $\mathfrak{p} \coloneqq  \mathfrak{P}\cap\mathcal{O}_K$. 
According to \Cref{lem basic cyclo}, $\mathfrak{P}$ (resp.\ $\mathfrak{p}$) is the unique prime ideal of $\mathcal{O}_E$ (resp.\ $\mathcal{O}_K$) that divides $p\mathcal{O}_E$ (resp.\ $p\mathcal{O}_K$).
Let $(L, b, f)$ be an even $\Phi_m$-lattice, and suppose that $D_f$ has order $p^l\leq p^{k-1}$.
Similarly to \Cref{propo: case 1 minpoly}, one can show that the previous condition implies that $(L, b)$ is $p$-elementary.
We denote by $(L, h)$ the hermitian structure of $(L, b, f)$ (see \Cref{subsec trace}). 
We recall that, by transfer, the equality $L^{\vee} = \mathfrak{D}_{E/\mathbb{Q}}^{-1}L^{\#}$ holds. 
Since by assumption $(L, b)$ is even, we moreover have that $ \mathfrak{n}(L)\subseteq \mathfrak{D}_{K/{\mathbb{Q}}}^{-1}.$
Following the trace equivalence, the $\mathcal{O}_E$-module structure on $L$ is given by $f$. Hence, since $\ord(D_f) = p^l$, we have the following inclusions:
\[
    \left(1-\zeta^{p^l}\right)L^{\vee} \leq L\leq L^{\vee}.
\]
Now, if we replace $L^\vee$ by $\mathfrak{D}_{E/\mathbb{Q}}^{-1}L^{\#}$, we can use the fact that $\mathfrak{D}_{E/\mathbb{Q}}$ is principal, generated by $\pi^{\alpha}$, where  $\alpha \coloneqq  p^{k-1}(pk-k-1)$ (see \Cref{lem basic cyclo}) and that $(1-\zeta^{p^l})\mathcal{O}_E = \pi^{p^{l}}\mathcal{O}_E$ (see \Cref{ramif tower}) to translate the previous equation into 
\begin{equation}\label{eq: decomp}
    \pi^{p^{l}-\alpha}L^{\#}\leq L\leq \pi^{-\alpha}L^{\#}.
\end{equation}

\begin{propo}\label{propo: decomp block val}
For all finite places $ \mathfrak{q}\neq \mathfrak{p}$ of $K$, the hermitian  $\mathcal{O}_{E_{\mathfrak{q}}}$-lattice $L_{\mathfrak{q}}$ is unimodular and uniquely determined by its rank.
Moreover, any Jordan decomposition of\, $L_{\mathfrak{p}}$ consists of at most $p^{l}+1$ orthogonal Jordan constituents that are, respectively, $\mathfrak{P}^{j-\alpha}$-modular for $j=0,\ldots, p^{l}$.
\end{propo}

\begin{proof}
    We have that $\pi\mathcal{O}_E = \mathfrak{P}$ by definition, and $\mathfrak{p} = \mathfrak{P}\cap\mathcal{O}_K$ represents the only bad finite place of $K$ (see \Cref{lem basic cyclo}).
    From \Cref{eq: decomp}, we therefore have that for all (good) finite places $\mathfrak{q}\neq \mathfrak{p}$ of $K$, the hermitian lattice $L_{\mathfrak{q}}$ is unimodular and uniquely determined by its rank; see \cite[Proposition 3.3.5]{kir16}. The Jordan decomposition at $\mathfrak{p}$ follows from the same equation.
\end{proof}

Therefore, in order to study the genus of $(L, h)$, we need to understand the Jordan constituents of $L_\mathfrak{p}$ (see \Cref{propo: decomp block val}). Recall that $\mathfrak{p}$ is the only bad place of $K$.

\begin{propo}\label{propo: orth decomp}
Let $L$ be a $\mathfrak{P}^i$-modular hermitian $\mathcal{O}_{E_{\mathfrak{p}}}$-lattice of rank $r$, for some $i\in\mathbb{Z}$. 
Let us write $\pi \coloneqq  1-\zeta_m$ and $\beta\coloneqq  \pi\iota(\pi)$, where $\iota\in\Gal(E_\mathfrak{p}/K_\mathfrak{p})$ is a generator.
\begin{enumerate}
    \item\label{p:od-1} If\, $p$ is odd, then
    \begin{enumerate}
        \item\label{p:od-1a} either $i$ is even and $L\simeq \langle\beta^{i/2},\ldots, u\beta^{i/2}\rangle$, where $uN^{E_\mathfrak{q}}_{K_\mathfrak{q}}(\mathcal{O}_{\mathfrak{q}}^\times) = \det(L)$;
        \item\label{p:od-1b} or $i$ is odd, $r$ is even, and $L\simeq H(\pi^i)^{\oplus r/2}$.
    \end{enumerate}
    \item\label{p:od-2} If\, $p=2$, then
    \begin{enumerate}
        \item\label{p:od-2a} either $r$ is odd, $i$ is even, and
        \[L\simeq \left\langle a\beta^{i/2}\right\rangle\oplus H\left(\pi^i\right)^{\oplus (r-1)/2}\]
        with $aN^{E_\mathfrak{p}}_{K_\mathfrak{p}}(E_\mathfrak{p}^\times) = (-1)^{(r-1)/2}\det(L)$ and $\mathfrak{n}(L)\mathcal{O}_{E_\mathfrak{p}} = \mathfrak{s}(L)$;
        \item\label{p:od-2b} or $r$ is even, $\mathfrak{n}(L) = \mathfrak{p}^k$ with 
        \[\mathfrak{P}^{i+2}\subseteq \mathfrak{p}^k\mathcal{O}_{E_{\mathfrak{p}}}\subseteq \mathfrak{P}^i,\]
        and $\mathfrak{p}^k\mathcal{O}_{E_{\mathfrak{p}}} = \mathfrak{P}^{i+2}$ if and only if\,  $L\simeq H(\pi^i)^{\oplus r/2}$. 
        Moreover,
        \[u^\epsilon N^{E_\mathfrak{p}}_{K_\mathfrak{p}}\left(E_\mathfrak{p}^\times\right) = (-1)^{r/2}\det(L), \]
        where $u\notin N^{E_\mathfrak{p}}_{K_\mathfrak{p}}(E_\mathfrak{p}^\times)$ and $\epsilon\in\{0,1\}$, with $\epsilon=0$ if and only if $L$ is hyperbolic.
    \end{enumerate}
\end{enumerate}
\end{propo}

\begin{proof}
This is a translation of \cite[Proposition 3.3.5 and Corollary 3.3.20]{kir16} to the prime-power cyclotomic case, together with \Cref{lem basic cyclo}, which tells us that the different ideal $\mathfrak{D}_{E_\mathfrak{p}/K_\mathfrak{p}}$ is generated by $\pi^{\gcd(2,p)}$.
\end{proof}

\begin{rem}\label{rem: cyclo modular bad}
    It is good to note that for cases (\ref{p:od-2})(\ref{p:od-2a}) and (\ref{p:od-2})(\ref{p:od-2b}) in \Cref{propo: orth decomp}, we have a fine description of $\mathfrak{n}(L) = \mathfrak{p}^k$. In fact, suppose that $p=2$, and let $L$ be $\mathfrak{P}^i$-modular of rank $r\geq 1$ for some $i\geq 1$. If $r$ is odd, then $i$ is even, and since $\mathfrak{p}\mathcal{O}_{E_\mathfrak{p}} = \mathfrak{P}^2$, we have that $\mathfrak{n}(L) = \mathfrak{p}^\frac{i}{2}$. Now, if $r$ is even, there are two cases. Either $i$ is odd, in which case so is $i+2$, and thus we must have $\mathfrak{n}(L) = \mathfrak{p}^\frac{i+1}{2}$. Otherwise, $i$ is even and $\mathfrak{n}(L)\in\{\mathfrak{p}^\frac{i+2}{2}, \mathfrak{p}^\frac{i}{2}\}$. The two cases are distinguished by $L$ being isometric to $H(\pi^i)^{\oplus r/2}$ or not.
\end{rem}

\begin{rem}\label{even rank jordan block}
    Note that since $\alpha=p^{k-1}(pk-k-1)$ and $p$ have the same parity, for all $j=0,\ldots, p^{l}$ such that $j+p$ is odd, the rank of a $\mathfrak{P}^{j-\alpha}$-modular hermitian $\mathcal{O}_{E_{\mathfrak{p}}}$-lattice is even according to \Cref{propo: orth decomp}.
\end{rem}

From this description of modular hermitian $\mathcal{O}_{E_{\mathfrak{p}}}$-lattices at the unique bad place $\mathfrak{p}$ of $K$, we can now state and prove a couple of results about the trace lattice of a hermitian $\mathcal{O}_E$-lattice.

\begin{lem}\label{lem even trace}
    Let $(L, h)$ be a hermitian $\mathcal{O}_E$-lattice with trace lattice $(L, b, f)$. We assume that $(L, b)$ is integral, i.e., $\mathfrak{s}(L, h)\subseteq \mathfrak{D}_{E/\mathbb{Q}}^{-1}$, and that $D_f$ has order $p^l\leq p^{k-1}$. 
    Then $(L, b)$ is even if and only if $p$ is odd or the $\mathfrak{P}^{-\alpha}$-modular Jordan constituent of\, $L_\mathfrak{p}$ is isometric to $H(\pi^{-\alpha})^{\oplus r}$ for some $r\in\mathbb{Z}_{\geq 0}$. 
\end{lem}

\begin{proof}
    In the case where $p$ is odd, this is a direct generalization of \cite[Lemma 2.4]{bc22}. 
    Now suppose that $p=2$; in this case, we know that $\mathfrak{D}_{K/\mathbb{Q}}^{-1}\mathcal{O}_E = \pi^{2-\alpha}\mathcal{O}_E$ (see \Cref{lem basic cyclo}). 
    By \cite[Proposition 6.6]{bh23}, we know that $(L, b)$ is even if and only if $\mathfrak{n}(L, h)\mathcal{O}_E\subseteq \mathfrak{D}_{K/\mathbb{Q}}^{-1}\mathcal{O}_E$. 
    Since for all prime ideals $\mathfrak{q}$ of $K$ not dividing $2\mathcal{O}_K$, we have an equality $\mathfrak{D}_{E_\mathfrak{q}/K_{\mathfrak{q}}} = \mathcal{O}_{E_\mathfrak{q}}$, we already know that locally at $\mathfrak{q}$, the inclusions
    \[ \mathfrak{n}\left(L_\mathfrak{q}\right)\mathcal{O}_{E_\mathfrak{q}} \subseteq \mathfrak{s}\left(L_\mathfrak{q}\right)\subseteq \mathfrak{D}_{E_\mathfrak{q}/\mathbb{Q}_{q}}^{-1} = \mathfrak{D}_{K_\mathfrak{q}/\mathbb{Q}_q}^{-1}\mathcal{O}_{E_\mathfrak{q}},\]
    where $q\mathbb{Z} \coloneqq  \mathfrak{q}\cap\mathbb{Z}\neq 2\mathbb{Z}$, hold.
    According to \Cref{eq: decomp}, $L_{\mathfrak{p}}$ admits a Jordan decomposition of the form 
    $ L_\mathfrak{p} = \bigoplus_{i=0}^{p^l}N_i$,
    where $N_i$ is $\mathfrak{P}^{i-\alpha}$-modular of rank $n_i\geq 0$. By definition, for all $i\geq 2$,
    \[\mathfrak{n}(N_i)\mathcal{O}_{E_\mathfrak{p}}\subseteq \mathfrak{s}(N_i) = \mathfrak{P}^{i-\alpha}\subseteq \mathfrak{P}^{2-\alpha} = \pi^{2-\alpha}\mathcal{O}_{E_\mathfrak{p}}.\]
    Moreover, if $n_1\neq 0$, then since $1-\alpha$ is odd, \Cref{rem: cyclo modular bad} gives us that $\mathfrak{n}(N_1)\mathcal{O}_{E_\mathfrak{p}} = \pi^{2-\alpha}\mathcal{O}_{E_\mathfrak{p}}.$
    So the only obstruction for $(L, b)$ to be even should come from $N_0$. 
    According to \Cref{propo: orth decomp}, either $n_0$ is odd and $\mathfrak{n}(N_0)\mathcal{O}_{E_{\mathfrak{p}}} = \mathfrak{s}(N_0) = \pi^{-\alpha}\mathcal{O}_{E_{\mathfrak{p}}}$, or $n_0$ is even and \(\pi^{2-\alpha}\mathcal{O}_{E_{\mathfrak{p}}}\subseteq \mathfrak{n}(N_0)\mathcal{O}_{E_{\mathfrak{p}}}\subseteq \pi^{-\alpha}\mathcal{O}_{E_{\mathfrak{p}}}.\)
    We conclude by remarking that since $\pi$ is not invertible in $\mathcal{O}_{E_{\mathfrak{p}}}$, then $\mathfrak{n}(N_0)\subseteq \pi^{2-\alpha}\mathcal{O}_{E_\mathfrak{p}}$ if and only if $\mathfrak{n}(N_0)= \pi^{2-\alpha}\mathcal{O}_{E_\mathfrak{p}}$, which holds if and only if $n_0$ is even and $N_0\simeq H(\pi^{-\alpha})^{\oplus n_0/2}$.
\end{proof}

We are now equipped to prove the following result, generalizing the necessary conditions from \cite[Proposition 2.15]{bc22}.

\begin{propo}\label{propo: info hermitian structure}
    Let $p$ be a prime number, let $l_+,l_-\geq 0$ be such that $l_++l_-\geq 1$, and let $k\geq \gcd(2,p)$ be positive. 
    If there exists an even $\Phi_{p^k}$-lattice $(L, b, f)$ of signatures $(l_+, l_-)$ such that $D_f$ has order $p^l<p^k$, then there exist integers $n_0,\ldots, n_{p^l}$, with $n_0$ even and $n_i$ even if $i+p$ is odd, such that
    \begin{enumerate}
        \item $(L, b)$ is $p$-elementary of absolute determinant $p^{\sum_{i=0}^{p^l}in_i}$;
        \item $l_++l_- = \varphi(p^k)\sum_{i=0}^{p^l}n_i$;
        \item $l_+, l_-\in 2\mathbb{Z}$;
        \item if\, $l > 0$, one among $n_{p^{l-1}+1}, \ldots, n_{p^l}$ is nonzero.
    \end{enumerate}
\end{propo}

\begin{proof}
 Let $\zeta$ be a primitive $\supth{(p^k)}$ root of unity, and let us  again set $E\coloneqq  \mathbb{Q}(\zeta)$ and $K \coloneqq  \mathbb{Q}(\zeta+\zeta^{-1})$. Recall that $\mathfrak{P}$ and $\mathfrak{p}$ are the unique prime ideals of $\mathcal{O}_E$ and $\mathcal{O}_K$, respectively, lying above $p\mathbb{Z}$.
    
    Let $(L, b, f)$ be an even $\Phi_{p^k}$-lattice of signatures $(l_+, l_-)$ and such that $D_f$ has order $p^l$ for some $0\leq l < k$. 
    Let moreover $(L, h)$ be its hermitian structure. 
    By \Cref{ranks trace constr}, we know that the rank of $L$ as a $\mathbb{Z}$-module is even, divisible by $\varphi(p^k)$, and $l_+$ and $l_-$ are both even according to \Cref{sign trace constr}. 
    Moreover, since $l<k$, we know that $(L, b)$ is $p$-elementary.
    According to \Cref{propo: decomp block val}, the local Jordan decomposition of $(L,h)$ at $\mathfrak{p}$ consists of at most $p^l+1$ Jordan constituents $N_0,\ldots, N_{p^l}$, and
    \[L_{\mathfrak{p}}=\bigoplus_{0\leq i \leq p^l}N_i.\]
    The $\mathfrak{P}$-valuations of their respective scales are $-\alpha, 1-\alpha,\ldots, p^l-\alpha$, where $\alpha \coloneqq  p^{k-1}(pk-k-1)$.
    For all $0\leq i\leq p^l$, let us  write $n_i \coloneqq  \rank_{\mathcal{O}_{E_\mathfrak{p}}}(N_i)$. 
    By \Cref{propo: orth decomp} and \Cref{lem even trace}, we have that $N_0$ is hyperbolic, and thus of even rank.
    The fact that $n_i$ is even whenever $i+p$ is odd follows from \Cref{even rank jordan block}.
    Now, since the rank of $L$, as an $\mathcal{O}_E$-module, is equal to the rank of $L_{\mathfrak{p}}$, as an $\mathcal{O}_{E_\mathfrak{p}}$-module, we have that  
    \[ \rank_\mathbb{Z}(L) = l_++l_- = \varphi\left(p^k\right)\sum_{0\leq i \leq p^l}n_i.\]
    Finally, one observes that
    \(L^{\vee}/L = \mathfrak{D}_{E/\mathbb{Q}}^{-1}L^{\#}/L\simeq \bigoplus_{0\leq i \leq p^l}\mathfrak{D}_{E_\mathfrak{p}/\mathbb{Q}_p}^{-1}N_i^{\#}/N_i\)
    since $(L, h)$ is locally unimodular at all finite places of $K$ outside of $\mathfrak{p}$. 
    If we set $\pi \coloneqq  1-\zeta$, since for all $0\leq i\leq p^l$ the hermitian lattice $N_i$ is $\pi^i\mathfrak{D}_{E_\mathfrak{p}/\mathbb{Q}_p}^{-1}$-modular, 
    one has that for all $0\leq i\leq p^l$, as abelian groups,
    \[\mathfrak{D}_{E_\mathfrak{p}/\mathbb{Q}_p}^{-1}N_i^{\#}/N_i = \pi^{-i}N_i/N_i\cong \left(\mathcal{O}_E/\mathfrak{P}^i\right)^{\oplus n_i},\]
    where the latter is an $\mathcal{O}_E/\mathfrak{P}$-vector space of dimension $in_i$, with $\mathcal{O}_E/\mathfrak{P}\cong \mathbb{F}_p$ (see \Cref{lem basic cyclo}).
    Therefore, $\val_p(\det(L, b)) = \sum_{i=0}^{p^l}in_i$. 
    Finally, if $l\geq 1$ and if $n_j = 0$ for all $p^{l-1}+1\leq j\leq p^l$, we would have that $\pi^{p^{l-1}-\alpha}L_\mathfrak{p}^\#\subseteq L_\mathfrak{p}$, contradicting the fact that $D_f$ has order $p^l$. 
\end{proof}

Before finishing our section on local information of hermitian lattices arising as hermitian structure of some $\Phi_m$-lattices, we prove the following lemma. 
Recall that for a $2$-elementary integral $\Z$-lattice $(L, b)$, one defines
\[
\delta_L \coloneqq \begin{cases}\,0&\text{if }n(L^{\vee})\subseteq \mathbb{Z},\\\,1&\text{else.}\end{cases}
\]

\begin{lem}\label{lem delta 2-elem}
     Let $p = 2$, and let $m = 2^k$ for some $k\geq 2$.
     Let $(L, h)$ be a hermitian $\mathcal{O}_E$-lattice with trace lattice $(L, b, f)$.
     We assume that $\mathfrak{s}(L, h)\subseteq \mathfrak{D}_{E/\mathbb{Q}}^{-1}$ and that $D_f$ has order $2^l\leq 2^{k-1}$. 
     Then $\delta_L = 0$  if and only if the $\mathfrak{P}^{2^{k-1}-\alpha}$-modular Jordan constituent of\, $L_\mathfrak{p}$ is isometric to $H(\pi^{2^{k-1}-\alpha})^{\oplus r}$ for some $r\in\mathbb{Z}_{\geq 0}$. 
\end{lem}

\begin{proof}
    According to \Cref{trace of norm}, $n(L^\vee)\subseteq \mathbb{Z}$ if and only if $2\mathfrak{n}(\mathfrak{D}_{E/\mathbb{Q}}^{-1}L^\#)\mathcal{O}_E\subseteq \mathfrak{P}^{2-\alpha}$.
    Now $2\mathcal{O}_E = \mathfrak{P}^{2^{k-1}}$ and therefore
    \[ 2\mathfrak{n}\left(\mathfrak{D}_{E/\mathbb{Q}}^{-1}L^\#\right)\mathcal{O}_E = \mathfrak{n}\left(\pi^{2^{k-2}}\mathfrak{D}_{E/\mathbb{Q}}^{-1}L^\#\right)\mathcal{O}_E.\]
    But, since the order of $D_f$ is at most $2^{k-1}$, we know that $(L, b)$ is $2$-elementary and $(L, h)$ is locally unimodular at all finite places $\mathfrak{q}\neq \mathfrak{p}$ of $K$. 
    In particular, $\mathfrak{n}(\pi^{2^{k-2}}\mathfrak{D}_{E/\mathbb{Q}}^{-1}L^\#) = \mathfrak{n}(\pi^{2^{k-2}}\mathfrak{D}_{E_\mathfrak{p}/\mathbb{Q}_2}^{-1}L_\mathfrak{p}^\#)$. 
    By \Cref{propo: decomp block val}, we have that
    \[\pi^{2^{k-2}}\mathfrak{D}_{E_\mathfrak{p}/\mathbb{Q}_2}^{-1}L_\mathfrak{p}^\#\simeq \bigoplus_{j=0}^{2^{k-1}}\pi^{-j+2^{k-2}}N_j, \]
    where $N_j$ is $\mathfrak{P}^{j-\alpha}$-modular. The rest of the proof is similar to that of \Cref{lem even trace} after remarking that for all $0\leq j\leq 2^{k-1}$, since $\alpha = (k-1)2^{k-1}$,
    \[ \mathfrak{s}\left(\pi^{-j+2^{k-2}}N_j\right) = \pi^{-2j+2^{k-1}}\mathfrak{s}\left(N_j\right) = \mathfrak{P}^{2^{k-1}-j-\alpha}\subseteq \mathfrak{P}^{-\alpha},\]
    with equality if and only if $j = 2^{k-1}$.
\end{proof}

In particular, we can prove the following direct corollary.

\begin{coro}\label{coro taki}
Let $(L, b, f)$ be an even $\Phi_{2^k}$-lattice for some $k\geq 2$, and suppose that $D_f$ has order at most $2^{k-2}$. 
Then $\delta_L = 0$. 
\end{coro}

\begin{proof}
    In this case, the Jordan decomposition of the hermitian structure of $(L, b, f)$ at the prime ideal $\mathfrak{p}$ has constituents whose scales have $\mathfrak{P}$-valuation at most $2^{k-2}-\alpha$. 
    Hence by \Cref{lem delta 2-elem}, we have that $n(L^\vee)\subseteq \mathbb{Z}$ and $\delta_L = 0$.
\end{proof}

\begin{rem}
    This generalizes the result of Taki \cite[Proposition 2.4]{tak12} for the case $k=2$.
\end{rem}

In the next subsection, we prove that \Cref{propo: info hermitian structure} admits a converse if one makes further assumption on the genus of $(L, b)$.

\subsection{Existence of \texorpdfstring{$\boldsymbol{\Phi_m}$}{Phi\textunderscore{m}}-lattices with given invariants}\label{subsec: exist phim lat}
Let $(L, b, f)$ be an even $\Phi_m$-lattice with $m\geq 3$ arbitrary, and suppose that $D_f$ has order at most 2. 
In regard to \Cref{propo: case 1 minpoly,propo: case 2 minpoly}, we have to consider the following cases:
\begin{enumerate}
    \item\label{s:epl-case1} $m$ is composite and $D_f$ is trivial.
    \item\label{s:epl-case2} $m$ is a prime power and $D_f$ is trivial.
    \item\label{s:epl-case3} $m$ is twice an odd prime power and $D_f = -\id_{D_L}$.
    \item\label{s:epl-case4} $m$ is a power of 2 and $D_f$ has order 2.
\end{enumerate}
In the first case, $(L, b)$ is unimodular, and the existence conditions are already well known; see \cite{bay84, mc15,bt20,bay24b}. 
For (\ref{s:epl-case2}), (\ref{s:epl-case3}) and (\ref{s:epl-case4}), we know that $(L, b)$ is $p$-elementary for some prime number $p$, and we can use the results of the previous section.

\begin{rem}\label{rem: do not do twice prime}
    For an odd prime number $p$ and some integer $k\geq 1$, if $(L, b, f)$ is a non-unimodular $\Phi_{p^k}$-lattice with $D_f$ trivial, then $(L, -f)$ is a $\Phi_{2p^k}$-lattice where $D_{-f} = -\id_{D_L}$ has order $2$. 
    And the converse also holds.
    Hence the existence of a $\Phi_{2p^k}$-lattice $(L, f)$ with $D_f = -\id_{D_L}$ is equivalent to the existence of a $\Phi_{p^k}$-lattice whose underlying isometry is stable. 
    Therefore, the case of twice an odd prime power in (\ref{s:epl-case3}) is already covered by (\ref{s:epl-case2}).
\end{rem}

For (\ref{s:epl-case1}), the existence conditions can be reduced to the following proposition.

\begin{propo}\label{existence unimod}
Let $m\geq 3$ be a composite integer, and let $l_+, l_-\geq 0$ with $l_++l_- \geq 1$. 
There exists an even unimodular $\Phi_m$-lattice $(L, b, f)$ of signatures $(l_+, l_-)$ if and only if there exists a positive integer $d>0$ such that
\begin{enumerate}
    \item\label{eu-1} $l_+\equiv l_-\mod 8$ and $l_++l_- = d\varphi(m)$;
    \item\label{eu-2} $l_+,l_-\equiv 0\mod 2$;
    \item\label{eu-3} if\, $m = 2p^k$ for some odd prime number $p$ and positive integer $k>0$, then $d\equiv 0\mod 2$.
\end{enumerate}
Moreover, up to fixing signatures , the genus of the hermitian structure of any such $(L, b, f)$ is uniquely determined by $(l_-, m, d)$.
\end{propo}

\begin{proof}
  The existence part follows from \cite[Th\'eor\`eme V.2.2 et Corollaire V.2.1]{ser70},  \Cref{sign trace constr} and \cite[Theorem 4.5]{bay24b}. 
Now, let $(L, b, f)$ be an even $\Phi_m$-lattice, and let $(L, h)$ be its hermitian structure over the field $E \coloneqq  \mathbb{Q}(\zeta)$, where $\zeta \coloneqq  \zeta_m$ is a primitive $\supth{m}$ root of unity.
If $m$ is not twice a prime power, then \Cref{lem ramif} tells us that the different ideal $\mathfrak{D}_{E/\mathbb{Q}} = \mathcal{O}_E$ is trivial. 
In particular, since $(L, b)$ is unimodular, $L^\# = L$ and the rank of $L$ uniquely determines the isometry class of $L_\mathfrak{p}$ for all finite places $\mathfrak{p}$ of $K \coloneqq  \mathbb{Q}(\zeta+\zeta^{-1})$. 
Otherwise, if $m = 2p^k$ is twice an odd prime power, the fact that $(L, b)$ is unimodular is equivalent to
\begin{equation}\label{eq: local modular composite}
    \mathfrak{P}^{-\alpha}L^\# = L, 
\end{equation}
where $\mathfrak{P}$ is the unique prime $\mathcal{O}_E$-ideal above $p\mathbb{Z}$ and $\alpha \coloneqq  \val_{\mathfrak{P}}(\mathfrak{D}_{E/\mathbb{Q}}) = p^{k-1}(pk-k-1)$. 
In that case, the isometry class of $L_\mathfrak{q}$ for all finite places $\mathfrak{q}\neq\mathfrak{p} \coloneqq  N^E_K(\mathfrak{P})$ of $K$ is uniquely determined by the rank of $L$. 
Furthermore, according to \Cref{eq: local modular composite}, we have that $\mathfrak{P}^{-\alpha}L^\#_\mathfrak{p} = L_\mathfrak{p}$, meaning that $L_\mathfrak{p}$ is $\mathfrak{P}^{-\alpha}$-modular by definition. Since $\alpha$ is odd, \Cref{propo: orth decomp} and \Cref{lem even trace} tell us that $d = \rank_{\mathcal{O}_E}(L)$ is indeed even and $L_\mathfrak{p}\simeq H(\pi^{-\alpha})^{\oplus d/2}$, where $\pi \coloneqq  1-\zeta$.
Hence, the isometry class of $L_\mathfrak{p}$ is uniquely determined too. Replacing $f$ by another generator of $\langle f\rangle \leq O(L, b)$ does not change 
the isometry class of $L_\mathfrak{p}$ but only the signatures of $(L, h)$. Hence, for fixed signatures, we have that the genus of $(L, h)$ is uniquely determined by the isometry class of $L_\mathfrak{p}$, which itself is determined by the conditions of the statement.
\end{proof}

We now need to settle the prime-power cases. 
For this, we adapt and extend the proof of \cite[Proposition 2.15]{bc22}. 
We recall that if $(L, b, f)$ is an even $\Phi_m$-lattice for some $m\geq 3$, then by transfer, there exists a hermitian $\mathbb{Z}[\zeta_m]$-lattice $(L, h)$ whose trace lattice is $(L, b, f)$, and we have determined a list of local invariants for such a hermitian lattice.
We now find sufficient conditions to prove the converse. 
We treat separately the case where $D_f$ is trivial and the case where $D_f$ has order 2.

\begin{propo}\label{theo: prime order}
Let $p$ be a prime number, let $l_+, l_-\geq0$ be such that $l_++l_-\geq 1$, and let $k\geq \gcd(2,p)$ be positive. 
Suppose that there exist two nonnegative integers $n_0\in2\mathbb{Z}$ and $n_1\in\mathbb{Z}$, with $n_1$ even if\, $p=2$, such that
\begin{enumerate}
    \item\label{t:po-1} $l_++l_- = \varphi(p^k)(n_0+n_1)$, 
    \item\label{t:po-2} $l_- \in 2\mathbb{Z}$, 
    \item\label{t:po-3} $l_+\equiv l_-\mod 8$ if $n_1 = 0$.
\end{enumerate}
Then there exists an even $p$-elementary $\Phi_{p^k}$-lattice $(L, b, f)$ of signatures $(l_+, l_-)$, absolute determinant $p^{n_1}$ and with $\delta_L = 0$ if\, $p=2$, such that $D_f$ is trivial. 
Moreover, up to fixing signatures, the genus of the hermitian structure of any such $(L, b, f)$ is uniquely determined by $(l_-, p, k, n_0, n_1)$.
\end{propo}

\begin{proof}
    Let $\zeta$ be a primitive $\supth{(p^k)}$ root of unity, and let us  again set $E\coloneqq  \mathbb{Q}(\zeta)$ and $K \coloneqq  \mathbb{Q}(\zeta+\zeta^{-1})$. 
    Recall that $\mathfrak{P}$ and $\mathfrak{p}$ are the unique respective prime ideals of $\mathcal{O}_E$ and $\mathcal{O}_K$ lying above $p\mathbb{Z}$. 
    We moreover denote by $\Omega_\infty(K)\subseteq \Omega(K)$ the set of (real) infinite places, respectively the set of all places, of $K$. 
    In \Cref{propo: info hermitian structure}, we have seen that if $(L, b, f)$ is an even $p$-elementary $\Phi_{p^k}$-lattice satisfying the conditions in the statement, then it is the trace lattice of a hermitian $\mathcal{O}_E$-lattice that is locally unimodular at all finite places of $K$ different from $\mathfrak{p}$. 
    Moreover, we know the invariants of its local isometry class at $\mathfrak{p}$. 
    Therefore, we need to prove that conditions (\ref{t:po-1})--(\ref{t:po-3}) are sufficient for the existence of such a hermitian $\mathcal{O}_E$-lattice, and that up to fixing the signatures at the real places of $K$, the genus of this lattice is uniquely determined. 
    
    For $i=0,1$, let $N_i$ be a $\mathfrak{P}^{i-\alpha}$-modular hermitian $\mathcal{O}_{E_\mathfrak{p}}$-lattice of rank $n_i$, where $\alpha \coloneqq  p^{k-1}(pk-k-1)$. 
    Moreover, with respect to \Cref{lem even trace}, we suppose that $N_0\simeq H(\pi^{-\alpha})^{\oplus n_0/2}$, where $\pi \coloneqq  1-\zeta$. 
    Now, the existence of a hermitian $\mathcal{O}_E$-lattice $(L, h)$ with given local structures $\{L_\mathfrak{q}\}_{\mathfrak{q}\in\Omega(K)}$
    is equivalent to the finite set
    \[ S \coloneqq  \left\{\mathfrak{q}\in \Omega(K) \colon \det\left(L_\mathfrak{q}\right)\neq N^{E_\mathfrak{q}}_{K_\mathfrak{q}}\left(E_{\mathfrak{q}}^{\times}\right)\right\}\]
    being of even cardinality; see \cite[Remark 3.4.2(3) and Algorithm 3.5.6]{kir16}. 
    Note that an infinite place $\mathfrak{q}\in \Omega_\infty(K)$ lies in $S$ if and only if $n(\mathfrak{q})$ is odd. 
    Hence, for fixed signatures $\{n(\mathfrak{q})\}_{\mathfrak{q}\in \Omega_\infty(K)}$, there exists a hermitian $\mathcal{O}_E$-lattice $(L, h)$ such that $L_\mathfrak{q}$ is unimodular for all finite places $\mathfrak{q}$ of $K$ different from $\mathfrak{p}$ and such that
    $L_\mathfrak{p} \simeq N_0\oplus N_1$
    if and only if the following equality holds:
    \begin{equation}\label{eq: exist condition proof}
      \sum_{\mathfrak{q}\in\Omega_\infty(K)}n(\mathfrak{q}) \equiv
      \begin{cases}
        \,0&\text{if }\det(N_0)\det(N_1)= N^{E_\mathfrak{p}}_{K_\mathfrak{p}}\left(E_{\mathfrak{p}}^{\times}\right),\\
        \,1&\text{else}
      \end{cases}\;\mod 2.
    \end{equation}
    By condition (\ref{t:po-2}), the integer $l_-$ is even, and according to \Cref{sign trace constr}, the left-hand side $\sum_{\mathfrak{q}\in\Omega_\infty(K)}n(\mathfrak{q})$ must be equal to $l_-/2$. 
    If $n_1$ is nonzero, then \Cref{eq: exist condition proof} and the congruence class $(n_0\mod 4)$ uniquely determine $\det(N_1)\in K^\times/N^{E_\mathfrak{p}}_{K_\mathfrak{p}}(E_{\mathfrak{p}}^{\times})$.
    In particular, there exists a hermitian lattice $(L, h)$ with the local structures as previously fixed. 
    Now, if $n_1 = 0$, we have that $L_\mathfrak{p}\simeq N_0$ is a direct sum of $\frac{n_0}{2}$ copies of $H(\pi^{-\alpha})$ whose determinant is $-N^{E_\mathfrak{p}}_{K_\mathfrak{p}}(E^\times_\mathfrak{p})$ (see \Cref{propo: orth decomp}). 
    Hence $\det(L_\mathfrak{p})=N^{E_\mathfrak{p}}_{K_\mathfrak{p}}(E_{\mathfrak{p}}^{\times})$ if and only if $n_0$ is divisibile by~4 or $-1$ is a local norm at $\mathfrak{p}$.
    By \Cref{th -1 loc norm}, this is equivalent to saying that $\det(L_\mathfrak{p})= N^{E_\mathfrak{p}}_{K_\mathfrak{p}}(E^\times_\mathfrak{p})$ if and only if $n_0\varphi(p^k)$ is divisible by $8$. 
    Thus \Cref{eq: exist condition proof} is equivalent to
    \[2l_- \equiv \varphi\left(p^k\right)n_0\;\mod 8.\]
    But condition (\ref{t:po-1}) enforces that $l_++l_- = \varphi(p^k)n_0$. 
    Thus, in this situation, such an $(L, h)$ exists if and only if
    \[ l_+-
    l_- \equiv 0\;\mod 8.\]
    Therefore, condition (\ref{t:po-3}) ensures that  in the case where $n_1=0$, there also exists a hermitian lattice $(L, h)$ with the local structures as previously fixed. 
    In both of the previous cases, up to fixing the signatures $\{n(\mathfrak{q}))\}_{\mathfrak{q}\in \Omega_\infty(K)}$, the genus of $(L, h)$ is uniquely determined by conditions (\ref{t:po-1})--(\ref{t:po-3}). 
\end{proof}

By using the same kind of arguments, we can also state a similar theorem for nontrivial discriminant actions in the case $p=2$.

\begin{propo}\label{existence even power}
   Let $l_+, l_-\geq0$ be such that $l_++l_-\geq 1$, and let $k\geq 2$ be positive. 
   Suppose that there exist four nonnegative integers $n_0, n_1, n_2$ and $\delta$, with $n_0,n_1\in 2\mathbb{Z}$, $n_2\neq 0$ and $\delta\in \{0,1\}$, such that
\begin{enumerate}
    \item\label{eep-1} $l_++l_- = \varphi(2^k)(n_0+n_1+n_2)$;
    \item\label{eep-2} $l_- \in 2\mathbb{Z}$;
    \item\label{eep-3} if\, $k=2$ and $n_2$ is odd, then $\delta = 1$ and $l_+-l_-\equiv \pm 2\mod 8$;
    \item\label{eep-4} if\, $k = 2$ and $n_2$ is even, then  $l_+-l_-\equiv 0,4\mod 8$ with $l_+-l_-\equiv 0\mod 8$ if $\delta = n_1 = 0$;
    \item\label{eep-5} if\, $k\geq 3$, then $\delta = 0$ and $l_+-l_-\equiv 0,4\mod 8$.
\end{enumerate}
Then there exists an even $2$-elementary $\Phi_{2^k}$-lattice $(L, b, f)$ of signatures $(l_+, l_-)$,  absolute determinant $p^{n_1+2n_2}$ and with $\delta_L = \delta$, such that $D_f$ has order $2$. Moreover, up to fixing signatures, the genus of the hermitian structure of any such $(L, b, f)$ is uniquely determined by $(l_-, k, n_0, n_1, n_2, \delta)$, except in the case where $k\geq 3$, $n_2$ is even, and either $n_1\neq 0$ or $l_+-l_-\equiv 0\mod 8$, where there are two possibilities.
\end{propo}

\begin{proof}
    We fix the same notation as in the proof of \Cref{theo: prime order}.
    Moreover, we let $u\in K_\mathfrak{p}^\times\setminus N^{E_\mathfrak{p}}_{K_\mathfrak{p}}(E_{\mathfrak{p}}^{\times})$, whose existence is ensured by \cite[Proposition 6.1]{joh68}. 
    This scalar $u$ satisfies that $u\in 1+\mathfrak{p}$ and $\mathcal{O}_{K_\mathfrak{p}}^\times = N^{E_\mathfrak{p}}_{K_\mathfrak{p}}(\mathcal{O}_{E_\mathfrak{p}}^\times)\sqcup uN^{E_\mathfrak{p}}_{K_\mathfrak{p}}(\mathcal{O}_{E_\mathfrak{p}}^\times)$ (see for instance \cite[Corollary 3.3.17]{kir16}).
    
    Since we want $D_f$ to have order $2$, we let $N_i$ be $\mathfrak{P}^{i-\alpha}$-modular of rank $n_i$ for $i=0,1,2$, with $n_2\neq 0$. 
    This time, the existence condition \Cref{eq: exist condition proof} can be rewritten as
    \begin{equation}\label{eq: second existence condition proof}
      2l_- \equiv
      \begin{cases}\,0&\text{if }\det(N_0)\det(N_1)\det(N_2)= N^{E_\mathfrak{p}}_{K_\mathfrak{p}}(E_{\mathfrak{p}}^{\times}),\\ \,4&\text{else }\end{cases}
      \;\mod 8.
    \end{equation}
  
    \begin{enumerate}
        \item  We first consider the case where $k=2$ and $u = -1$ is not a local norm at $\mathfrak{p}$ (see \Cref{th -1 loc norm}). 
        Hence, 
        \[\det(N_0)= (-1)^{n_0/2}N^{E_\mathfrak{p}}_{K_\mathfrak{p}}(E^\times_\mathfrak{p})\quad \text{and}\quad \det(N_1) = (-1)^{n_1/2+\epsilon_1}N^{E_\mathfrak{p}}_{K_\mathfrak{p}}(E^\times_\mathfrak{p}),\]
        where $\epsilon_1\in \{0,1\}$ satisfies that $\epsilon_1 = 0$ for $n_1 = 0$ (\Cref{propo: orth decomp}).
    
    \begin{enumerate}
        \item If $n_2$ is odd, then there exists an $\epsilon_2\in \{0,1\}$ such that $\det(N_2)= (-1)^{(n_2-1)/2+\epsilon_2}N^{E_\mathfrak{p}}_{K_\mathfrak{p}}(E^\times_\mathfrak{p})$. The existence condition \Cref{eq: second existence condition proof} can then be formulated as follows:
    \[2l_-\equiv 2(n_0+n_1+n_2-1) + 4(\epsilon_1+\epsilon_2)\mod 8.\]
    Since condition (\ref{eep-1}) in the statement imposes that $l_++l_- = 2(n_0+n_1+n_2)$, the previous equation is equivalent to
    \[l_+-l_-\equiv 2-4(\epsilon_1+\epsilon_2)\equiv \pm 2\mod 8.\]
    Hence condition (\ref{eep-3}) of the statement ensures the existence of $(L, h)$.
    
    Fixing $(l_+-l_-\mod 8)$ actually fixes the parity $\epsilon_1+\epsilon_2$; this gives rise in particular to two possible Jordan decompositions for $L_\mathfrak{p}=N_0\oplus N_1\oplus N_2$. 
    These Jordan decompositions have the same invariants, except for the determinants of $N_1$ and $N_2$, which are both simultaneously changed (so that $\det(L_\mathfrak{p})$ is fixed). 
    According to \cite[Theorem 11.4]{jac62}, these two decompositions define isometric $\mathcal{O}_{E_\mathfrak{p}}$-hermitian lattices if and only if  
    \[u\mathcal{O}_{K_\mathfrak{p}}\subseteq 1+ \mathfrak{n}(N_1)\mathfrak{n}(N_2)\mathfrak{s}(N_1)^{-2}.\]
    By \Cref{propo: orth decomp} and \Cref{rem: cyclo modular bad}, one can check that $\mathfrak{n}(N_1) = \mathfrak{n}(N_2) = \mathfrak{p}^{1-\alpha/2}$, while by the definition of $N_1$, we know that $\mathfrak{s}(N_1)^{-2} = \mathfrak{p}^{\alpha-1}$. 
    Hence
    \[\mathfrak{n}(N_1)\mathfrak{n}(N_2)\mathfrak{s}(N_1)^{-2} = \mathfrak{p}^{1-\alpha/2}\mathfrak{p}^{1-\alpha/2}\mathfrak{p}^{\alpha-1} = \mathfrak{p} \supseteq (u-1)\mathcal{O}_{K_\mathfrak{p}},\]
    and thus the two previous Jordan decompositions define isometric $\mathcal{O}_{E_\mathfrak{p}}$-lattices.
    Therefore, the genus of $(L, h)$ is uniquely determined, up to fixing $\{n(\mathfrak{q}))\}_{\mathfrak{q}\in \Omega_\infty(K)}$. 
    \item If $n_2$ is even, according to \Cref{lem delta 2-elem},  either $\delta = 0$ and $\det(L_2)= (-1)^{n_2/2}N^{E_\mathfrak{p}}_{K_\mathfrak{p}}(E^\times_\mathfrak{p})$, or there exists an $\epsilon_2\in \{0,1\}$ such that $\det(L_2) = (-1)^{n_2/2+\epsilon_2}N^{E_\mathfrak{p}}_{K_\mathfrak{p}}(E^\times_\mathfrak{p})$.
    In this case, \Cref{eq: second existence condition proof} is equivalent to 
    \[l_+-l_-\equiv 4(\epsilon_1+\delta\epsilon_2)\mod 8.\]
    Thus, condition (\ref{eep-4}) of the statement ensures the existence of $(L, h)$. 
    
    Now, if $\delta = 0$, then $\epsilon_1$ is fixed by $(l_+-l_-\mod 8)$, and in particular, the genus of $(L, h)$ is uniquely determined.
    Otherwise, the parity of $\epsilon_1+\epsilon_2$ is determined by $(l_+-l_-\mod 8)$. As in the case $n_2$ odd, since $\mathfrak{n}(N_2) = \mathfrak{p}^{1-\alpha/2}$ for $\delta = 1$, we have that the two possible Jordan decompositions for a fixed parity of $\epsilon_1+\epsilon_2$ give rise to isometric hermitian lattices.
    \end{enumerate}
      
    \item Now, in the case where $k \geq 3$, we have that $-1$ is a local norm at $\mathfrak{p}$, and therefore $\det(N_0) = N^{E_\mathfrak{p}}_{K_\mathfrak{p}}(E^\times_\mathfrak{p})$, and there exists an $\epsilon_1\in \{0,1\}$ such that $\det(N_1)=u^{\epsilon_1}N^{E_\mathfrak{p}}_{K_\mathfrak{p}}(E^\times_\mathfrak{p})$, with $\epsilon_1 = 0$ for $n_1 = 0$. Moreover, there exists an $\epsilon_2\in \{0,1\}$ such that $\det(N_2)= u^{\epsilon_2}N^{E_\mathfrak{p}}_{K_\mathfrak{p}}(E^\times_\mathfrak{p})$, with $\epsilon_2=0$ if $\mathfrak{n}(N_2) = \mathfrak{p}^{2-\alpha/2}$ (recall that $n_2\neq 0$). In this setting, the existence condition \Cref{eq: second existence condition proof} is equivalent to the following:
    \[l_+-l_-\equiv 2^{k-2}(n_0+n_1+n_2)-4(\epsilon_1+\epsilon_2)\equiv 0,4\mod 8.\]
    Hence condition (\ref{eep-5}) ensures the existence of $(L, h)$. 
    As before, fixing $(l_+-l_-\mod 8)$ and the parity of $n_2$ fixes the parity of $\epsilon_1+\epsilon_2$. 
    Furthermore, the following hold:
    \begin{enumerate}
     \item For $n_1=0$ and $n_2$ odd, the congruence class $(l_+- l_-\mod 8)$ uniquely determines $\epsilon_2$, and there is a unique possible genus for $(L, h)$.
     \item For $n_1 = 0$ and $n_2$ even, we have that $l_+-l_-\equiv 4\epsilon_2\mod 8$. 
    If the previous congruence class is $(4\mod8)$, then the genus of $(L, h)$ is again uniquely determined. If it is $(0\mod8)$, then $\epsilon_2=0$ and there are two nonisometric choices for $N_2$ depending on $\mathfrak{n}(N_2)\in \{\mathfrak{p}^{1-\alpha/2}, \mathfrak{p}^{2-\alpha/2}\}$ (see \Cref{rem: cyclo modular bad}), giving rise to two possible genera for $(L, h)$. 
    \item For $n_1\neq 0$,  either $\mathfrak{n}(N_2) = \mathfrak{p}^{2-\alpha/2}$, in which case $\epsilon_2= 0$, $n_2$ is even and the congruence class $(l_+-l_-\mod 8)$ uniquely determines $\epsilon_1$; or the congruence class $(l_+-l_-\mod 8)$ and the parity of $n_2$ determine the parity of $\epsilon_1+\epsilon_2$. 
    Since we now suppose  that $\mathfrak{n}(N_2) \neq \mathfrak{p}^{2-\alpha/2}$, in both the cases $n_2$ even or odd, we have that
    \[\mathfrak{n}(N_1)\mathfrak{n}(N_2)\mathfrak{s}(N_1)^{-2} = \mathfrak{p}.\]
    Hence, by applying again \cite[Theorem 11.4]{jac62}, we know that the parity of $\epsilon_1+\epsilon_2$ determines two possible Jordan decompositions of $L_\mathfrak{p}$ that are isometric.
    \end{enumerate}
    To summarize, when $n_1\neq 0$ and $n_2$ is odd, there is only one possible genus for $(L, h)$, but if $n_2$ is even, the two choices for $\mathfrak{n}(N_2)\in\{\mathfrak{p}^{1-\alpha/2}, \mathfrak{p}^{2-\alpha/2}\}$ give rise to two possible genera for $(L, h)$.\qedhere
    \end{enumerate}
\end{proof}

\begin{rem}\label{rem induced herm action}
    Let $(L, b, f)$ be a $\Phi_{2^k}$-lattice satisfying the assumptions of \Cref{existence even power}, with hermitian structure $(L, h)$. According to the proof of that proposition, since $f$ corresponds to multiplication by~$\zeta$ on the $\mathcal{O}_E$-module $L$, we know that $f$ acts trivially on $\mathfrak{D}^{-1}_{E_\mathfrak{p}/\mathbb{Q}_2}N_1^\#/N_1$. 
    Now, if $n_2 = 1$, we have that as $\mathcal{O}_{E_\mathfrak{p}}$-modules, 
    \[\mathfrak{D}^{-1}_{E_\mathfrak{p}/\mathbb{Q}_2}N_2^\#/N_2\cong \left(\mathcal{O}_{E_\mathfrak{p}}\oplus \zeta\mathcal{O}_{E_\mathfrak{p}}\right)/\left(1-\zeta^2\right)\mathcal{O}_{E_\mathfrak{p}}.\]
    The latter is isomorphic to $\mathbb{Z}/2\mathbb{Z}\oplus \mathbb{Z}/2\mathbb{Z}$ as an abelian group, where an isomorphism is given by $[1]\mapsto (1,0)$ and $[\zeta]\mapsto (0,1)$.
    In particular, since for all $x\in L$, we have that $h(x,x) = h(\zeta x, \zeta x)$, the torsion quadratic form on the copy $\mathbb{Z}/2\mathbb{Z}\oplus \mathbb{Z}/2\mathbb{Z}$ in $D_L$ corresponding to $\mathfrak{D}^{-1}_{E_\mathfrak{p}/\mathbb{Q}_2}N_2^\#/N_2$ takes the same value for $(1, 0)$ and $(0, 1)$. 
    Moreover, multiplication by $\zeta$ corresponds to exchanging those two generators, meaning that $D_f$, which has order 2, is represented by $\begin{psmallmatrix}
        0&1\\1&0
    \end{psmallmatrix}$ on this summand of $D_L$.
    The previous argument can be generalized for every direct summand of $\mathfrak{D}^{-1}_{E_\mathfrak{p}/\mathbb{Q}_2}N_2^\#/N_2$ when $n_2>1$.
\end{rem}

\section{Conjugacy classes of isometries of even unimodular \texorpdfstring{$\Z$}{Z}-lattices}\label{sec: corr th}
In the previous sections, we have set up a framework for studying certain isometries of even unimodular $\Z$-lattices.
Given an isometry $g$ of an even unimodular $\Z$-lattice $M$ with minimal polynomial $\Phi_1\Phi_m$ or $\Phi_1\Phi_2\Phi_m$ for some integer $m$, we have determined necessary and sufficient conditions for the existence of the kernel sublattices of $(M, g)$ (see \Cref{sec: existence isometries}). 
We would now like to find sufficient conditions for the existence of such isometries and to classify them up to conjugacy. We start by determining finer information about how such kernel sublattices glue, in order to derive sufficient conditions on the existence of such even unimodular lattice with isometry $(M, g)$.

\subsection{Existence of isometries of given type}
Let $m\geq 3$, and let $l_+,l_-\geq 0$ be such that $l_++l_-\geq 1$. 
We first cover the existence of even unimodular $\Phi_1\Phi_m^\ast$-lattices of signatures $(l_+,l_-)$.
Note that this is actually already understood; see \cite[Theorem 4.5]{bay24b}.
Nonetheless, we provide alternative statements that include local invariants of the associated invariant and coinvariant sublattices.

In the case where $m$ is composite,  we have already seen in \Cref{propo: case 1 minpoly} that the associated kernel sublattices are unimodular.

\begin{theo}\label{necsuf composite}
    Let $m\geq 3$ be composite. 
    Let $l_+,l_-,s_+, s_-\geq 0$ be such that $l_++l_-> s_++s_-\geq 1$. 
    There exists an even unimodular $\Phi_1\Phi_m^\ast$-lattice $(M, g)$ where $ M\in\II_{(l_+, l_-)}$ and $M_g\in\II_{(s_+, s_-)}$ if and only if there exists an integer $d\geq 1$ such that
    \begin{enumerate}
        \item\label{nc-1} $l_\pm\geq s_\pm$;
        \item\label{nc-2} $l_+-l_-\equiv 0\mod 8$;
        \item\label{nc-3} $s_\pm\equiv 0\mod 2$;
        \item\label{nc-4} $s_++s_- = d\varphi(m)$;
        \item\label{nc-5} $s_+-s_-\equiv 0\mod 8$;
        \item\label{nc-6} if $m=2p^k$ is twice an odd prime power, then $d\equiv 0\mod 2$.
    \end{enumerate}
\end{theo}

\begin{proof}
    This follows from \cite[Theorem 3.5]{bay24b}: In particular, condition (\ref{nc-6}) ensures that $\Phi_m^d(-1)$ is a square, and conditions (\ref{nc-4}) and (\ref{nc-5}) together give that $d\varphi(m) = \deg(\Phi_m^d)\equiv 2s_+\mod 8$. 
    Note that if conditions (\ref{nc-1})--(\ref{nc-6}) hold, then the genera $\II_{(l_+, l_-)}$, $\II_{(s_+, s_-)}$ and $\II_{(l_+-s_+, l_--s_-)}$ are nonempty, and there exists a $\Phi_m$-lattice $(L, f)$ in $\II_{(s_+, s_-)}$ by \Cref{existence unimod}.
    Moreover, for any representative $F$ of an isometry class in $\II_{(l_+-s_+, l_--s_-)}$, we have that $(F, \id_F)\oplus (L, f)$ is an even unimodular $\Phi_1\Phi_m^\ast$-lattice in the genus $\II_{(l_+, l_-)}$ whose coinvariant sublattice is $(L, f)$.
\end{proof}

For the case where $m = p^k$ is a prime power, we follow the same proof as that of \cite[Theorem 1.1]{bc22}.
We recall again that according to \Cref{propo: case 1 minpoly}, the kernel sublattices of $(M, g)$ are $p$-elementary. 

\begin{theo}\label{necsuf prime power}
    Let $m = p^k\geq 3$ for some prime number $p$ and some positive integer $k\geq \gcd(2,p)$.
    Let $l_+, l_-, s_+, s_-, n_1\geq 0$ be such that $l_++l_-\geq s_++s_-+n_1\geq n_1+1$. 
    There exists an even unimodular $\Phi_1\Phi_m^\ast$-lattice $(M, g)$, where $ M\in\II_{(l_+, l_-)}$, such that $M_g$ is $p$-elementary of absolute determinant $p^{n_1}$ and signatures $(s_+, s_-)$ if and only if there exists an even integer $n_0\geq 0$ such that
    \begin{enumerate}
        \item\label{npp-1} $l_\pm \geq s_\pm$,
        \item\label{npp-2} $l_+-l_-\equiv 0\mod 8$,
        \item\label{npp-3} $s_\pm\equiv 0\mod 2$,
        \item\label{npp-4} $s_++s_- = \varphi(p^k)(n_0+n_1)$,
        \item\label{npp-5} $s_+-s_-\equiv 0\mod 8$ if $n_1 = 0$ or $n_1 = (l_+-s_+)+(l_--s_-)$,
        \item\label{npp-6} $n_1$ is even if $p=2$.
    \end{enumerate}
\end{theo}

\begin{proof}
    By \cite[Th\'eor\`eme V.2.2 et Corollaire V.2.1]{ser70}, \cite[Theorem 15.13]{splg} and \Cref{theo: prime order}, we have that if such a pair $(M, g)$ exists, then conditions (\ref{npp-1})--(\ref{npp-6}) hold. 
    Conversely, suppose that conditions (\ref{npp-1})--(\ref{npp-6}) hold. 
    In particular, the genus $\II_{(l_+, l_-)}$ is nonempty, and there exists an even $\Phi_{p^k}$-lattice $(L, f)$ of absolute determinant $p^{n_1}$, signatures $(s_+, s_-)$ and with $\delta_L=0$ if $p=2$, such that $D_f$ is trivial.
    Moreover, by conditions (\ref{npp-1}) and (\ref{npp-5}) and \cite[Theorem 1.10.1]{nik79b}, the genus of even $\mathbb{Z}$-lattices determined by $D_L(-1)$ and $(l_+-s_+, l_--s_-)$ is nonempty.
    Hence for any representative $F$ of an isometry class in such a genus, according to \Cref{eq:egc}, any glue map $\gamma\colon D_F\to D_L$ is $(\id_F, f)$-equivariant, and it gives rise to an equivariant primitive extension $(F, \id_F)\oplus (L, f)\leq (M_\gamma, g)$. 
    By construction, the lattice with isometry $(M_\gamma, g)$ satisfies the first part of the statement.
\end{proof}

\begin{rem}
    Note that as for \Cref{necsuf composite}, parts of the conditions in \Cref{necsuf prime power} can be recovered from \cite[Theorem 3.5]{bay24b}. 
    In particular, since for any power $m$ of 2, we have that $\Phi_m(-1) = 2$, condition (\ref{npp-6}) of \Cref{necsuf prime power} is equivalent to $\Phi_m^{n_0+n_1}(-1) = 2^{n_0+n_1}$ being a square. 
\end{rem}

We now aim to prove similar results in the case where $m = 2p^k$ is twice a nontrivial prime power for some prime number $p$, the isometry $g\in O(M)$ has minimal polynomial $\Phi_1\Phi_2\Phi_m$ and the kernel sublattice $M^{-g}$ is isometric to $\langle2p\rangle$ (see \Cref{nontriv: all}). 
For this, since there are three potential kernel sublattices, we cannot invoke directly Nikulin's results about equivariant primitive extensions. 
However, we show, by investigating further how these  $\Z$-lattices can glue, that there are few possibilities for the local invariants of $M^g$ and $M^{\Phi_m(g)}$, and thus one can apply Nikulin's methods iteratively.%\smallskip

Let us first assume that $p$ is odd. We prove the following lemma.

\begin{lem}\label{lem: 3 poly odd p}
    Let $p$ be an odd prime number, and let $m = 2p^k$ for some $k\geq 1$. 
    Let $(M,g)$ be an even unimodular $\Phi_1\Phi_2\Phi_m^\ast$-lattice such that $K\coloneqq M^{-g}\simeq \langle2p\rangle$. 
    Then there exist three positive even integers $s_+, s_-, n_0$ satisfying 
    \begin{enumerate}
        \item $s_++s_-\leq \rank_\mathbb{Z}(M)-2$,
        \item\label {l:3pop-2} $s_++s_- = \varphi(p^k)(n_0+1)$,
        \item $s_+-s_-\equiv 2\left(\frac{-2}{p}\right)-1-p\mod 8$
    \end{enumerate}
    and such that the $\Phi_m$-kernel sublattice $(L, f)$ of\, $(M, g)$ lies in the genus $\II_{(s_+, s_-)}p^{\left(\frac{-2}{p}\right)}$ with $D_f = -\id_{D_L}$.
\end{lem}

\begin{proof}
    As abelian groups, $D_K\cong \mathbb{Z}/2\mathbb{Z}\oplus \mathbb{Z}/p\mathbb{Z}$, and the torsion quadratic form on $D_K$ is $\begin{psmallmatrix} p/2&0\\0&2/p \end{psmallmatrix}$.
    Note that $-\id_K$ induces the identity on the $2$-primary part of $D_K$, and $-\id$ on the $p$-primary part.
    Moreover, following \cite[Chapter 15]{splg}, we have that $K$ lies in the genus $\II_{(1,0)}2^{\epsilon}_\delta p^{\epsilon}$, where $\delta \in\{1,3,5,7\}$ is congruent to $p$ modulo 8 and $\epsilon = \left(\frac{2}{p}\right)$.
    Let $(L, f) \coloneqq  (M^{\Phi_m(g)}, g_m)$ be the $\Phi_m$-kernel sublattice of $(M, g)$.
    By the proof of \Cref{propo: case 2 minpoly} and \Cref{eq:egc}, we have that $L$ has absolute determinant $p$, the discriminant group $D_L$ of $L$ is equipped with the discriminant form $(-2/p)$, and $D_f = -\id_{D_L}$.
    Let us denote by $(s_+, s_-)$ the signatures of $L$; note that since $g$ admits at least one fixed vector in $M$ and $M^{-g}$ has rank 1, we must have $s_++s_-\leq \rank_{\mathbb{Z}}(M)-2$. 
    Moreover, $L$ lies in the genus $\II_{(s_+, s_-)}p^{\left(\frac{-1}{p}\right)\epsilon}$. 
    Finally, since $s_++s_-\geq \varphi(p^k)>1$, the previous genus is nonempty if and only if $s_+-s_-\equiv 2\left(\frac{-2}{p}\right)-1-p\mod 8$ (see \cite[Theorem 15.13]{splg}). 
    Note that (\ref{l:3pop-2}) follows from \Cref{propo: info hermitian structure}.
\end{proof}

We keep the notation of \Cref{lem: 3 poly odd p} and its proof, and let us  further set $F \coloneqq  M^g$. 
If we denote by $(l_+, l_-)$ the signatures of $M$, then $i_+ \coloneqq  l_+-s_+-1$ and $i_- \coloneqq  l_--s_-$ define the signatures of $F$. 
By similar arguments as in the proof of \Cref{lem: 3 poly odd p}, we have that $D_F$ is equipped with the torsion quadratic form $(-p/2)$. We then deduce that $F$ lies in the genus $\II_{(i_+, i_-)}2^\epsilon_{-\delta}$, where we recall that $\delta\in\{1,3,5,7\}$ is congruent to $p$ modulo 8 and $\epsilon = \left(\frac{2}{p}\right)$. 
Now if $i_++i_->1$, then the genus $\II_{(i_+, i_-)}2^\epsilon_{-\delta}$ is nonempty if and only if $s_--s_+-1\equiv i_+-i_-\equiv  -p+2-2\left(\frac{2}{p}\right)\mod 8$ (see \cite[Theorem 15.13]{splg}).
Put together, we have that the genera of $F$ and $L$ are simultaneously nonempty if and only if
\[ 2\left(\frac{-2}{p}\right)-1-p\equiv 2\left(\frac{2}{p}\right)-3+p\mod 8,\]
and one can easily check that the latter is always true (it follows from the fact that for any odd prime number $p$, $\left(\frac{-1}{p}\right)\equiv p\mod 4$). 
In the case where $F$ has rank 1, either $F\in \II_{(1,0)}2^{+1}_{1} = \II_{(1,0)}2^{-1}_5$, or $F\in \II_{(0,1)}2^{+1}_7 = \II_{(0,1)}2^{-1}_3$.
We must have that $p\equiv 3,7\mod 8$ in the former case, and $p\equiv 1,5\mod 8$ in the latter case. 
By straightforward computations, one can show that the existence conditions for $F$ and $(L, f)$ are still equivalent in both cases. 
Therefore, if $p$ has the appropriate residue class modulo 8 when $l_+-s_+ + l_--s_- = 2$, the existence of $(L, f)$ ensures the existence of $F$.

\begin{theo}\label{gen stat 3 odd}
    Let $m = 2p^k\geq 3$ for some odd prime number $p$ and some nonnegative integer $k\geq 1$. 
    Let $l_+, l_-, s_+, s_-\geq 0$ be such that $l_+\geq 1$ and $l_++l_-\geq s_++s_- +2 \geq 3$. 
    There exists an even unimodular $\Phi_1\Phi_2\Phi_m^\ast$-lattice $(M, g)$, where $ M\in\II_{(l_+, l_-)}$, such that $M^{-g}\simeq \langle 2p\rangle$ and $M^{\Phi_m(g)}$ is $p$-elementary of absolute determinant $p$ and signatures $(s_+, s_-)$ if and only if there exists an even integer $n_0\geq 0$ such that
    \begin{enumerate}
        \item\label{gs3o-1} $l_\pm \geq s_\pm$;
        \item\label{gs3o-2} $l_+-l_-\equiv 0\mod 8$;
        \item\label{gs3o-3} $s_\pm\equiv 0\mod 2$;
        \item\label{gs3o-4} $s_++s_- = \varphi(p^k)(n_0+1)$;
        \item\label{gs3o-5} $s_+-s_-\equiv 2\left(\frac{-2}{p}\right)-1-p\mod 8$;
        \item\label{gs3o-6} $p\equiv 1\mod 4$ if $(l_+-s_+,\; l_--s_-) = (1, 1)$, and $p\equiv 3\mod 4$ if $(l_+-s_+,\; l_--s_-) = (2, 0)$.
    \end{enumerate}
\end{theo}

\begin{proof}
    Let $(M, g)$ be as in the first part of the statement. 
    Then, according to \cite[Th\'eor\`eme~V.2.2 et Corollaire V.2.1]{ser70}, condition (\ref{gs3o-2}) holds. 
    Now, \cite[Theorem 15.13]{splg} and \Cref{theo: prime order}, together with \Cref{rem: do not do twice prime} for the case where $p$ is odd, give us that conditions (\ref{gs3o-3}) and (\ref{gs3o-4}) hold too. 
    Conditions (\ref{gs3o-1}), (\ref{gs3o-5}) and (\ref{gs3o-6}) follow directly from \Cref{lem: 3 poly odd p} and the follow-up discussions.
    Now suppose that conditions (\ref{gs3o-1})--(\ref{gs3o-6}) hold. 
    In particular, the genus $\II_{(l_+, l_-)}$ is nonempty, and there exists a $\Phi_{2p^k}$-lattice $(L, f)$ of absolute determinant~$p$ and signatures $(s_+, s_-)$ such that $D_f = -\id_{D_L}$.
    By conditions (\ref{gs3o-4}) and (\ref{gs3o-5}), the torsion quadratic form on $D_L$ is $(-2/p)$. 
    Moreover, by conditions (\ref{gs3o-5}) and (\ref{gs3o-6}), there exists an even  $\Z$-lattice $F$ with signatures $(l_+-s_+-1, l_--s_-)$ and absolute determinant 2, such that the torsion quadratic form on $D_F$ is $(-p/2)$. 
    Hence, if we let $K\coloneqq  \langle 2p\rangle$, since $D_{F\oplus L} = D_F\oplus D_L\simeq D_K(-1)$, we obtain that $F\oplus K\oplus L$ has an overlattice $M$ in $\II_{(l_+, l_-)}$. 
    Moreover, by \Cref{eq:egc}, since $\id_F\oplus f$ and $-\id_K$ agree along any glue map $D_{F\oplus L}\to D_K$, we have that $\id_f\oplus (-\id_K)\oplus f\in O(F\oplus K\oplus L)$ extends to an isometry $g\in O(M)$ with minimal polynomial $\Phi_1\Phi_2\Phi_m$, such that $M^{-g} = K$ and $(M_m, g_m) = (L, f)$.
\end{proof}

Let us conclude with the case of powers of 2.
Let $m = 2^k$ for some $k\geq 2$, and let $(M, g)$ be an even unimodular $\Phi_1\Phi_2\Phi_{2^k}^\ast$-lattice such that $K\coloneqq M^{-g}\simeq  \langle 4\rangle$.
Let us  moreover denote by $(l_+, l_-)$ the signatures of $M$, and let $F \coloneqq  M^g$, of signatures $(i_+, i_-)$. 
The $\Phi_m$-kernel sublattice $(L, f)$ of $(M, g)$ has signatures $(s_+, s_-) \coloneqq  (l_+-i_+-1, l_--i_-)$. %\smallskip

By \Cref{existence even power}, there exist three integers $n_0, n_1,n_2\geq 0$ with $n_0, n_1\in2\mathbb{Z}$ and $n_2\neq 0$ such that $\rank_\mathbb{Z}(L) = \varphi(2^k)(n_0+n_1+n_2)$. 
According to \Cref{propo: case 2 minpoly}, we  moreover know that $L$ is 2-elementary, the $\Z$-lattice $F$ is 4-elementary, and $F$ and $K$ glue along elementary abelian 2-groups $H_F\leq D_F$ and $H_K\leq D_K$. 
Remark that as abelian groups, $D_K\cong \mathbb{Z}/4\mathbb{Z}$, and $D_K$ is equipped with the torsion quadratic form $(1/4)$. 
Hence, $H_K\cong \mathbb{Z}/2\mathbb{Z}$ and the torsion quadratic form on $H_K$ is $(1)$. 
By \cite[Proposition 4.10]{bh23}, the glue map $H_F\to H_K$ maps  $2D_F$ isomorphically to $2D_K=H_K$. 
Therefore, there exists an $n\geq 0$ such that as abelian groups $D_F\cong  D_{F, 4}\oplus D_{F, 2}$, where $D_{F, 4} \cong \mathbb{Z}/4\mathbb{Z}$ and $D_{F,2}\cong (\mathbb{Z}/2\mathbb{Z})^{\oplus n}$, and $H_F\leq D_{F, 4}$ is the subgroup generated by twice a generator. 
By the classification of torsion quadratic forms on the abelian group $\mathbb{Z}/4\mathbb{Z}$ (see \cite[Proposition 1.8.1]{nik79b}), there exists an $\alpha\in\{1,3,5,7\}$ such that the torsion quadratic form on $D_{F, 4}\leq D_F$ is $(\alpha/4)$.
Now, let us set $N \coloneqq  L^\perp_M$, which is a primitive extension of $F\oplus K$ in $M$.

\begin{claim}\label{claim even form}
We have that  $D_N\cong (\mathbb{Z}/2\mathbb{Z})^{\oplus 2}\oplus D_{F,2}$, where the torsion quadratic form on the first two summands is
    \begin{equation}\label{form gamma}
    \scriptstyle{\begin{pmatrix}
        (\alpha+1)/4&(\alpha+3)/4\\(\alpha+3)/4&(\alpha+1)/4
    \end{pmatrix}}\,.
    \end{equation}
    Moreover, $\id_F\oplus (-\id_K)$ extends to an isometry $h$ of\, $N$ such that $D_h$ acts by $\begin{psmallmatrix}0&1\\1&0\end{psmallmatrix}$ on the first two summands of\, $D_N$ and by the identity on $D_{F,2}$. 
    In particular, $(n_1, n_2) = (n, 1)$, and the torsion quadratic form on $D_{F,2}$ takes integer values.
\end{claim}

\begin{proof}[Proof of \Cref{claim even form}]
    Let $\gamma\colon H_F\to H_K$ be the glue map of the primitive extension $F\oplus K\hookrightarrow N$, and let $\Gamma\leq D_{F,4}\oplus D_K$ be the graph of $\gamma$.
    According to \cite[Proposition 1.5.1]{nik79b}, as abelian groups we have that $D_N\cong \Gamma^\perp/\Gamma\oplus D_{F,2}$.
    Now, by straightforward computations, one may show that $\Gamma^\perp/\Gamma\cong (\mathbb{Z}/2\mathbb{Z})^{\oplus 2}$ is generated by the classes $(1,1)+\Gamma$ and $(1,3)+\Gamma$. 
    Moreover, since the torsion quadratic form on $D_{F,4}\oplus D_K$ is $\begin{psmallmatrix}\alpha/4&0\\0&1/4\end{psmallmatrix}$, it follows that the torsion quadratic form on $\Gamma^\perp/\Gamma$ is as stated.
    We remark that 
    $\id_F\oplus (-\id_K)$ extends to an isometry $h\in O(N)$ such that $D_h$ acts by exchanging the two generators on $\Gamma^\perp/\Gamma$ and $D_h$ is trivial on $D_{F, 2}$. 
    For the final statements, since $N$ and $L$ glue equivariantly along their respective discriminant groups, we have that $n_1=n$ and $n_2 = 1$ by \Cref{eq:egc} and \Cref{rem induced herm action}. The fact that the torsion quadratic form on $D_{F,2}$ takes integer values follows from the proofs of \Cref{lem even trace,lem delta 2-elem}.
\end{proof}

Let $\epsilon_1,\epsilon_2\in \{\pm 1\}$ be such that $F\in \II_{(i_+, i_-)}2_{\II}^{\epsilon_1n_1}4^{\epsilon_2}_\alpha$. 
By \Cref{propo: existence 24 elem}, we know that $n_1$ must indeed be even, $\epsilon_2 = (\frac{\alpha}{2})$ and $\epsilon_1 = 1$ if $n_1 = 0$. 
Moreover, we have that $i_++i_-\leq n_1+1$, with equality only if $\epsilon_1 = \epsilon_2$, and finally that 
\begin{equation}\label{eq: another useful eq}
    i_+-i_-\equiv \alpha+2-2\epsilon_1\mod 8.
\end{equation}
We aim to show that the genus of $F$ uniquely determines that of $L$ (similarly to what has been done in \Cref{gen stat 3 odd}).
We explain our argument for the case $\alpha = 1$; the other cases follow similarly.

If $\alpha = 1$, then the torsion quadratic form \Cref{form gamma} is given by $\begin{psmallmatrix}1/2&0\\0&1/2\end{psmallmatrix}$. Hence, according to \cite[Theorem 15.13]{splg} and \Cref{eq: another useful eq}, we have that
\(N\in\II_{(i_++1, i_-)}2^{\epsilon_1(n_1+2)}_2.\)
Since $N\oplus L$ admits an even unimodular primitive extension, we have that $L\in \II_{(s_+, s_-)}2^{\epsilon_L(n_1+2)}_6$, where $\epsilon_L\in \{\pm 1\}$ satisfies that $s_+-s_-\equiv -2\epsilon_L\mod 8$ and $\epsilon_L = 1$ for $s_++s_- = n_1+2$ (see \cite[Theorem 15.13]{splg}). Since $(s_+, s_-) = (l_+-i_+-1, l_--i_-)$, \Cref{eq: another useful eq} tells us, moreover, that $\epsilon_L = \epsilon_1$.  Analogously, one can check that for all $\alpha\in\{1,3,5,7\}$, the genus of $F$ determines those of $N$ and $L$, which are given in \Cref{fig:genera}.

{\setlength{\tabcolsep}{5pt}
\renewcommand\arraystretch{1.5}
\begin{table}[!ht]

    \caption{Genera of $F$, $N$ and $L$ depending on $\alpha\in\{1,3,5,7\}$}\label{fig:genera}
    \centering
    \begin{tabular}{ccccc}
         $\alpha$&\Cref{form gamma}& $g(F)$&$g(N)$ & $g(L)$ \\
          \hline
         \cellcolor{lightgray!40!white}1&  \cellcolor{lightgray!40!white}$\scriptscriptstyle{\begin{pmatrix}1/2&0\\0&1/2\end{pmatrix}}$&  \cellcolor{lightgray!40!white}$\II_{(i_+, i_-)}2_\II^{\epsilon_1n_1}4^1_1$& \cellcolor{lightgray!40!white}$\II_{(i_++1, i_-)}2^{\epsilon_1(n_1+2)}_2$& \cellcolor{lightgray!40!white} $\II_{(s_+, s_-)}2^{\epsilon_1(n_1+2)}_6$\\
         3& $\scriptscriptstyle{\begin{pmatrix}1&1/2\\1/2&1\end{pmatrix}}$& $\II_{(i_+, i_-)}2_\II^{\epsilon_1n_1}4^{-1}_3$&$\II_{(i_++1, i_-)}2^{-\epsilon_1(n_1+2)}_\II$& $\II_{(s_+, s_-)}2^{-\epsilon_1(n_1+2)}_\II$\\
         \cellcolor{lightgray!40!white}5& \cellcolor{lightgray!40!white}$\scriptscriptstyle{\begin{pmatrix}3/2&0\\0&3/2\end{pmatrix}}$&  \cellcolor{lightgray!40!white}$\II_{(i_+, i_-)}2_\II^{\epsilon_1n_1}4^{-1}_5$& \cellcolor{lightgray!40!white}$\II_{(i_++1, i_-)}2^{\epsilon_1(n_1+2)}_6$& \cellcolor{lightgray!40!white}$\II_{(s_+, s_-)}2^{\epsilon_1(n_1+2)}_2$\\
         7&$\scriptscriptstyle{\begin{pmatrix}0&1/2\\1/2&0\end{pmatrix}}$& $\II_{(i_+, i_-)}2_\II^{\epsilon_1n_1}4^1_7$&$\II_{(i_++1, i_-)}2^{\epsilon_1(n_1+2)}_\II$&$\II_{(s_+, s_-)}2^{\epsilon_1(n_1+2)}_\II$\\
         \hline
    \end{tabular}
\end{table}
}

Since the genus of $L$ determines that of $N$, it is not hard to see that the genus of $L$ actually also determines that of $F$, from the description given in \Cref{claim even form}.

\begin{rem}
    Note that if $i_++i_- = n_1+1$, then $\epsilon_1 = 1$ by \cite[Theorem 15.13]{splg}. In this case, \Cref{propo: existence 24 elem} gives that $\epsilon_2 = \epsilon_1 = 1$. Since $\epsilon_2 = \left(\frac{\alpha}{2}\right)$, we therefore know that $i_++i_- = n_1+1$ only if $\alpha\in\{1,7\}$.
    Moreover, from the symbol of the genus of $L$, we know that $\delta_L=1$ for $\alpha\in\{1,5\}$ and $\delta_L = 0$ otherwise. 
    Therefore, \Cref{existence even power} tells us that $\alpha \equiv 1\mod 4$ if and only if $m=4$.
\end{rem}

\begin{theo}\label{gen stat 3 even}
    Let $m = 2^k\geq 3$ for some nonnegative integer $k\geq 2$. 
    Let $l_+, l_-, s_+, s_-, n_1\geq 0$ be such that $l_+\geq 1$, $n_1\in 2\mathbb{Z}$ and $l_++l_-\geq s_++s_-+n_1+2\geq n_1+3$. There exists an even unimodular $\Phi_1\Phi_2\Phi_m^\ast$-lattice $(M, g)$, where $ M\in\II_{(l_+, l_-)}$, such that $M^{-g}\simeq \langle 4\rangle$ and $M^{\Phi_m(g)}$ is $2$-elementary of absolute determinant $2^{n_1+2}$ and signatures $(s_+, s_-)$ if and only if there exist an even integer $n_0\geq 0$ and two integers $\alpha\in \{1,3,5,7\}$ and $\epsilon\in\{\pm 1\}$ such that
    \begin{enumerate}
        \item $l_\pm \geq s_\pm$;
        \item $l_+-l_-\equiv 0\mod 8$;
        \item $s_\pm\equiv 0\mod 2$;
        \item $s_++s_- = \varphi(2^k)(n_0+n_1+1)$;
        \item $s_+-s_-\equiv 1-\alpha-2\epsilon\mod 8$;
        \item\label{gs3e-6} $\alpha\equiv 1\mod 4$ if and only if $m = 4$;
        \item if\, $(l_+-s_+)+(l_--s_-) = n_1+2$, then $\alpha\equiv \pm 1\mod 8$ and $\epsilon = +1$;
        \item\label{gs3e-8} if\, $s_++s_- = n_1+2$, then $\epsilon = +1$ and $\alpha\in \{1,5\}$.
    \end{enumerate}
\end{theo}

\begin{proof}
    The proof is similar to that of \Cref{gen stat 3 odd} by using \Cref{claim even form} and the follow-up discussions. 
    Note that $s_++s_- = n_1+2$ can happen if and only if $k=2$ and $n_0=n_1 = 0$.
    In that situation, we know that we must have $\epsilon=+1$ by \Cref{propo: existence 24 elem} and $\alpha\in \{1,5\}$ by condition (\ref{gs3e-6}) of the statement.
    Hence condition (\ref{gs3e-8}) is necessary.
\end{proof}

\subsection{Classification up to conjugacy}
Let $m\geq 3$, and let $\zeta \coloneqq  \zeta_m$ be a primitive $\supth{m}$ root of unity. 
We first prove the following proposition, which is a generalization of \cite[Proposition 2.9]{bc22}.

\begin{propo}\label{propo class gen herm}
    Let $E \coloneqq  \mathbb{Q}(\zeta)$ and $(L, h)$ be a hermitian $\mathcal{O}_E$-lattice. 
    Suppose that $(L, h)$ is indefinite or that $\rank_{\mathcal{O}_E}(L) = 1$. 
    Then the number of isometry classes in the genus of\, $(L, h)$ is the relative class number $h^{-}(E)$.
\end{propo}

\begin{proof}
    The case where $\rank_{\mathcal{O}_E}(L) = 1$ has already been proven in more generality in \cite[Proposition~2.9]{bc22}. 
    The indefinite case has been proven for $m = p$ an odd prime number in the same proposition. 
    To generalize to any $m\geq 3$, one can follow the proof of the aforementioned proposition and remark that $E/K$ has either zero or a ramified prime ideal, depending on whether $m$ is composite or (twice) a prime power respectively (see \Cref{lem basic cyclo}).
    In the latter case, one needs to replace $(-1 \cdot \mathcal{E}_1^\mathfrak{p})$ by $(-\zeta\cdot \mathcal{E}_1^{\mathfrak{p}})$ in the proof of \cite[Proposition 2.9]{bc22} to conclude.\qedhere
\end{proof}

According to \cite[Section 3, Tables]{was97}, the relative class number in the cyclotomic case $E \coloneqq  \mathbb{Q}(\zeta_m)$ is 1 for all $3\leq m\leq 66$ with $\varphi(m)\leq 22$, except for $m = 23, 46$, where $h^{-}(E) = 3$. 
For our classification of nonsymplectic Hodge monodromies of IHS manifolds, the hermitian lattices we are interested in correspond to the hermitian structure on a $\Phi_m$-lattice that itself is the coinvariant sublattice of the aforementioned nonsymplectic monodromies.
In particular, such a $\mathbb{Z}$-lattice is indefinite or of rank $\varphi(m)$, implying that the associated hermitian structure is itself indefinite or of rank 1.
Let us prove the following results; we split the cases between even unimodular $\Phi_1\Phi_m$-lattices and even unimodular $\Phi_1\Phi_2\Phi_m$-lattices.

\begin{theo}\label{class 2 div}
    Let $M$ be an even indefinite unimodular lattice, let $g\in O(M)$ be an isometry of finite order $m\geq 3$, and let $E \coloneqq  \mathbb{Q}(\zeta_m)$. 
    Suppose that the minimal polynomial of $g$ divides $\Phi_1\Phi_m$ and that $M_g$ is indefinite or of rank $\varphi(m)$. Then
    \begin{enumerate}
        \item the isometry class of the invariant sublattice $M^g$ and
        \item the isometry class of the hermitian structure of\, $(M_g, g_m)$
    \end{enumerate}
    form a complete set of invariants for the isomorphism class of\, $(M, g)$. 
    In particular, the number of isomorphism classes in the type of\, $(M, g)$ is given by $c\cdot h^{-}(E)$, where $c$ is the number of isometry classes in the genus of\, $M^g$.
\end{theo}

\begin{proof}
    The proof is similar to that of \cite[Theorem 1.2]{bc22}.
    Let us assume that we are given two isometries $g,h\in O(M)$ such that $M^{g}\simeq M^{h}$ and that $(M_g, g_m)$ and $(M_h, h_m)$ are isomorphic. 
    Now, $g$ and $h$ are conjugate in $O(M)$ if and only if $(M, g)$ and $(M, h)$ are isomorphic, which holds if and only if the equivariant primitive extensions
    \begin{equation}\label{eq proof eq glue}
      \left(M^g, \id\right)\oplus \left(M_g, g_m\right)\leq (M,g)\quad \text{and}\quad \left(M^h, \id\right)\oplus \left(M_h, h_m\right)\leq (M,h)\end{equation}
    are isomorphic.
    If we denote by $\gamma_g\colon D_{M^g}\to D_{M_g}$ and $\gamma_h\colon D_{M^h}\to D_{M_h}$ the respective equivariant glue maps, \cite[Corollary 1.5.2]{nik79b} tells us that the primitive extensions in \Cref{eq proof eq glue} are isomorphic if and only if there exist an isometry $\psi_1\colon M^g\to M^h$ and an isomorphism $\psi_2\colon (M_g, g_m)\to (M_h, h_m)$ such that $\gamma_h\circ\overline{\psi_1} = \overline{\psi_2}\circ \gamma_g$. Let us fix an isometry $\psi_1\colon M^g\to M^h$ and an isomorphism $\psi_2\colon (M_g, g_m)\to (M_h, h_m)$, and  let $\kappa \coloneqq  \gamma_h\circ\overline{\psi_1}\circ\gamma_g^{-1}\circ\overline{\psi_2}^{-1}\in O(D_{M_h})$.
    If $m$ is not a prime power, then we know from \Cref{propo: case 1 minpoly} that $M_h$ is unimodular, meaning that $\kappa = \id_{D_{M_h}}$ and $g,h$ are conjugate in $O(M)$.
    Otherwise, if $m = p^k$ is a prime power, the same proposition tells us that $M_h$ is an even $p$-elementary  $\Phi_{p^k}$-lattice with $D_h$ is trivial. 
    Therefore, by \Cref{strong approx th}, there exists a $\psi\in O(M_h, h_m)$ such that $\kappa = D_\psi$. 
    In particular, up to replacing $\psi_2$ by $\psi\circ\psi_2\colon (M_g, g_m)\to (M_h, h_m)$, we have that the primitive extensions in \Cref{eq proof eq glue} are isomorphic, and hence $g,h\in O(M)$ are conjugate. 
    The last assertion of the statement follows from \Cref{propo class gen herm}.
\end{proof}

\begin{theo}\label{class 3 div}
    Let $M$ be an even indefinite unimodular lattice, let $g\in O(M)$ be an isometry of even order $m = 2p^k$, where $p$ is a prime number and $k\geq 1$, and let $E = \mathbb{Q}(\zeta_m)$.
    Suppose that the minimal polynomial of $g$ is $\Phi_1\Phi_2\Phi_m$, that the kernel sublattice satisfies $M^{-g}\simeq \langle 2p\rangle$, and that $M^{\Phi_m(g)}$ is indefinite or of rank $\varphi(m)$. Then
    \begin{enumerate}
        \item the isometry class of the invariant sublattice $M^g$ and
        \item the isometry class of the hermitian structure of $(M^{\Phi_m(g)}, g_m)$
    \end{enumerate}
    form a complete set of invariants for the isomorphism class of\, $(M, g)$.
    In particular, the number of isomorphism classes in the type of\, $(M, g)$ is given by $c\cdot h^{-}(E)$, where $c$ is the number of isometry classes in the genus of\, $M^g$.
\end{theo}

\begin{proof}
This time, by \Cref{propo: case 2 minpoly}, we have that $M^g$ and $M^{-g}$ glue along abelian groups of order 2, which therefore have no nontrivial isomorphism.
For two isometries $g,h\in O(M)$ such that $M^g\simeq M^h$ and $M^{-g}\simeq M^{-h}\simeq \langle 2p\rangle$, we actually obtain that $(M^{g^2-1}, g_{\mid M^{g^2-1}})$ and $(M^{h^2-1}, h_{\mid M^{h^2-1}})$ are isomorphic too. 
If $p$ is odd, we have that $-D_{g_m} $ is the identity, the centralizers $O(M^{\Phi_m(g)}, g_m)$ and $O(M^{\Phi_m(g)}, -g_m)$ coincide, and $O(D_{M^{\Phi_m(g)}}, D_{g_m}) = O(D_{M^{\Phi_m(g)}})$. 
Otherwise, if $p = 2$, we know that $(M^{\Phi_m(g)}, g_m)$ is an even $2$-elementary $\Phi_{2^{k}}$-lattice with $D_{g_m}$ of order at most 2. 
Hence, in both cases, by \Cref{strong approx th}, we know that $O(M^{\Phi_m(g)}, g_m)\to O(D_{M^{\Phi_m(g)}}, D_{g_m})$ is surjective, and we conclude by applying similar arguments to those in the proof of \Cref{class 2 div}.
Note that we have already seen that in such a situation, the genus of $M^{\Phi_m(g)}$ together with $M^{-g}\simeq  \langle 2p\rangle$ uniquely determine the genus of $M^g$; hence the last statement also holds.
\end{proof}

These theorems give us an algorithm to compute representatives of isomorphism classes of such isometries. 
Up to fixing the signatures, one starts by determining the possible genera for the hermitian structure of $(M^{\Phi_m(g)}, g_m)$ using  \Cref{existence unimod,theo: prime order,existence even power}. 
For each such genus, one constructs the trace lattice associated to each representative of an isometry class in this genus. 
Each of the previous genera actually determines the genus $G$ of the candidates for $M^g$. 
We then conclude using \Cref{class 2 div,class 3 div} that one has to enumerate~$G$ and compute the corresponding equivariant primitive extensions. 
In particular, if $h^{-}(\mathbb{Q}(\zeta_m)) = 1$ and~$G$ consists of a unique isometry class, one obtains a unique pair $(M, g)$ for a given set of signatures of $(M^{\Phi_m(g)}, g_m)$ (see \Cref{rem perm}).

\subsection{From isometries of even unimodular \texorpdfstring{$\Z$}{Z}-lattices to monodromies}
Let $\T$ be a known deformation type of IHS manifolds, and let $\Lambda, M, V, v$ and $K$ be as given in \Cref{tab:extdata}. 
We recall that there is a group homomorphism $\gamma\colon\Mon^2(\Lambda)\to S(M, V, v),\; h\mapsto \hat{\chi}(h)\oplus h$ whose image is $G \coloneqq  \ker(\vartheta\cdot\chi_V\colon S(M, V, v)\to \{\pm1\})$. 

\begin{rem}\label{rem signatures extension}
    Eventually, we will be interested in the case where $h\in \Mon^2(\Lambda)$ is effective and nonsymplectic. In particular, $\Lambda_h$ has signatures $(2, \ast)$. Since we assume $m>2$, we have that $\Lambda_h$ is also the $\Phi_m$-kernel sublattice of $(M, \gamma(h))$. Moreover, either $\Lambda_h$ is  indefinite,  or $m=3,4,6$ and $\rank(\Lambda_h) = \varphi(m)=2$.
\end{rem}

For a given group $H$ and two elements $g,h\in H$, we denote by $\phantom{}^hg \coloneqq  hgh^{-1}$ the conjugation of $g$ by $h$, and $\phantom{}^Hg$ denotes the conjugacy class of $g$ in $H$. 
We let $\operatorname{cl}(H)$ denote the set of conjugacy classes in $H$ (see \cite[Section~4.2]{bc22}). 

\begin{theo}\label{th: corr1}
    Let $\psi\colon \operatorname{cl}\left(G\right)\to \operatorname{cl}(S(M))$ be the natural map. 
    Let $g\in G$ be of order $m\geq 3$, where $m_g(X) = \Phi_1(X)\Phi_m(X)\in\mathbb{Q}[X]$, the coinvariant sublattice $M_g$ is of signatures $(2, \ast)$ and $V\subseteq M^g$. Then, the map
    \begin{align*}
         \phi\colon\psi^{-1}\left(\phantom{}^{S(M)}g\right) \ &\lra\ S\left(M^g\right)\backslash \left\{ W\subseteq M^g \text{ primitive }  W\simeq V\right\}\\
         \phantom{}^{G}h\ &\longmapsto\ S(M^g)fV \quad\text{where } g ={}^fh,\, f\in S(M) 
    \end{align*}
 is a bijection.
\end{theo}

\begin{proof}
    The proof is analogous to that of \cite[Theorem 4.7]{bc22}. For the reader's convenience, we briefly explain how to adapt the proof in \textit{loc.\ cit.}~to this more general setting; we refer to the proofs of \cite[Theorem 4.7]{bc22} and \Cref{th: corr2} for further details.%\smallskip

    Let $h\in G$, and let $f\in S(M)$ be such that $g = {}^fh$. Since $g\in G$ is not necessarily of prime order, we must ensure that $fV\leq M^g$ in order to show that $\phi$ is well defined; the rest of the proof of that part is the same as that of Brandhorst--Cattaneo. The isometry $g = fhf^{-1}$ preserves $fV$ and fixes $fv$ because $h\in G\leq S(M, V, v)$. In particular, $g$ preserves the orthogonal $K'$ of $fv$ in $fV$, which has dimension 1. Hence it must send a generator $k$ of $K'$ to $\pm k$; but since $M^{-g}$ is trivial ($m\geq3$), we have that $gk=k$ and therefore, $fV\leq M^g$.

    The proof of the surjectivity of $\phi$ is identical to that from  \cite[Theorem 4.7]{bc22} since we assume that $g\in G$.
    For the proof of the injectivity of $\phi$, we follow again verbatim the proof in \textit{loc.\ cit.}; this time, we use \Cref{strong approx th}, which is a generalization of \cite[Lemma 2.21]{bc22}, to ensure that the natural map
    \[\SO(M_g, g) = O\left(M_g, g\right)\lra O\left(D_{M_g}\right)\]
    is surjective.
    \qedhere
\end{proof}

\begin{rem}
    For $g\in G$ as in the statement of \Cref{th: corr1}, since $V\subseteq M^g$, we have that $\chi_V(g) = 1$ (see \Cref{lem: character for K}). 
    In particular, the fact that $g$ is in $G$ is equivalent to $\vartheta(g) = 1$. 
    In the statement and the proof of \cite[Theorem 4.7]{bc22}, the authors dropped the condition for $g$ to be in $G$ since odd-order isometries have positive real spinor norm.
\end{rem}

Note that \Cref{th: corr1} is actually constructive: The isometry $g$ has a well-defined restriction to $W^\perp_{M}$ for all $W\subseteq M^g$ isometric to $V$.
The corresponding sublattice with isometry is isomorphic to $(\Lambda, h)$, where $\gamma(h)$ and $g$ are $S(M)$-conjugate in $G$.%\medskip

Now let $h\in \Mon^2(\Lambda)$ be nonstable of even order $m\geq 3$ and such that $m_h = \Phi_1\Phi_m$. 
The minimal polynomial of $g\coloneqq \gamma(h)$ is $\Phi_1\Phi_2\Phi_m$, and according to \Cref{rem: K compl}, we have a succession of primitive sublattices $K= M^{-g}\leq V\leq M^{g^2-1}$. We can adapt the statement of \Cref{th: corr1} to this case.
\begin{theo}\label{th: corr2}
    Let $\psi\colon \operatorname{cl}\left(G\right)\to \operatorname{cl}(S(M))$ be the natural map. 
    Let $g\in G$ be of even order $m\geq 3$, where $m_g(X) = \Phi_1(X)\Phi_2(X)\Phi_m(X)\in\mathbb{Q}[X]$, the kernel sublattice $M^{\Phi_m(g)}$ is of signatures $(2, \ast)$, and $K = M^{-g}\leq V\leq M^{g^2-1}$.
    Then, the map
    \begin{align*}
        \phi \colon \psi^{-1}\left(^{S(M)}g\right) \ &\lra \ S\left(M^{g^2-1}, M^{-g}\right)\backslash \left\{ W\;\colon M^{-g}\leq W\leq M^{g^2-1} \text{ primitive}, \; W\simeq V\right\}\\
         \phantom{}^{G}h\ &\longmapsto\  S\left(M^{g^2-1}, M^{-g}\right)fV\,,\quad \text{where } g ={}^fh,\, f\in S(M) 
    \end{align*}
 is a bijection.
\end{theo}

\begin{proof}
    We follow the same steps as in the proof of \cite[Theorem 4.7]{bc22}. Let us set $\Gamma \coloneqq  S(M^{g^2-1}, M^{-g})$.

    \begin{enumerate}
        \item \emph{We start by proving that $\phi$ is well defined.} 
    Let $h\in G$ and $f\in S(M)$ be such that $g = fhf^{-1}$. Since $g$ and $h$ are conjugate, their kernel sublattices are isometric via $f$, and their spinor norms are equal. The former implies that $M^{-h} = f^{-1}(M^{-g})$ has rank 1, and the latter implies that $\chi_V(h) = \chi_V(g)$ by the assumption $g,h\in G$. Since $K= M^{-g}$, \Cref{lem: character for K} tells us that $\chi_V(h) = \chi_V(g)=-1$, meaning that $M^{-g} = K\leq M^{-h}$ too. Therefore, since $\rank(M^{-g}) = \rank(M^{-h})$, we get that $M^{-h} = M^{-g} (= K)$. Now, we also have that $gfv = fhv = fv$ and $gfV = fhV = fV$ since $h\in G\leq S(M, V, v)$. In particular, $fv\in M^g$, and we have already seen that $fM^{-g}= M^{-g}$. Hence, using the fact that $V$ is an overlattice of $\mathbb{Z}v\oplus M^{-g}$ in $M^{g^2-1}$, we obtain $M^{-g}\leq fV\leq M^{g^2-1}$.
    
    Now let $h_1,h_2,s\in G$ be such that $h_2 =\phantom{}^sh_1$. 
    Let moreover $f_1,f_2\in S(M)$ be such that $g = \phantom{}^{f_1}h_1 = \phantom{}^{f_2}h_2$.
    We need to show that $\Gamma f_1V = \Gamma f_2V$.
    For that, let $t \coloneqq  f_2sf_1^{-1}\in S(M)$; by straightforward computations, it follows that $gt=tg$ holds and, in particular, $t$ preserves the kernel sublattices of~$g$.
    According to \cite[Lemma 2.21]{bc22}, since $O(M_m, g_m) = \SO(M_m, g_m)$ and $t_{\mid M_m}$ commutes with $g_m$, we have that $\det(t) = \det(t_{\mid M_m})\det(t_{\mid M^{g^2-1}})=\det(t_{\mid M^{g^2-1}})$ and therefore $t_{\mid M^{g^2-1}}\in \Gamma$.
    Moreover, $tf_1V = f_2sV = f_2V$ since $s$ preserves $V$. 
    Hence, $\Gamma f_1V = \Gamma f_2V$, and $\phi$ is well defined.

    \item \emph{We now want to prove the surjectivity of $\phi$.} Let $M^{-g}\leq W\leq M^{g^2-1}$ be a succession of primitive sublattices, and suppose that $W\simeq V$. By \cite[Proposition 1.6.1]{nik79b} and \Cref{tab:extdata}, the primitive embedding of $K = M^{-g}$ into $V$ is unique up to isometry. Hence, one can extend the identity on $M^{-g}$ to an isometry $\overline{f}\colon V \xrightarrow{\lowsimeq}W$. Similarly to \cite[Theorem 4.7]{bc22}, one can actually show that $\overline{f}$ extends to an isometry $f\in S(M)$ such that $W = fV$ and $f_{\mid M^{-g}} = \id_{ M^{-g}}$. Let $k\in M$ be such that $M^{-g} = \mathbb{Z}k$. Note that $f$ maps $v\in V\cap M^g$ to a generator $w$ of $(M^{-g})^\perp_W\leq M^g$. 
    If we let $h \coloneqq  f^{-1}gf$, we have that 
    \(hk = f^{-1}gfk = f^{-1}gk = -f^{-1}k = -k\) because $\chi_V(g) = -1$ and $f$ is the identity on $M^{-g}$. Moreover, $hv = v$. Hence $h$ restricts to $h_V\in O(V)$, meaning that $h\in S(M, V, v)$.
    Moreover, $(\vartheta\cdot \chi_V)(h) = (\vartheta\cdot \chi_V)(g) = 1$. 
    Therefore, $h\in G$, and $h$ satisfies $\phi(\phantom{}^Gh) = \Gamma fV = \Gamma W$.

    \item \emph{Finally, we prove the injectivity of $\phi$.}
    In order to do so, we let $h_1, h_2\in G$ and $f_1, f_2\in S(M)$ be such that $g = \phantom{}^{f_1}h_1 = \phantom{}^{f_2}h_2$, and we suppose that there exists a $t\in \Gamma=S(M^{g^2-1}, M^{-g})$ such that $tf_1V = f_2V$. 
    We aim to show that $h_2\in\phantom{}^Gh_1$.

    Recall from the proof of the well-definedness of $\phi$ that $g,h_1, h_2\in G$ with $g = \phantom{}^{f_1}h_1 = \phantom{}^{f_2}h_2$ implies that $f_1M^{-g} = M^{-g} = f_2M^{-g}$. In particular, since $tf_1V = f_2V$ and $tf_1M^{-g} = f_2M^{-g}$, we have that $tf_1v = \pm f_2v$.
    By a similar argument to one in the proof of \cite[Theorem 4.7]{bc22}, we can actually assume that $tf_1v = f_2v$, up to composing $t$ with an appropriate element of the joint stabilizer $S(M^{g^2-1}, M^{-g}, V)$. 
    According to \Cref{propo: case 2 minpoly}, we have that $M_m$ is unimodular or $p$-elementary for some prime number $p$, and $D_{g_m}$ has order at most $2$. Hence, according to \Cref{rem signatures extension}, \Cref{strong approx th} and \cite[Lemma 2.21]{bc22}, the map $\SO(M_m, g_m) = O(M_m, g_m) \to O(D_{M_m}, D_{g_m})$ is surjective.
    Since $M$ is unimodular, $M^{g^2-1}$ and $M_m$ glue along their respective discriminant groups, and for any given glue map $\eta\colon D_{M^{g^2-1}}\to D_{M_m}$, there exists a $t'\in \SO(M_m, g_m)$ such that $\eta\circ  D_t\circ \eta^{-1} = D_{t'}\in O(D_{M_m})$.
    Hence, by \Cref{eq:egc}, $t\oplus t'\in S(M^{g^2-1}\oplus M_m)$ extends to an isometry $\widetilde{t}\in S(M)$ commuting with $g$. 
    Therefore, $u \coloneqq  f_2^{-1}\circ \widetilde{t}\circ f_1\in S(M, V, v)$ and $h_2 = \phantom{}^uh_1$. 
    Finally, we conclude by remarking that if $(\vartheta\cdot \chi_V)(u)\neq 1$, we can always compose $u$ by $\widehat{\chi}(-\id_\Lambda)\oplus (-\id_\Lambda)$ to ensure that $u\in G$, where $\widehat{\chi}(-\id_\Lambda) = \id_V$ for $\T = \textnormal{OG6}$ and $\widehat{\chi}(-\id_\Lambda) =h_V$ otherwise. \qedhere
    \end{enumerate}
\end{proof}

For the rest of this section, let $g$ satisfy the assumptions of \Cref{th: corr2}. We would like to describe a constructive way to obtain representatives for the classes in $\psi^{-1}(\phantom{}^{S(M)}g)$ using $\phi$.

Let us first suppose that $\T = \textnormal{K3}^{[n]}, \textnormal{Kum}_n$; according to \Cref{tab:extdata}, then $V$ has rank 1. Thus, there exists a succession of primitive sublattices
\[M^{-g}\leq W\leq M^{g^2-1}\]
with $W\simeq V$ if and only if $W = M^{-g}$. In particular, the following holds.

\begin{coro}\label{rem after corr 2}
    Suppose that $\T = \textnormal{K3}^{[n]}, \textnormal{Kum}_n$, and let $m\geq 3$ even. The set of\, $\Mon^2(\Lambda)$-conjugacy classes of nonstable isometries $h\in \Mon^2(\Lambda)$ with minimal polynomial $\Phi_1\Phi_m$ such that $\Lambda_h$ has signatures $(2, \ast)$ is in bijection with the set of\, $S(M)$-conjugacy classes of isometries $g\in G$ of minimal polynomial $\Phi_1\Phi_2\Phi_m$ such that $M^{\Phi_m(g)}$ has signatures $(2, \ast)$ and $M^{-g}\simeq V$. To any such isometry $g\in G$ corresponds the restriction $h$ of $g$ to $(M^{-g})^\perp_M\simeq  \Lambda$.
\end{coro}

\begin{proof}
    For $\T = \textnormal{K3}^{[n]}, \textnormal{Kum}_n$, \Cref{tab:extdata} tells us that $v = 0$ and in particular $K= V$. Hence, for each ${g\in G}$ satisfying the assumptions of \Cref{th: corr2}, we have that $V \simeq M^{-g}$ and the codomain of $\phi$ has car\-di\-na\-lity~1. 
\end{proof}

Now suppose that $\T = \textnormal{OG6}$. In this case, $M^{-g} \simeq \langle 4\rangle$ (see \Cref{tab:extdata}). If we let $k\in M$ be such that $M^{-g} = \mathbb{Z}k$, we know that $V = \mathbb{Z}v+\mathbb{Z}\frac{v+k}{2}$ is a primitive sublattice of $M^{g^2-1}$. In fact, if $h\in\Mon^2(\Lambda)$ is so that $g = \gamma(h)$, by construction and by the fact that $\Lambda_h = M^{\Phi_m(g)}$, we know that $M^{g^2-1}$ is a primitive extension of $\Lambda^h\oplus V$. But now, $\frac{v+k}{2}\notin M^g\oplus M^{-g}$ and $M^{g^2-1}/(M^g\oplus M^{-g})$ has order 2. In particular, since $D_{M^{-g}}\cong \Z/4\Z$ as abelian groups, we infer that
\begin{equation}\label{eq glue map 1}
    M^{g^2-1} = (M^{g}\oplus M^{-g}) + \mathbb{Z}\frac{v+k}{2}.
\end{equation}
Note that $v/2\in (M^g)^\vee$ and $v^2=4$, so  the divisibility $d$ of $v$ in $M^g$ is either 2 or 4.

Similarly, suppose that $\T = \textnormal{OG10}$. This time, $M^{-g} \simeq \langle 6\rangle$ (see \Cref{tab:extdata}), and the proof of \Cref{gen stat 3 odd} tells us that $M^{g^2-1}/(M^g\oplus M^{-g})$ has order 2. Similarly to before, we deduce that
\begin{equation}\label{eq glue map 2}
    M^{g^2-1} = (M^{g}\oplus M^{-g}) + \mathbb{Z}\frac{v+k}{2}.
\end{equation}
This time, we observe that $v^2 = \div_{M^g}(v) = 2$.

\begin{lem}\label{eq: og6}
    Let $L$ be an even lattice, and let $v\in L$ be a primitive vector. We set $I \coloneqq  v^\perp_{L}$, and we let $d$ be the divisibility of $v$ in $L$. Then, the index of\, $I\oplus \mathbb{Z}v$ in $L$ is $v^2/d$.
\end{lem}

\begin{proof}
    We have a succession of inclusions
    \[\Z v\oplus I\leq L\leq L^\vee\leq (\Z v)^\vee\oplus I^\vee.\]
    Let us denote by $\pi\colon L\to (\Z v)^\vee$ the first projection. We know that $[L\colon \Z v\oplus I]$ is equal to the order $h$ of the glue domains for the primitive extension $\Z v\oplus I$. Such finite abelian groups are isomorphic to $\pi(L)/\Z v = (\Z (v/h))/\Z v$. Moreover, we observe
    \[  d\Z = v\cdot L = v\cdot\pi(L) = v\cdot(\Z (v/h)) = \left(v^2/h\right)\Z.\]
    Hence $h = v^2/d$.
\end{proof}

\begin{rem}\label{eq: og10}\leavevmode
\begin{enumerate}
    \item For $\T=\textnormal{OG6}$, according to the discussion prior to \Cref{claim even form}, we have that $D_{M^g} = D_4\oplus D_2$, where, as abelian groups, $D_4\cong \Z/4\Z$ and $D_2\cong (\Z/2\Z)^{\oplus n}$ for some $n\geq 0$. 
    \begin{enumerate}
        \item If $v$ has divisibility 4 in $M^g$, \Cref{eq: og6} tells us that $M^g = v^\perp \oplus \mathbb{Z}v$ and $v^\perp$ is 2-elementary. Indeed, without loss of generality, we may assume that $v/4+M^g$ generates $D_4$, in which case $D_{v^\perp}\simeq D_2$.
        \item If $v$ has divisibility 2 in $M^g$, we know that $2M^g\leq v^\perp\oplus \Z v$, and $\Z v$ and $v^\perp$ glue along elementary abelian 2-groups. Since $D_{\Z v}\cong \Z/4$ as abelian groups (see \Cref{tab:extdata}), we know that $\Z v$ and $v^\perp$ glue along order 2 subgroups of their respective discriminant groups.
    \end{enumerate}
    \item For $\T=\textnormal{OG10}$, according to \Cref{gen stat 3 odd}, we have that $D_{M^g} \cong \mathbb{Z}/2\mathbb{Z}$ as abelian groups. Since $v^2=2$ and $\div_{M^g}(v) = 2$, \Cref{eq: og6} tells us that $M^g = v^\perp \oplus \mathbb{Z}v$ and, in particular, $v^\perp$ is even unimodular.
\end{enumerate}
\end{rem}

It is hard in general to compute representatives for the cosets in the codomain of $\phi$ (see \Cref{th: corr2}). In the next proposition, we would like to use the description given in \Cref{eq: og10} in order to find an alternative way to describe such cosets, and thus use $\phi$ in a more explicit way.

\begin{propo}\label{cosets 1}
    Suppose that $\T = \textnormal{OG6},\,\textnormal{OG10}$, and let $g\in G$ satisfy the assumptions of \Cref{th: corr2}. We denote by $d\geq 2$ the divisibility of $v$ in $M^g$. Then the sets of cosets 
    \[C_1 \coloneqq  O\left(M^{g^2-1}, M^{-g}\right)\backslash \left\{ W\;\colon M^{-g}\leq W\leq M^{g^2-1} \text{ primitive}, \; W\simeq V\right\}\]
    and
    \[
    C_2\coloneqq O\left(M^g, \frac{v}{2}+M^g\right)\backslash\left\{\Z w\;\colon w\in M^g\text{ primitive},\;w^2 = v^2,\; \div(w, M^g) = d,\; w+2M^g = v+2M^g \right\}
    \]
    are in bijection.
\end{propo}

\begin{proof}
    Let $I \coloneqq  v^\perp_{M^g}$, let again $k\in M$ be such that $M^{-g} = \mathbb{Z}k$, and let us define
    \[\kappa\colon C_1\lra C_2,\quad O\left(M^{g^2-1}, M^{-g}\right)\cdot W\longmapsto O\left(M^g, \frac{v}{2}+M^g\right)\cdot (M^{-g})^\perp_W.\]
    Let us first remark the following: Since the glue map of $M^g\oplus M^{-g}\leq M^{g^2-1}$ is given by 
    \[ v/2 + M^g\longmapsto k/2+M^{-g}\]
    (see \Cref{eq glue map 1,eq glue map 2}), we see that any isometry in $O(M^{g^2-1}, M^{-g})$ restricts to an isometry of $M^g$ preserving $v/2+M^g$ (see \Cref{eq:egc}), and \textit{vice versa}, any isometry in $O(M^g, v/2+M^g)$ can be extended to an isometry in $O(M^{g^2-1}, M^{-g})$. More precisely, restriction to $M^g$ induces a surjective group homomorphism
    \begin{equation}\label{equation of pi}
      \pi\colon O\left(M^{g^2-1}, M^{-g}\right)\xrightarrow{f \mapsto f_{\mid M^g}} O(M^g, v/2+M^g)
    \end{equation}
    whose kernel is $\langle\id_{M^g}\oplus (-\id_{M^{-g}})\rangle\cong O(M^{-g})$.
    \begin{enumerate}
        \item \emph{Let us prove that $\kappa$ is well defined.} Let $M^{-g}\leq W\leq M^{g^2-1}$ be a succession of primitive sublattices with $W\simeq V$. We moreover let $w\coloneqq  (M^{-g})^\perp_W\leq M^g$ and $J\coloneqq  w^\perp_{M^g}$. Following the proof of the surjectivity in \Cref{th: corr2}, one can find an isometry $\overline{f}\colon V\to W$ preserving $M^{-g}$. In particular, $\overline{f}$ restricts to an isometry $\widetilde{f}\colon \Z v\to \Z w$, and 
        \[v^2 = w^2.\]
        According to \Cref{eq: og10}, we know that either $I$ and $\Z v$ are in orthogonal direct sum in $M^g$, or they glue along order 2 subgroups of their respective discriminant groups, which therefore have no nontrivial automorphisms, and similarly for $J$ and $\Z w$. In particular, in both cases, we know that we can extend $\widetilde{f}$ to an isometry $\widehat{f}\in O(M^g)$ such that $\widehat{f}(v) = w$ and $\widehat{f}(I) = J$. Since $M^g$ and $M^{-g}$ also glue along subgroups of order 2 (see prior discussions), we  moreover have that $f\coloneqq \widehat{f}\oplus \overline{f}_{\mid {M^{-g}}}$ defines an isometry of $M^{g^2-1}$ and that it is an extension of $\overline{f}$ to $M^{g^2-1}$ satisfying that $f(I) = J$. As a consequence, since $f$ preserves $M^g$ and $f(v) = w$, we obtain that 
        \[\div(w, M^g) = \div(v, M^g).\]
        Now recall that $M^{g^2-1}/(M^g\oplus M^{-g})$ has order 2, generated by $\frac{v+k}{2} + (M^g\oplus M^{-g})$. Since $f$ preserves $M^g\oplus M^{-g}$, we know that this quotient is also generated by 
    \[f\left(\frac{v+k}{2} + (M^g\oplus M^{-g})\right) = \frac{w+k}{2} + (M^g\oplus M^{-g}).\]
    This implies in particular that $\frac{v-w}{2}\in (M^g\oplus M^{-g})\cap (M^g)^\vee = M^g$ and 
    \[v-w\in 2M^g.\] 
    Hence $\kappa$ is well defined.
    \item \emph{We now  show  that $\kappa$ is surjective. }Let $w\in M^g$ be a primitive vector so that $w^2 = v^2$, the divisibility satisfies $\div(w, M^g) = d$, and $w+2M^g = v+2M^g$. We define $W \coloneqq  \mathbb{Z}w+\mathbb{Z}\frac{w+k}{2}\leq (M^g\oplus M^{-g})^\vee$; it is a primitive extension of $\Z w\oplus M^{-g}$, and we observe that $W$ is a primitive sublattice of $M^{g^2-1} = (M^g\oplus M^{-g})+ \mathbb{Z}\frac{w+k}{2}$. Moreover, we have that $W\simeq V$. Hence, $\kappa$ is surjective.
    \end{enumerate}
    The injectivity of $\kappa$ follows from the surjectivity of $\pi$ (see \Cref{equation of pi}).
\end{proof}

Following \Cref{cosets 1}, we see how to make \Cref{th: corr2} constructive even in the OG6 and OG10 cases.
Indeed, the problem reduces to finding orbits of primitive vectors with given norm and divisibility that generate the glue domain of $M^g$ for the primitive extension $M^g\oplus M^{-g}\leq M^{g^2-1}$. 
For each such orbit of vectors $O\left(M^g, \frac{v}{2}+M^g\right)\cdot\mathbb{Z}w$, one reconstructs a representative for the corresponding monodromy conjugacy class by restricting $g$ to the orthogonal complement of  $\mathbb{Z}w\oplus M^{-g}$ in $M$.

\section{Some classification results}\label{sec: geom}
In this section, we apply the lattice-theoretic results obtained so far to prove the following main statements.

\begin{theo}\label{mainth1}
    Let $X$ be a projective IHS manifold of known deformation type $\T$, and let $f\in\Aut(X)$ be nonsymplectic and algebraically trivial. Let $M \coloneqq  M_\T$ be the corresponding Mukai lattice. Suppose $h \coloneqq  \rho_X(f)\in \Mon^2(X)$ is stable of composite order $m\geq 3$, and let $g \coloneqq  h\oplus \id_{\Lambda^\perp_M}\in O(M)$. Then the $O(M)$-conjugacy class of $\langle g\rangle$ is uniquely determined by the order of $g$, the genus of $M^g$ and the genus $M_g$.
\end{theo}

\begin{theo}\label{mainth2}
    Let $X$ be a projective IHS manifold of known deformation type $\T$, let $m\geq 2$ be a composite integer, and let $\ker\rho_X\lneq G\leq\Aut(X)$ be algebraically trivial with $[G:\ker\rho_X]=m$. Let us set $\Lambda\coloneqq  H^2(X, \Z)$ and suppose that the cyclic group $H \coloneqq  \rho_X(G)$ acts nontrivially on the discriminant group of $\Lambda$. Then, up to deformation and birational conjugacy, the pair $(X, G)$ is uniquely determined by $(\T, m, \Lambda^H, \Lambda_H)$, except in the case $(\T, m) = (\textnormal{K3}^{[24]}, 46)$, where there are three such pairs. The corresponding tuples $(\T, m, \Lambda^H, \Lambda_H)$ are given in \Cref{disc non triv}.
\end{theo}

We later make some comments about the geometric aspect of the classification and give some explicit examples.
One can combine the results from this section with \Cref{main thhhh} to conclude about the geometry  of certain moduli spaces of symmetric IHS manifolds.
Eventually, we aim to show the reader how to use in practice the methods described in this paper and in the reference paper \cite{bc22}.

\subsection{About induced actions}\label{subsec induced descr}
For every known deformation type $\T$, one can construct an example of an IHS manifold $X\sim \T$ by considering a certain moduli space of stable sheaves on some projective $K$-trivial surface (see \cite[Theorems 1.2 and 1.3]{pr13} and the reference therein). 
Through these constructions, there exists a natural way of defining \emph{induced actions}. 
Let us set $M_{24}\coloneqq  U^{\oplus 4}\oplus E_8^{\oplus 2}$ and $M_8\coloneqq  U^{\oplus4}$.

\begin{defin}[\textit{cf.} {\cite[Definition 4.1]{mw15}}]\label{defin induced act}
    Let $n\geq 2$, let $X\sim \textnormal{K3}^{[n]}$, and let $G\leq \Aut(X)$. The group $G$ is said to be an \emph{induced group of automorphisms} if there exist a projective K3 surface $S$ with $G\hookrightarrow \Aut(S)$, a $G$-invariant Mukai vector $v\in H^\ast(S,\mathbb{Z})\simeq M_{24}$ and a $v$-generic stability condition $\tau$ such that $X$ is isomorphic  to the moduli space $M_\tau(v)$ of $\tau$-stable sheaves on $S$ and the action of $G$ induced on $M_\tau(v)$ coincides with that of $G$ on $X$. A similar definition hold for $X\sim \textnormal{Kum}_n$ and $S$ an abelian surface (where this time $H^\ast(S,\mathbb{Z})\simeq M_8$).
\end{defin}

Note that given a projective K3 or abelian surface $S$, one can endow $H^{2\ast}(S, \mathbb{Z})$ with a weight 2 Hodge structure such that there is an integral Hodge isometry
\[ H^2\left(M_\tau(v), \mathbb{Z}\right)\lra v^\perp\subseteq H^{2\ast}(S, \mathbb{Z}),\]
    with the notation of \Cref{defin induced act}. This Hodge isometry is equivariant with respect to the induced actions, and in particular, any induced action must restrict to the identity on the discriminant group of $H^2(M_\tau(v), \mathbb{Z})$ since the chosen Mukai vector $v$ is invariant in $H^{2\ast}(S, \mathbb{Z})$ (see \cite[Section~2]{mw15}). Moreover, the following holds. 

\begin{propo}[\textit{cf.} {\cite[Theorems 4.4 and 4.5]{mw15}, \cite[Proposition 4.8]{bc22}}]\label{propo induced}
Let $n\geq 2$, let $\T \coloneqq  \textnormal{K3}^{[n]}$ or $\textnormal{Kum}_n$, and let $g\in O(M_\T)$ satisfy the assumptions of \Cref{th: corr1}. Then the monodromy classes in the fiber $\psi^{-1}(\phantom{}^{S(M_\T)}g)$ $($see the notation of \Cref{th: corr1}\,$)$ admit a geometric realization as actions of induced automorphisms $($in the sense of \Cref{defin induced act}\,$)$ if and only if $M_\T^g$ contains a copy of\, $U$.
\end{propo}

\begin{rem}
    There is a generalization of this criterion where one considers an embedding of a twisted copy $U(k)$ of $U$ in $M^g_\T$ instead, for some $k\geq 2$. These correspond to induced actions on moduli spaces of $k$-twisted sheaves (see \cite[Section~3]{ckkm19}).
\end{rem}

\begin{rem}
    One can extend \Cref{defin induced act} and \Cref{propo induced} for detecting induced actions on IHS manifolds of type $\textnormal{OG6}$ and $\textnormal{OG10}$ (see \cite[Section~3.2]{gro22b} for the OG6 case and \cite[Section~5]{mw15} for the OG10 case). In both cases, the induced groups of Hodge monodromies  again have trivial actions on the discriminant groups of the associated BBF forms.
\end{rem}

Besides the previous notion of induced automorphisms, there are other ways of constructing geometric examples of birational automorphisms of IHS manifolds via given constructions (see for instance \cite[Section~6]{bea83b}, \cite[Section~4]{og06}, \cite[Section~4]{ow13},  \cite[Section~4.1]{cc19}, \cite[Section~3]{ikkr19}, \cite[Section~4]{gro22b} or \cite[Section~3.1]{sac23}).
For most of these actions, though, we do not have numerical criteria to decide whether certain lattice data correspond to such actions; this highlights the fact that it is hard in general to translate geometric information to transcendental (lattice-theoretic) data, and \textit{vice versa}. 
Moreover, for most of the cases, the induced action on the discriminant groups of the associated BBF forms is trivial.
Nonetheless, there are known geometric examples of birational automorphisms that act nontrivially on the discriminant groups and for which it is not known whether they can be realized as induced in any meaningful way (see for instance \cite[Section~4]{fer12}, \cite[Corollary 5.11]{mw15} or \cite[Theorems 1.1 and 1.2]{ccl22}).%\smallskip

In the next parts, we give complete lattice-theoretic classifications related to the results that were established in this paper. Whenever it makes sense, we mention whether the actions described are possibly induced, as explained in this section.

\subsection{Trivial discriminant action}\label{triv disc sec}
Using \Cref{necsuf composite}, \Cref{necsuf prime power}, \Cref{propo class gen herm} and \Cref{class 2 div},  we can establish a list of genera for the invariant and coinvariant sublattices associated to representatives of conjugacy classes of isometries $g$ of any finite order $m\geq 3$, with minimal polynomial $\Phi_1\Phi_m$ and coinvariant sublattice of signatures $(2, \ast)$, on the even unimodular $\Z$-lattices $M_\T$ for each known deformation type $\T$ of IHS manifolds. Before describing the tables of results, we remark the following.

\begin{lem}\label{bad order all actions}
    Let $\T$ be one of the known deformation types. There are no IHS manifolds $X\sim\T$ admitting a nonsymplectic birational automorphism whose action on $H^2(X, \mathbb{Z})$ has  minimal polynomial $\Phi_1\Phi_m$ with $m\in\{15, 20, 24, 30, 40, 48, 60\}$.
\end{lem}

\begin{proof}
    Let $X\sim\T$, and let us suppose that there exists a  nonsymplectic $f\in \Bir(X)$ such that $h \coloneqq  \rho_X(f)$ has order $m\in\{15, 20, 24, 30, 40, 48, 60\}$, and with minimal polynomial $\Phi_1\Phi_m$. Since $f$ is nonsymplectic, we know that $H^2(X, \Z)^{\Phi_m(h)}$ has signatures $(2, k)$ for some $k\geq 0$, so  $\varphi(m)$ divides $k+2$. Let $g\coloneqq \gamma(h)\in O(M_\T)$ be defined as in \Cref{lem: well defined extension}, where we replace $\Lambda_\T$ by $H^2(X, \Z)\simeq \Lambda_\T$. In particular, $(M_\T, g)$ is either a $\Phi_1\Phi_m^\ast$-lattice or a $\Phi_1\Phi_2\Phi_m^\ast$-lattice, and $M_\T^{\Phi_m(g)} = H^2(X, \Z)^{\Phi_m(h)}$. According to \Cref{propo: case 1 minpoly}(\ref{p:c1m-3}) and \Cref{propo: case 2 minpoly}(\ref{p:c2m-3}), since $m$ is neither a prime power nor twice a prime power, we have that $M_\T^{\Phi_m(g)}$ is even unimodular, of signatures $(2, k)$. Hence, \cite[Th\'eor\`eme V.2.2 et Corollaire V.2.1]{ser70} tells us in particular that $2-k$ is divisible by 8, and so is $2+k$ since $\varphi(m)\equiv 0\mod8$ by assumption. But this is absurd, since the two latter conditions would imply that 4 is divisible by 8. Hence such a birational automorphism $f\in \Bir(X)$ cannot exist.
\end{proof}

\begin{lem}\label{bad order stable actions}
    Let $\T$ be one of the known deformation types. There are no IHS manifolds $X\sim\T$ admitting a nonsymplectic birational automorphism whose action on $H^2(X, \mathbb{Z})$ is stable of minimal polynomial $\Phi_1\Phi_m$ with $m\in\{10, 26, 34, 38, 46, 50,54\}$.
\end{lem}

\begin{proof}
    We suppose existence, and we follow the same notation as in the proof of \Cref{bad order all actions}. This time, the isometry $g\in O(M_\T)$ has minimal polynomial $\Phi_1\Phi_m$ by the stability assumption, and its coinvariant sublattice $M_\T^{\Phi_m(g)}$ is still unimodular, of signatures $(2,k)$ for some $k\geq 0$. Now, note that for any known deformation type $\T$, the Mukai lattice $M_{\T}$ has negative signature strictly less than 22 (see \Cref{tab:extdata}). Therefore, if $m = 26, 34, 38, 46, 50$ or $54$ is twice a prime power with $\varphi(m)\geq 12$, then  by \Cref{existence unimod}(\ref{eu-3}) we would need that the rank of the previously mentioned coinvariant sublattice is a nontrivial multiple of $2\varphi(m)\geq 24$, meaning that $k\geq 22$, giving rise to a contradiction. Similarly, if $m=10$, we would have that $2+k$ is divisible by $2\varphi(10) = 8$; by similar arguments to those in the proof of \Cref{bad order all actions}, we conclude that this is absurd.
\end{proof}

\begin{rem}\label{rem bad orders}
    Following the arguments of the proofs of \Cref{bad order all actions,bad order stable actions}, it follows that if $X\sim \text{Kum}_n$ for some $n\geq 2$ or $X\sim \text{OG6}$, then $X$ does not admit any nonsymplectic birational automorphism $f$ such that $\rho_X(f)$ has order $m \in \{8, 14,16, 18\}$ and minimal polynomial $\Phi_1\Phi_m$. Moreover, as a consequence of \Cref{theo: prime order}, there does not exist any IHS manifold $X$ of known deformation type admitting a nonsymplectic birational automorphism $f$ such that $\rho_X(f)$ is stable of order 32 and minimal polynomial $\Phi_1\Phi_{32}$. 
\end{rem}

\begin{proof}[Proof of \Cref{mainth1}]
    We use \Cref{bad order all actions,bad order stable actions}, as well as \Cref{rem bad orders}, to restrict the values of $m$ depending on $\T$. We then apply \Cref{necsuf composite,necsuf prime power} to the remaining orders to determine the genera of $M^g$ and $M_g$. For each known deformation type $\T\neq \text{K3}$ and each possible order $m$, we record the previous pairs of genera in \Cref{tab u4,tab u5,tab k3n,tab og10}. Combining this with \Cref{rem perm}, \Cref{class 2 div} and \Cref{propo class gen herm}, we are able to conclude that given $m$, the genus of $M^g$ and the genus of $M_g$, there exists a unique $O(M)$-conjugacy class of cyclic subgroup $\langle g'\rangle\leq O(M)$ such that $g'$ has minimal polynomial $\Phi_1\Phi_m$ and $M^{g'}$ and $M_{g'}$ are in the same genera as $M^g$ and $M_g$, respectively.
\end{proof}

For the case of K3 surfaces, we do not display our result in a proper table since their algebraically trivial automorphisms are known (see for instance \cite{ast11} and the reference therein for the case of prime order, and \cite{kon92,tak12} for the other orders).

\subsection{Nontrivial discriminant action}
We recall that according to \Cref{nontriv: all}, there are only finitely many deformation types $\T$ for which there exists an algebraically trivial nonsymplectic automorphism whose action on cohomology has finite order $m$ and nontrivial discriminant action. The possible pairs $(\T, m)$ were given in \Cref{tab:nontrivact}. Furthermore, as for the case of trivial discriminant actions, we can already discard some pairs $(\T, m)$.

\begin{propo}\label{propo: discard nontriv act}
    Let $(\T, m)$ be one of
    \[(\textnormal{OG10}, 54),\; \left(\textnormal{K3}^{[4]}, 54\right),\; \left(\textnormal{K3}^{[12]}, 22\right),\;\left(\textnormal{K3}^{[18]}, 34\right),\;\left(\textnormal{K3}^{[20]}, 38\right),\;(\textnormal{Kum}_2, 6),\; (\textnormal{Kum}_4, 10),\;(\textnormal{Kum}_6, 14).\]
    There exists no IHS manifold $X\sim\T$ with a nonsymplectic birational automorphism whose action on $H^2(X, \mathbb{Z})$ is nonstable of finite order $m$, and with minimal polynomial $\Phi_1\Phi_m$.
\end{propo}

\begin{proof}
    Let $(\T, m)$ be one of the pairs given in the statement. Let $X\sim \T$, and suppose that there exists a nonsymplectic $f\in \Bir(X)$  such that $h\coloneqq \rho_X(f)$ is nonstable of finite order $m$ and minimal polynomial $\Phi_1\Phi_m$. First note the following: The restriction of $h$ to $\Lambda_\T^{\Phi_m(h)}$ has determinant 1 according to \cite[Lemma~2.21]{bc22}, meaning that $\det(h) = +1$. In that case, we know that $\det(h)\cdot D_h = -\id_{D_{\Lambda_\T}}$. Thus $h\notin \Mon^2(\Lambda_\T)$ for $\T = \textnormal{Kum}_2,\, \textnormal{Kum}_4,\, \textnormal{Kum}_6$ (see \Cref{tab:extdata}). Hence, for the three last pairs $(\T, m)$ of the statement, such a pair $(X, f)$ cannot exist.
    
    Now, as in the proof of \Cref{bad order all actions}, we let $g \coloneqq  \gamma(h)\in O(M_\T)$; the pair $(M_\T, g)$ is a $\Phi_1\Phi_2\Phi_m^\ast$-lattice, and the $\Z$-lattice $M_\T^{\Phi_m(g)} = H^2(X, \Z)^{\Phi_m(h)}$ has signatures $(2, k)$ for some $k\geq 1$ (because $f$ is nonsymplectic). Now according to \Cref{gen stat 3 odd}, since $m=2p^l$ is twice an odd prime power,  there exists an $n_0\geq 0$ such that
    \begin{equation}\label{tmporary equation}
    4\equiv \varphi(p^l)(n_0+1)+2\left(\frac{-2}{p}\right)-1-p\mod 8.
    \end{equation}
    Since moreover $\varphi(m) = \varphi(p^l)$ divides $2+k\leq 22$ (see \Cref{tab:extdata}), we see that $n_0 = 0$ except when $(\T, m) = (\textnormal{K3}^{[12]}, 22)$, where $n_0\in \{0,1\}$. In particular, it follows that in all cases, the right-hand side of \Cref{tmporary equation} is $(0,2 \mod 8)$, which is absurd. Hence such a pair $(X, f)$ cannot exist.  
\end{proof}

\begin{proof}[Proof of \Cref{mainth2}]
    According to \Cref{main thhhh} and \Cref{rem perm}, it is enough to classify $\Mon^2(\Lambda_\T)$-conjugacy classes of nonstable isometries $h\in\Mon^2(\Lambda_\T)$ of finite order $m$ and minimal polynomial $\Phi_1\Phi_m$ such that $\Lambda_\T^{\Phi_m(h)}$ has signatures $(2, \ast)$. For this, we proceed as in the proof of \Cref{mainth1}. We  first use  \Cref{nontriv: all,propo: discard nontriv act} to restrict to the pairs $(\T, m)$ for which such an isometry $h$ can exist. Let us fix such a pair $(\T, m)$ for the sake of the proof. Let $h\in \Mon^2(\Lambda_\T)$ be a nonstable isometry of order $m$ and minimal polynomial $\Phi_1\Phi_m$. We  moreover let $g\coloneqq  \gamma(h)\in \Mon^2(\Lambda_\T)$. Using \Cref{gen stat 3 odd,gen stat 3 even}, we determine the possible genera for the kernel sublattices $M_\T^g$ and $M_\T^{\Phi_m(g)}$, based on the fact that the isometry class of $M_\T^{-g}$ is fixed by $\T$. We observe in particular that in all cases, $M_\T^g$ is unique in its genus, up to isometry (see \Cref{disc non triv} for a description of the genus of $M_\T^g$ in all cases, by means of symbols). Moreover, the hermitian structure of $M_\T^{\Phi_m(g)}$ is indefinite or of rank 1 by the assumption on the real signatures of $M_\T^{\Phi_m(g)}$. Therefore, using \Cref{class 3 div} and \Cref{propo class gen herm}, we infer that there is exactly one $O(M)$-conjugacy class of finite-order isometries $g'\in O(M)$ such that $M_\T^g$ and $M_\T^{g'}$ are in the same genus, and such that the hermitian structures of $(M_\T^{\Phi_m(g)}, g_m)$ and $(M_\T^{\Phi_m(g')}, g_m')$ are in the same genus, except when $m=46$, where there are three such classes. Using the fact that $O(M_\T^{\Phi_m(g)}, g_m) = \SO(M_\T^{\Phi_m(g)}, g_m)$, see \cite[Lemma 2.21]{bc22}, one can actually show in the proof of \Cref{class 3 div} that the $O(M)$-conjugacy class and the $\SO(M)$-conjugacy class of~$g$ are the same. We conclude by applying \Cref{th: corr2} together with \Cref{rem after corr 2} and \Cref{cosets 1}. The results are presented in \Cref{disc non triv}.
\end{proof}

\begin{notat}
    The $\Z$-lattices described in \Cref{disc non triv} are given in terms of a representative of their isometry class. We try as much as possible to choose a representative that is a direct sum of (rescaled) ADE root lattices and (rescaled) hyperbolic plane lattices $U$. However, we sometimes have to resort to other well-known $\Z$-lattices that do not fit in the previous list. In particular, we fix
    \[H_5 \coloneqq  {\scriptstyle{\begin{pmatrix}
        2&1\\1&-2
    \end{pmatrix}}}, \quad K_7 \coloneqq  {\scriptstyle{\begin{pmatrix}
        -4&1\\1&-2
    \end{pmatrix}}}, \quad K_{23} \coloneqq  {\scriptstyle{\begin{pmatrix}
        -12&1\\1&-2
    \end{pmatrix}}}.\]
    Moreover, we denote by $L_8^5$ and $L_8^{13}$ negative-definite $\mathbb{Z}$-lattices of rank 8 and determinant 5 and 13, respectively.
\end{notat}

\begin{rem}
    The square of the isometries representing the entries of \Cref{disc non triv} are known already, as well as the involutions they induced. However, it is more difficult to construct geometric examples realizing any of these cases.
\end{rem}

\subsection{Examples}\label{subsec examples}
There are examples of IHS manifolds that are equipped with particular nonsymplectic involutions with nontrivial discriminant actions. Such involutions can be used to produce examples of nonsymplectic automorphisms of higher even order and with nontrivial discriminant actions. For instance, we show in the first two 
following examples that we can realize at least two of the cases in \Cref{disc non triv} starting from large families of cubic fourfolds.

\begin{ex}[$(\textnormal{OG10}, 6)$ with $\Lambda^g = U$]\label{first example}
    In \cite[Theorem 3.8]{gal21}, the authors show the existence of a 10-dimensional family $\mathcal{C}$ of smooth cubic fourfolds $V$, all equipped with an automorphism $g\in\Aut(V)$ of order~3. For any such pair $(V, g)$, the induced action of $g$ on $H^4(V, \mathbb{Z})$ only fixes the square $h^2$ of the hyperplane class (see \cite[Example 6.4]{bcs16} or \cite[Table 4]{bg24}). For a very general $V\in\mathcal{C}$, the primitive algebraic lattice $A(V)_{\prim} \coloneqq  (h^2)^\perp\leq H^4(V,\mathbb{Z})\cap H^{2,2}(V)$ of $V$ is therefore trivial. Consequently,  any \emph{LSV tenfold} $X_V$ associated to a very general $V$ in $\mathcal{C}$, of deformation type OG10, has Picard rank 2 with $\NS(X_V)\simeq U$ (see \cite{lsv17} for a definition and construction of $X_V$). According to \cite[Section~3.1]{sac23}, the automorphism $g$ induces a nonsymplectic birational automorphism $\tilde{g}\in\Bir(X_V)$ of order 3, and in our context its action on cohomology fixes $\NS(X_V)$ pointwise. Therefore, $\tilde{g}$ is algebraically trivial.

    By construction $X_V$ comes equipped with a Lagrangian fibration $\pi\colon X_V\to (\mathbb{P}^5)^\vee$ extending the Donagi--Markman fibration on the locus of smooth hyperplane sections of $V$. There  therefore exists a birational involution $\tau$ on $X_V$ given by $-1$ on the smooth fibers of $\pi$ (see \cite{sac23}). This involution is known to be nonsymplectic, it is nonstable, and it commutes with all induced birational automorphisms. Moreover, the associated invariant sublattice $H^2(X_V,\mathbb{Z})^{\tau^{\ast}}\leq \NS(X_V)$ is isometric to $U$ (see \cite[Lemma 3.5]{sac23}). Hence, in our particular example, the map $\tau$ is algebraically trivial, and moreover $\tau\circ\tilde{g}$ is an algebraically trivial purely nonsymplectic automorphism of order 6, with nontrivial discriminant action and associated invariant sublattice isometric to $U$.
\end{ex}

\begin{ex}[$(\textnormal{K3}^{[4]}, 6)$ with $\Lambda^g \simeq \langle2\rangle$]
    Let $\mathcal{C}$ be the 10-dimensional family of smooth cubic fourfolds from \Cref{first example}, and let $V\in\mathcal{C}$ be very general and not containing a plane. We again denote  by $g\in\Aut(V)$ the automorphism of order 3 considered in \Cref{first example}. In \cite[Section~4.1]{cc19}, the authors show that the IHS eightfold $Z_V\sim \textnormal{K3}^{[4]}$ associated to $V$, known as the \emph{LLSvS eightfold} (see \cite{llsvs17} for a definition and construction of $Z_V$), has an induced nonsymplectic automorphism $\tilde{g}\in\Aut(Z_V)$ whose action on $H^2(Z_V, \mathbb{Z})$ only fixes the polarization $\delta$ of the projective manifold $Z_V$. Their argument actually shows that $\NS(Z_V)\simeq H^2(Z_V, \mathbb{Z})^{\tilde{g}^{\ast}}\simeq \langle2\rangle = \mathbb{Z}\delta$. Note that the polarization satisfies $\delta\in 2H^2(Z_V,\mathbb{Z})^\vee$, and in particular $H^2(Z_V, \mathbb{Z}) = \NS(Z_V)\oplus T_{Z_V}$, where $T_{Z_V}\simeq U^{\oplus2}\oplus A_2\oplus E_8^{\oplus2}$. Following \cite[Section~5.2]{lpz23}, the reflection $\tau$ defined by $\delta$ is a nonsymplectic Hodge monodromy, and it clearly commutes with the induced action of $\tilde{g}$ on $H^2(Z_V, \mathbb{Z})$. By its description, we also have that $D_\tau$ is nontrivial. Hence, the composition $\tau\circ \tilde{g}^\ast$ is a Hodge monodromy, and it is realized as an algebraically trivial purely nonsymplectic automorphism of order 6 on $Z_V\sim \textnormal{K3}^{[4]}$ with associated invariant lattice $\langle2\rangle$. This also shows that the involution of \cite[Lemma~3.7]{llms18} whose cohomological action coincides with $\tau$ actually commutes with the induced automorphism $\tilde{g}$.
\end{ex}

\begin{ex}[$(\textnormal{OG6}, 8)$ with $\Lambda^g \simeq U\oplus \langle-2\rangle^{\oplus2}$]\label{ex: anna}
In \cite[Theorem 1.4]{gro22b}, Grossi proves the existence of an OG6-type IHS manifold $X$ equipped with a nonsymplectic involution $\tau$ such that $H^2(X, \mathbb{Z})^{\tau^{\ast}}\simeq U\oplus \langle-2\rangle^{\oplus2}$ and $H^2(X, \mathbb{Z})_{\tau^{\ast}}\simeq U\oplus U(2)$.
The author also shows, using some numerical criteria, that $X$ is a \emph{numerical moduli space} (see \cite[Definition 3.2]{gro22b}) and the involution $\tau$ can be geometrically realized as induced from an associated abelian surface.
Moreover, the author shows that $X$ admits an \emph{MRS birational model} (see \cite{mrs18,gro22b}), which means that $X$ is birational to the quotient of a hyperk\"ahler manifold $Y\sim \textnormal{K3}^{[3]}$ by a birational symplectic involution $i$, after contracting the 256 copies of $\mathbb{P}^3$ along which $i$ is not defined. 

It turns out that on this birational model, the involution $\tau$ can be realized as induced from a nonsymplectic involution on $Y$. 
Since the numerical criteria from \cite[Theorem 1.3]{gro22b} are independent of the order and the action on the discriminant group, we obtain that up to deformation, the involution $\tau$ is the fourfold iterate of an algebraically trivial nonsymplectic automorphim of order 8, with the same invariant and coinvariant sublattices. 
Note that in \Cref{disc non triv} there exists a pair $(\textnormal{K3}^{[3]}, 8)$ with the same coinvariant sublattice $U\oplus U(2)$; however, there is no reason for this to be linked with the previous automorphism of order 8 on~$X$. 
In fact, since we know the action on cohomology for the birational symplectic involution on $Y\sim \textnormal{K3}^{[3]}$ from the MRS model, one can actually show that $Y$ cannot admit an algebraically trivial nonsymplectic automorphism of order 8 with such a lattice action.
\end{ex}

\begin{ex}[$(\textnormal{K3}^{[3]}, 4)$ with $\Lambda^g \simeq \langle4\rangle$]
    In \cite[Theorem 5.2]{bmw24}, the authors show that the IHS sixfold $X\sim \textnormal{K3}^{[3]}$ constructed by Debarre and Mongardi in \cite[Proposition 3.2]{dm22} admits a purely nonsymplectic automorphism $\tilde{\delta}$ of order 4. Such a manifold $X$ is constructed as an explicit example of \emph{double EPW-cubes} (see \cite{ikkr19} for a definition and construction of $X$), and it is by construction equipped with a nonsymplectic involution $\iota$. It turns out that $\tilde{\delta}^2 = \iota$. Such an automorphism $\tilde{\delta}$ is nonstable, and the pair $(X, \tilde{\delta})$ deforms in a 10-dimensional family whose general member is equipped with an algebraically trivial order 4 nonsymplectic automorphism, whose invariant sublattice is isometric to $\langle4\rangle$. It is not clear, though, whether such a general member is also a double EPW-cube.
\end{ex}

\begin{rem}
    From \Cref{disc non triv}, we observe that the case 
    \[(\T, m, \Lambda^g) = \left(\textnormal{K3}^{[24]}, 46, \langle2\rangle\right)\]
    could also deserve some attention. In fact, in such a case, there is a canonical invariant \emph{polarization} (\textit{i.e.}, class of positive square), and it is of small degree. We do not know a general description for IHS manifolds of deformation type $\textnormal{K3}^{[24]}$ and polarization of degree 2 and divisibility 2. Note that similarly to what was proved in \cite[Section~6]{bcms16}, there exists a unique IHS manifold $X\sim\textnormal{K3}^{[24]}$ equipped with a nonstable purely nonsymplectic automorphism of order 46. It would be interesting to know whether one could relate such a manifold and the one from \cite[Section~6]{bcms16} via the \emph{strange duality} observed in \cite[Proposition~3.9]{bbfp24}.
\end{rem}

\appendix

\section{Cyclotomic fields}
For the purpose of this paper, we recall and prove some results that we need about cyclotomic fields. 
Some of these results might already be known to the experts, and we make them available for the reader's convenience. One can find more results about cyclotomic fields in \cite{was97}.%\smallskip

\subsection{General facts}
Let $m \geq 3$ be an integer, and let $E \coloneqq  \mathbb{Q}(\zeta_m)$ be the $\supth{m}$ cyclotomic field. 
It contains a maximal subfield~$K$ of index 2, which is generated by $\zeta_m+\zeta_m^{-1}$ over $\mathbb{Q}$. 
The rings $\mathcal{O}_E \coloneqq  \mathbb{Z}[\zeta_m]$ and $\mathcal{O}_K \coloneqq  \mathbb{Z}[\zeta_m+\zeta_m^{-1}]$ are maximal orders of $E$ and $K$, respectively. 
The extension $E/K$ is Galois, and $\Gal(E/K)$ is generated by the complex conjugation mapping $\iota\colon\zeta_m\mapsto \iota(\zeta_m) \coloneqq  \zeta_m^{-1}$. 
Moreover, $E/K$ is CM, and in particular all the infinite places of $K$ are real. 
For any place $\mathfrak{q}$ of $K$, we denote by $K_{\mathfrak{q}}$ and $E_{\mathfrak{q}} = E\otimes_KK_{\mathfrak{q}}$ the respective $\mathfrak{q}$-adic completions of $K$ and $E$.
In order to study hermitian lattices over the extension $E/K$, we would like to know which prime ideals ramify in $E/\mathbb{Q}$ and $E/K$, respectively.

{\samepage\begin{lem}[\textit{cf.} {\cite[Propositions 2.3 and 2.15]{was97}}]\label{lem ramif}\leavevmode
    \begin{enumerate}
        \item A prime number $p\in\mathbb{Z}$ ramifies in $E/\mathbb{Q}$ if and only if $p$ divides $m$.
        \item A prime ideal $\mathfrak{p}$ of\, $\mathcal{O}_K$ ramifies in $E/K$ if and only if $m$ is a power $($or twice a power\,$)$ of a prime number $p$ and $\mathfrak{p}$ divides $p\mathcal{O}_K$.
    \end{enumerate}
  \end{lem}
  }

As we have seen in \Cref{prelim hermlat} about hermitian lattices, if a prime ideal does not ramify in $E/K$, then the local isometry class of a unimodular hermitian $\mathcal{O}_E$-lattice at the corresponding finite place is uniquely determined by the rank of the lattice.
Moreover, regarding \Cref{propo: case 1 minpoly,propo: case 2 minpoly}, the hermitian lattices we are interested in are locally unimodular at each place associated to prime $\mathcal{O}_K$-ideals that are not ramified. 
Therefore, the most interesting part for us is to study hermitian lattices over cyclotomic fields locally at a ramified finite place of $K$. 
By \Cref{lem ramif}, such a prime ideal exists in $E/K$, and it is unique if and only if $m$ is a prime power (or twice a prime power).

\begin{propo}\label{lem basic cyclo}
Suppose that $m = p^k\geq 3$ for some prime number $p$ and some positive integer $k$.
Let us set $\zeta \coloneqq  \zeta_{p^k}$ and $\pi \coloneqq  1-\zeta$. 
All of the following hold:
\begin{enumerate}
    \item\label{lbc-1} $\mathfrak{P} \coloneqq  \pi\mathcal{O}_E$ is a prime ideal and $p\mathcal{O}_E = \mathfrak{P}^{\varphi(m)}$.
    \item\label{lbc-2} $\mathfrak{p} \coloneqq  \mathfrak{P}\cap \mathcal{O}_K$ is generated by $\pi\iota(\pi)$. It is the only prime $\mathcal{O}_K$-ideal that ramifies in $E/K$. 
    \item\label{lbc-3} $E_{\mathfrak{p}}/K_{\mathfrak{p}}/\mathbb{Q}_p$ is a tower of totally ramified extensions. In particular, $\mathcal{O}_E/\mathfrak{P} \cong \mathcal{O}_K/\mathfrak{p}\cong \mathbb{F}_p$. 
    \item\label{lbc-4} The absolute different $\mathfrak{D}_{E/\mathbb{Q}}$ is principal generated by $\pi^{p^{k-1}(pk-k-1)}$, and the relative one $\mathfrak{D}_{E/K}$ by $\pi$ if $p$ is odd, and by $\pi^2$ otherwise.
\end{enumerate}
\end{propo}

\begin{proof}
(\ref{lbc-1})~ This is a generalization of \cite[Lemma 1.4]{was97} from prime order to prime-power order.

(\ref{lbc-2})~   We have that $\mathfrak{p} =  N^E_K(\mathfrak{P})$, and in particular $\mathfrak{p}$ is generated by $N^E_K(\pi) = \pi\iota(\pi)$. 
        The ramification statement follows from \Cref{lem ramif}.

(\ref{lbc-3})~ We have that $[E:\mathbb{Q}] = \varphi(m)$, so $E_\mathfrak{p}/\mathbb{Q}_p$ is totally ramified according to (1). 
        \textit{A fortiori}, so is the tower $E_\mathfrak{p}/K_\mathfrak{p}/\mathbb{Q}_p$. 
        In particular, we have that $\mathcal{O}_E/\mathfrak{P}\cong \mathcal{O}_K/\mathfrak{p}\cong \mathbb{F}_p$.

(\ref{lbc-4})~ Since $p\mathbb{Z}$ is the only prime ideal that ramifies in $E/\mathbb{Q}$, we get that $\mathfrak{P}$ is the only prime $\mathcal{O}_E$-ideal that divides $\mathfrak{D}_{E/\mathbb{Q}}$. 
        Now, the norm of the latter ideal is equal to the absolute discriminant of $E$; by \cite[Proposition 2.1]{was97}, it is equal to $p^{p^{k-1}(pk-k-1)}$. 
        Let $\alpha \coloneqq  \val_\mathfrak{P}(\mathfrak{D}_{E/\mathbb{Q}})$. Since the norm of $\mathfrak{P}$ over $\mathbb{Q}$ is exactly $p$, we have that $\alpha = p^{k-1}(pk-k-1)$, as expected.
        Now, the minimal polynomial of $\zeta$ over $K$ is $\mu(t) \coloneqq  t^2-(\zeta+\zeta^{-1})t+1\in K[t]$, so the relative different $\mathfrak{D}_{E/K}$ is generated by $\mu'(\zeta) = \zeta-\zeta^{-1} = \zeta\iota(\pi)(1+\zeta^{-1})$. 
        We conclude by remarking that $\zeta$ is a unit in $\mathcal{O}_E$, $\iota(\pi)\in\iota(\mathfrak{P}) = \mathfrak{P}$, and moreover, $(1+\zeta^{-1}) \in \mathfrak{P}$ if and only if $[2] = [0] \in\mathcal{O}_E/\mathfrak{P}\cong \mathbb{F}_p$.
\end{proof}

\subsection{Local norms}Let $m = p^k\geq 3$ be a prime power, and let $\zeta \coloneqq  \zeta_m$. 
Since the prime ideal $\mathfrak{p}\subseteq \mathcal{O}_K$ above $p\mathbb{Z}$ ramifies in $E/K$, we have that $E_\mathfrak{p}/K_\mathfrak{p}$ is a totally ramified degree 2 extension of local fields over $\mathbb{Q}_p$. 
For the proofs of \Cref{theo: prime order,existence even power}, we need to decide whether $-1$ is in the image of the norm map
\[ N^{E_\mathfrak{p}}_{K_\mathfrak{p}}\colon E_\mathfrak{p}\lra K_\mathfrak{p}.\]
We call an element of $K\cap \im(N^{E_\mathfrak{p}}_{K_\mathfrak{p}})$ a \emph{local norm at $\mathfrak{p}$}.
We aim to prove the following theorem.

\begin{theo}\label{th -1 loc norm}
    Let $m = p^k \geq 3$ be a prime power. Then $-1$ is a local norm at $\mathfrak{p}$ if and only if $\varphi(m)\equiv 0\mod 4$.
\end{theo} 

\begin{rem}
    For $p$ odd, we have that $\varphi(p^k)\equiv 0\mod 4$ if and only if $p\equiv 1\mod 4$. 
    For nontrivial powers $2^k$ of 2, we have that $\varphi(2^k)\equiv0\mod 4$ if and only if $k\geq 3$.
\end{rem}

In order to prove \Cref{th -1 loc norm}, we separate the odd case and the case of powers of 2. 
First, let us remark the following: The field $E$ is generated over $K$ by $\zeta-\zeta^{-1}$, whose square is 
\[w \coloneqq  \left(\zeta-\zeta^{-1}\right)^2 = \zeta^2+\zeta^{-2}-2 = \left(\zeta+\zeta^{-1}\right)^2-4\in K.\]
In particular, $E = K(\sqrt{w})$. 
Hence, according to \cite[Example 63:10]{om71}, $-1$ is a local norm at $\mathfrak{p}$ if and only if the {\em local Hilbert symbol} $(-1, w)_\mathfrak{p}$ is $1$ (see \cite[Section~63B]{om71} for a definition). 
So our problem reduces to a computation of Hilbert symbols in local fields. 
In the case where $p$ is odd, we have the following.

\begin{lem}\label{lem -1 p odd}
    Suppose that $p$ is odd. 
    Then $-1$ is a local norm at $\mathfrak{p}$ if and only if $p\equiv 1\mod 4$.
\end{lem}

\begin{proof}
    If $p$ is odd, then 
    \[w = \left(\zeta+\zeta^{-1}\right)^2-4 = \left(\zeta+\zeta^{-1}-2\right)\left(\zeta+\zeta^{-1}+2\right) = -\pi\iota(\pi)\left(2+\zeta+\zeta^{-1}\right)\in\mathfrak{p}\setminus \mathfrak{p}^2\]
    and $\val_\mathfrak{p}(w) = 1$. 
    Hence by \cite[Corollary 5.6]{voi13}, we have that $(-1, w)_\mathfrak{p} = \left(\frac{-1}{\mathfrak{p}}\right)$, and therefore, $-1$ is a norm in $K_\mathfrak{p}$ if and only if $-1$ is a square in $\mathcal{O}_{K_\mathfrak{p}}/\mathfrak{p} \cong \mathbb{F}_p$. 
    The latter holds if and only if $p\equiv 1\mod 4$.
\end{proof}

Now suppose that $p=2$; for convenience, we write $\beta \coloneqq  \pi\iota(\pi)$ for a generator of $\mathfrak{p}$. 
Since we choose $m = 2^k\geq 3$, we have that $k\geq 2$ and $\zeta^{2^{k-2}}$ is a square root of $-1$.
But since $E/K$ is CM, the number field $K$ does not contain $\zeta^{2^{k-2}}$, and we can write $E = K(\zeta^{2^{k-2}})$.
In particular, from now on we consider $w = -1$ and $E = K(\sqrt{-1})$.
This time, to compute $(-1, -1)_\mathfrak{p}$, we follow an approach of Kirschmer using {\em quadratic defects} (see \cite[Algorithm 3.1.3]{kir16})---alternatively, one could use an algorithm of Voight \cite[Section~5]{voi13}.

\begin{defin}[\textit{cf.} {\cite[Section~63A]{om71}, \cite[Definition 3.1.1]{kir16}}]
    Let $\mathbb{K}$ be a local field, and let $a\in \mathbb{K}$.
    The \emph{quadratic defect} $\mathfrak{d}(a)$ of $a$ is the fractional $\mathcal{O}_\mathbb{K}$-ideal defined by
    \[\mathfrak{d}(a) \coloneqq  \bigcap_{b\in \mathbb{K}}(a-b^2)\mathcal{O}_\mathbb{K}.\]
\end{defin}

We recall the setup in which we use these quadratic defects: The integer $m = 2^k\geq 3$ is a power of 2, $\zeta$ is a primitive $\supth{m}$ root of unity, $K = \mathbb{Q}(\zeta+\zeta^{-1})$ is totally real with $[K:\mathbb{Q}] = 2^{k-2}$, and $\mathfrak{p}$ is the unique prime $\mathcal{O}_K$-ideal dividing $2\mathcal{O}_K = \mathfrak{p}^{[K:\mathbb{Q}]}$. 
We set $\mathbb{K} = K_\mathfrak{p}$, and we consider $a = -1\in K_\mathfrak{p}$.
\begin{claim}\label{propo quad def}
We have that    $\val_\mathfrak{p}(\mathfrak{d}(-1)) = 2^{k-1} - 1$.
\end{claim}

\begin{proof}[Proof of \Cref{propo quad def}]
    According to \cite[Lemma 3.1.2]{kir16}, we have that $\mathfrak{d}(-1)$ is the smallest among the ideals
    \[ (0)\subsetneq 4\mathcal{O}_{K_\mathfrak{p}}\subsetneq 4\mathfrak{p}^{-1}\subsetneq 4\mathfrak{p}^{-3}\subsetneq \cdots\subsetneq \mathfrak{p}\]
    such that $-1$ is a square modulo $\mathfrak{d}(-1)$. 
    We aim to prove that $\mathfrak{d}(-1) = 4\mathfrak{p}^{-1}$; since $\val_\mathfrak{p}(4) = 2\val_\mathfrak{p}(2) = 2^{k-1}$, we can then conclude. 
    First of all, we recall that $\beta = 2-\zeta-\zeta^{-1}$ is a generator of $\mathfrak{p}\mathcal{O}_{K_\mathfrak{p}}$; in particular, the ideal $4\mathfrak{p}^{-1}$ is generated by $\alpha \coloneqq  \frac{4}{2-\zeta-\zeta^{-1}}$. 
    Now, using the fact that $\zeta^{2^{k-1}} = -1$, one can show that
    \[
        \alpha-1 = -\frac{(1+\zeta)^2}{(1-\zeta)^2} = -\left(\zeta+\ldots+\zeta^{2^{k-1}-1}\right)^2 = \left(-1 - \displaystyle\sum_{i = 1}^{2^{k-2}-1}(\zeta^i+\zeta^{-i})\right)^2\in \left(K_\mathfrak{p}^\times\right)^2.
    \]
    Hence, $\alpha-1$ is a square in $K_\mathfrak{p}$, meaning that $\mathfrak{d}(-1)\subseteq 4\mathfrak{p}^{-1}$. 
    To conclude, we remark the following:
    \begin{enumerate}
        \item Since $E_\mathfrak{p} = K_{\mathfrak{p}}(\sqrt{-1})$, we see that $X^2+1 \in K_\mathfrak{p}[X]$ is irreducible and so $\mathfrak{d}(-1) \neq (0)$.
        \item If $\mathfrak{d}(-1) = 4\mathcal{O}_{K_\mathfrak{p}}$, by \cite[Theorem 3.1.7-5]{kir16}, we would have that $K_\mathfrak{p}(\sqrt{-1}) = E_\mathfrak{p}$ is unramified over $K_{\mathfrak{p}}$; this is absurd because $E_\mathfrak{p}/K_\mathfrak{p}$ is totally ramified.\qedhere
    \end{enumerate}
\end{proof}

We are now ready to prove the equivalent of \Cref{lem -1 p odd} in the case $p=2$.
\begin{lem}
    Suppose that $p=2$.
    Then $-1$ is a local norm at $\mathfrak{p}$ if and only if $m = 2^k \geq 8$.
\end{lem}
\begin{proof}
Since $2^{k-1}-1$ is odd for all $k\geq 2$, \Cref{propo quad def} together with \cite[Lemma 3.1.2--4]{kir16} tell us that there exists a unit $u\in -(\mathcal{O}_{K_\mathfrak{p}}^\times)^2$ such that $\mathfrak{d}(u) = \mathfrak{d}(-1) = 4\mathfrak{p}^{-1}$ and $\val_\mathfrak{p}(1-u) = 2^{k-1}-1$. 
Since $-u\in (\mathcal{O}_{K_\mathfrak{p}}^\times)^2$ is a square, we get that $(-u, -1)_\mathfrak{p} = (1, -1)_\mathfrak{p} = 1$, and by the multiplicativity of Hilbert symbols, we obtain 
\begin{equation}\label{step 1}
    (u, -1)_\mathfrak{p} = (-1, -1)_\mathfrak{p}.
\end{equation}
By the definition of the Hilbert symbols, we know that $(u, 1-u)_\mathfrak{p} = 1$. 
Combining this with the fact that $2^{k-1}-2$ is even for $k\geq 2$, we have that $\beta^{2^{k-1}-2}\in (\mathcal{O}_{K_\mathfrak{p}}^\times)^2$, and we deduce that
\begin{equation}\label{step 2}
    \left(u, \frac{u-1}{\beta^{2^{k-1}-2}}\right)_\mathfrak{p} = \left(u, u-1\right)_\mathfrak{p} = \left(u, -1\right)_\mathfrak{p}\left(u, 1-u\right)_\mathfrak{p}=(u, -1)_\mathfrak{p}
\end{equation}
with $\val_\mathfrak{p}\left(\frac{u-1}{\beta^{2^{k-1}-2}}\right) = \val_\mathfrak{p}(u-1)-2^{k-1}+2 = 1$. 
We set $c \coloneqq  \frac{u-1}{\beta^{2^{k-1}-2}}$. 
By \Cref{step 1,step 2}, the scalar $c$ satisfies $(u, c)_\mathfrak{p} = (u, -1)_\mathfrak{p} = (-1, -1)_\mathfrak{p}.$
Moreover, one can compute
\begin{align*}
    \left(u\left(1-\beta^{2^{k-1}-2}c\right),\, c\right)_\mathfrak{p} &= (u,\,c)_\mathfrak{p}\left(1-\beta^{2^{k-1}-2}c,\,c\right)_\mathfrak{p} = (u,c)_\mathfrak{p}\left(1-\beta^{2^{k-1}-2}c,\, \beta^{2^{k-1}-2}c\right)_\mathfrak{p} = (u, c)_\mathfrak{p},
\end{align*}
and $u(1-\beta^{2^{k-1}-2}c) - 1 = u(2-u) - 1 = -(u-1)^2$ has $\mathfrak{p}$-valuation $2^k-2$. Hence, there are two cases.

If $2^k = 4$, then $\val_\mathfrak{p}(u(2-u)-1) = 2 = \val_\mathfrak{p}(4)$. 
So there exists a unit $\delta\in \mathcal{O}_{K_\mathfrak{p}}^\times$ such that $u(2-u) = 1-4\delta$. 
In particular, $u(2-u)$ is a square modulo $4\mathcal{O}_{K_\mathfrak{p}}$, and either $\mathfrak{d}(u(2-u)) = (0)$ or $\mathfrak{d}(u(2-u)) = 4\mathcal{O}_{K_\mathfrak{p}}$. 
The former is possible, by Hensel's lemma, if and only if $1-4\delta\equiv (1+2\alpha)^2\mod 4\mathfrak{p}$ for some $\alpha\in \mathcal{O}_{K_\mathfrak{p}}$, or equivalently $\delta\equiv \alpha^2+\alpha\mod\mathfrak{p}$. 
However, the previous congruence system has a solution in $\mathcal{O}_{K_\mathfrak{p}}$ only if $X^2+X+1$ has a solution in $\mathcal{O}_{K_\mathfrak{p}}/\mathfrak{p}\cong \mathbb{F}_2$, which is not true (see also \cite[Algorithm 3.1.3]{kir16} for a general argument).
Hence $\mathfrak{d}(u(2-u)) = 4\mathcal{O}_{K_\mathfrak{p}}$, and since $c$ has $\mathfrak{p}$-valuation 1, \cite[63:11a]{om71} 
tells us that $(u(2-u), c)_\mathfrak{p} = -1$. 
We therefore conclude that $(-1, -1)_\mathfrak{p} = (u(2-u), c)_\mathfrak{p} = -1$ and that $-1$ is not a local norm at $\mathfrak{p}$.

Otherwise, if $2^k \geq 8$, we have that $\val_\mathfrak{p}(u(2-u)-1) =2^k-2 > 2^{k-1} = \val_\mathfrak{p}(4)$. 
By \cite[63:2]{om71},
this implies that $\mathfrak{d}(u(2-u)) = (0)$, and $u(2-u)$ is therefore a square in $K_\mathfrak{p}$.
In that case, $(-1, -1)_\mathfrak{p} = (u(2-u), c)_\mathfrak{p} = 1$.
\end{proof}

\subsection{Congruence classes of units} Again, let $m = p^k$ be a prime power, $\zeta \coloneqq  \zeta_m$ and $E \coloneqq  \mathbb{Q}(\zeta)$. For $\pi \coloneqq  1-\zeta$, we recall that $\mathfrak{P} = \pi\mathcal{O}$ is the unique prime $\mathcal{O}_E$-ideal dividing $p\mathcal{O}_E = \mathfrak{P}^{\varphi(p^k)}$. 
To end this section on results related to cyclotomic fields, we study the congruence classes of some units in $E$ modulo ideals of the form $1+\mathfrak{P}^i$ for $i\geq 1$. 
Indeed, in order to prove \Cref{lem intermediaire}, or more precisely \Cref{strong approx th}, we need to count the number of classes of units of norm 1 in $E$ modulo such ideals.
Let $\iota\in\Gal(E/K)$ be the generator, and let $\mathcal{F}(E) \coloneqq  \{e\in \mathcal{O}_E^\times: e\iota(e) = 1\}$ be the group of units of norm 1 in $E$. 
By \cite[Theorem 4.12 and Corollary 4.13]{was97}, the set $\mathcal{F}(E)$ coincides with the group $\mu(E)$ of roots of unity in $E$. 
In particular,  
\(\mathcal{F}(E) = \{\pm \zeta^a: 0\leq a\leq p^k-1\}\), 
and it has order $\lcm(2, p^k)$. 
For all $1\leq j\leq p^k$, we set $\mathcal{F}_j(E) \coloneqq  \ker(\mathcal{F}(E)\to \mathcal{O}^\times_E/(1+\mathfrak{P}^j))$.

\begin{lem}\label{ramif tower}
    Let $m = p^k\geq 3$ be a prime power. For all $1\leq a\leq p^k-1$, we have that 
    \[\val_{\mathfrak{P}}(1-\zeta^a) = p^{\val_p(a)}.\]
\end{lem}

\begin{proof}
    Let us write $a = p^lb$, where $\gcd(p,b) = 1$, and let $\xi \coloneqq  \zeta^a$. The algebraic number $\xi$ is a primitive $\supth{(p^{k-l})}$ root of unity, and the extension $E/\mathbb{Q}(\xi)$ has degree 
    \[ \frac{[E:\mathbb{Q}]}{[\mathbb{Q}(\xi):\mathbb{Q}]} = \frac{p^{k-1}(p-1)}{p^{k-l-1}(p-1)} = p^l.\]
    Note that $(1-\xi)\mathcal{O}_{\mathbb{Q}(\xi)}$ is the unique prime ideal of $\mathbb{Q}(\xi)$ lying over $p\mathbb{Z}$ (see \Cref{lem basic cyclo}). This implies that $(1-\xi)\mathcal{O}_{\mathbb{Q}(\xi)}$ totally ramifies in $E$ and thus $(1-\xi)\mathcal{O}_E = \mathfrak{P}^{[E:\mathbb{Q}(\xi)]} = \mathfrak{P}^{p^l}$. We can conclude by remarking that $l = \val_p(a)$.
\end{proof}

\begin{coro}\label{lem cyclo unit}
  Let $m = p^k\geq 3$ be a prime power. 
  For all $1\leq i\leq k$ and for all $p^{i-1}< j\leq p^{i}$,
   \(\#\mathcal{F}_j(E) = p^{k-i}\)
   and $\#\mathcal{F}_1(E) = p^k$.
\end{coro}

\begin{proof}
First of all, we remark that for all $1\leq a\leq p^k-1$, the residue class $\overline{1-\zeta^a}\in \mathcal{O}_E/\mathfrak{P}$ is trivial, since $\mathfrak{P} = (1-\zeta)\mathcal{O}_E$. 
For a similar reason, the residue class $\overline{1+\zeta^a}\in \mathcal{O}_E/\mathfrak{P}$ is trivial if and only if $p = 2$. 
This already tells us that
\( \#\mathcal{F}_1(E) = p^k.\)

We conclude the rest of the proof by invoking \Cref{ramif tower}, which tells us that for all $1\leq i\leq k$ and for all $p^{i-1}< j \leq p^i$, the following hold:
  \[ \#\mathcal{F}_j(E) = \#\mathcal{F}_{p^{i}}(E) = \#\left\{1\leq a\leq p^k-1: \val_p(a) \geq i\right\} + 1= p^{k-i}.\qedhere\]
\end{proof}

\section{Cyclotomic hermitian Miranda--Morrison theory}
In \cite[Proposition 2.20]{bc22}, the authors determine when two isometries $f,g\in O(M)$ of an even unimodular $\mathbb{Z}$-lattice $M$, with minimal polynomial $\Phi_1\Phi_p$ for some odd prime number $p$, are conjugate. 
They show, thanks to strong approximation, that for an indefinite $p$-elementary $\Phi_p$-lattice $(L, f)$ with $D_f$ trivial, the discriminant representation $O(L, f)\to O(D_L)$ is surjective. 
In this section, we show that their result holds in greater generality using a refinement of their argument that one can find in \cite[Section~6]{bh23}, based on a hermitian analog of Miranda--Morrison theory; see \cite{mm09}. We actually aim to prove the following theorem.

\begin{theo}\label{strong approx th}
    Let $m = p^k\geq 3$ for some prime number $p$, and let $(L, f)$ be an even $p$-elementary $\Phi_{p^k}$-lattice. 
    Assume that $D_f$ has order at most $2$.
    If\, $L$ has rank $\varphi(p^k)$ or is indefinite, then
    \[\pi_{L,f}\colon O(L, f) \lra O\left(D_L, D_f\right)\]
    is surjective.
\end{theo}

Before going through the proof of this theorem, let us remark the following easy cases.

\begin{lem}
    Let $m, L, f$ be as in the statement of \Cref{strong approx th}. 
    If\, $L$ is unimodular or has rank $\varphi(m)$, then $\pi_{L,f}\colon O(L,f)\to O(D_L, D_f)$ is surjective.
\end{lem}

\begin{proof}
    If $L$ is unimodular, it is necessarily true. 
    Now let us assume that $L$ is not unimodular. 
    By \Cref{propo: info hermitian structure}, either $D_f$ is trivial and $D_L\cong \mathbb{Z}/p\mathbb{Z}$ as abelian groups, or $p=2$ and $D_L\cong (\mathbb{Z}/2\mathbb{Z})^{\oplus2}$. 
    In the former case, $O(D_L, D_f) = O(D_L)$ is generated by $-\id_{D_L} = \pi_{L, f}(-\id_L)$. 
    In the latter case, when $p = 2$ and $D_f$ has order 2, the isometry $D_f$ generates $O(D_L)$, and it acts by exchanging the generators of both copies of $\mathbb{Z}/2\mathbb{Z}$ in $D_L$ (see \Cref{rem induced herm action}).
    Again, we obtain that $O(D_L, D_f) = O(D_L)$, and it is generated by $\pi_{L, f}(f)$. In both cases, $\pi_{L, f}$ is surjective.
\end{proof}

Now let $m = p^k$, and let $(L, b, f)$ be an even $p$-elementary $\Phi_m$-lattice as in the statement of \Cref{strong approx th}. 
Let us now assume  that $(L, b)$ is indefinite, not unimodular and of rank larger than $\varphi(p^k)$. 
We let $(L, h)$ be the associated hermitian structure, and we  again set $\zeta \coloneqq  \zeta_m$, $E \coloneqq  \mathbb{Q}(\zeta)$, $K \coloneqq  \mathbb{Q}(\zeta+\zeta^{-1})$ and $\pi \coloneqq  1-\zeta$. 
Moreover, we let $\mathfrak{P}$ and $\mathfrak{p}$ be the unique prime ideals of $\mathcal{O}_E$ and $\mathcal{O}_K$, respectively, lying above $p\mathbb{Z}$. 
We recall that $\mathcal{O}_E \coloneqq  \mathbb{Z}[\zeta]$ and $\mathcal{O}_K \coloneqq  \mathbb{Z}[\zeta+\zeta^{-1}]$ are respective maximal orders of $E$ and $K$. 
We  again denote by $\iota\in \Gal(E/K)\cong \Gal(E_\mathfrak{p}/K_\mathfrak{p})$ a generator. 
For the reader's convenience, in order to  introduce   hermitian Miranda--Morrison theory as described in \cite[Section~6]{bh23}, we fix some further notation (see \cite[Section~6.5]{bh23} for more details).
We denote by $\mathcal{O}_{\mathbb{A}_K}$ the ring of integral finite adeles of $K$. We let $U^\#(L, h)$ be the kernel of the map $U(L, h) = O(L, b, f)\to O(D_L, D_f)$, and we define $U^\#(L_\mathfrak{q})$ similarly for each finite place $\mathfrak{q}$ of $K$.
For $i\geq 1$, we define the following groups:
\[\renewcommand\arraystretch{1.2}
\begin{array}{ll}
    \bullet\,\;\mathcal{F}(E) \coloneqq  \left\{e\in \mathcal{O}_E^\times\colon e\iota(e) = 1\right\}, & \bullet\,\;\mathcal{F}_i(E) \coloneqq  \left\{e\in \mathcal{F}(E)\colon e\equiv 1\mod \mathfrak{P}^i\right\},\\
    
    \bullet\,\;\mathcal{F}\left(E_\mathfrak{p}\right)\coloneqq  \left\{e\in\mathcal{O}^\times_{E_\mathfrak{p}}\colon e\iota(e) = 1\right\},&  \bullet\,\;\mathcal{F}_i\left(E_\mathfrak{p}\right) \coloneqq  \left\{ e\in \mathcal{F}\left(E_\mathfrak{p}\right)\colon e\equiv 1\mod\mathfrak{P}^i\right\},\\
    
    \bullet\,\;\mathcal{F}(L, h) \coloneqq  \det\left(U\left((L, h)\otimes_{\mathcal{O}_K} \mathcal{O}_{\mathbb{A}_K}\right)\right), & \bullet\,\;\mathcal{F}\left(L_\mathfrak{p}\right) \coloneqq  \det\left(U\left(L_\mathfrak{p}\right)\right)\subseteq \mathcal{F}\left(E_\mathfrak{p}\right),\\
    
    \bullet\,\;\mathcal{F}^\#(L, h) \coloneqq  \det\left(U^\#\left((L, h)\otimes_{\mathcal{O}_K} \mathcal{O}_{\mathbb{A}_K}\right)\right),&\bullet\,\; \mathcal{F}^\#\left(L_\mathfrak{p}\right) \coloneqq  \det\left(U^\#\left(L_\mathfrak{p}\right)\right).
\end{array}\]
In what follows, by abuse of notation, we sometimes identify $\mathcal{F}(E)$ with its image along the map
\(E\to \prod_{\mathfrak{q}\;\text{finite}}E_\mathfrak{q}$, $ e\mapsto (e)_\mathfrak{q}.\)
In \cite[Theorem 6.15]{bh23}, the authors prove that the following sequence is exact: 
\begin{equation}\label{exact hermmirmor}
O(L,f)\xrightarrow{\pi_{L, f}} O\left(D_L, D_f\right)\overset{\delta}\lra \mathcal{F}(L,h)/(\mathcal{F}(E)\cap\mathcal{F}(L,h))\cdot \mathcal{F}^{\#}(L,h)\lra 1,
\end{equation}
where $\delta$ is induced by the determinant morphism $\det\colon U((L, h)\otimes_{\mathcal{O}_K} \mathcal{O}_{\mathbb{A}_K})\to \prod_{\mathfrak{q}\;\text{finite}}E_\mathfrak{q}$.
In order to show that $O(L, f)\to O(D_L, D_f)$ is surjective, it is therefore equivalent to show that $\delta$ is trivial. 
Since by assumption $(L, b)$ is $p$-elementary, and $\mathfrak{p}$ is the unique prime $\mathcal{O}_K$-ideal lying above $p\mathbb{Z}$, \cite[Remark 6.16]{bh23} gives us the following group isomorphism: 
\[\mathcal{F}(L,h)/(\mathcal{F}(E)\cap \mathcal{F}(L,h))\cdot \mathcal{F}^\#(L,h) \cong \mathcal{F}\left(L_{\mathfrak{p}}\right)/\left(\mathcal{F}(E)\cap \mathcal{F}\left(L_{\mathfrak{p}}\right)\right)\cdot \mathcal{F}^\#\left(L_{\mathfrak{p}}\right) =: I_{\mathfrak{p}}.\]
Hence, the surjectivity of $\pi_{L, f}$ is equivalent to $I_{\mathfrak{p}}$ being trivial, by the exactness of the sequence \eqref{exact hermmirmor}. 
Following the idea of \cite[Section~6.8]{bh23}, we actually show in what follows that the quotient 
\[Q_\mathfrak{p} \coloneqq  \mathcal{F}\left(E_\mathfrak{p}\right)/\mathcal{F}(E)\cdot \mathcal{F}^\#\left(L_\mathfrak{p}\right)\]
is trivial (as long as $(L, h)$ satisfies the assumptions of \Cref{strong approx th}). Since $I_\mathfrak{p}$ can be identified with a subgroup of $Q_\mathfrak{p}$, we then conclude.

According to \cite[Theorem 6.25]{bh23}, there exists an integer $j_0\geq 0$ such that $\mathcal{F}^\#(L_{\mathfrak{p}}) = \mathcal{F}_{j_0}(E_{\mathfrak{p}})$.
Note that for a fixed $j\geq 0$, the quotient $\mathcal{F}(E_{\mathfrak{p}})/\mathcal{F}_{j}(E_{\mathfrak{p}})$ is naturally isomorphic (see \cite[Section~6.8]{bh23}) to the kernel of the norm map 
\begin{equation}\label{norm equa}
    N^{E_\mathfrak{p}}_{K_\mathfrak{p}}\colon\mathcal{O}^\times_{E_\mathfrak{p}}/\left(1+\mathfrak{P}^{j}\right) \lra \mathcal{O}^\times_{{K_\mathfrak{p}}}/\left(1+\mathfrak{p}^{\phi_p(j)}\right),
\end{equation}
where we define $\phi_p(j) \coloneqq  \frac{j+1}{2}$ if $j$ is odd, and $\phi_p(j) \coloneqq  \frac{j+(2\val_{\mathfrak{P}}(\mathfrak{D}_{E/K})-2)}{2}$ if $j$ is even. 
We recall that by \Cref{lem basic cyclo}, we know that $e \coloneqq  \val_{\mathfrak{P}}(\mathfrak{D}_{E/K}) = \gcd(2,p)$.

\begin{lem}[\textit{cf.} {\cite[Corollaries V.3.1, V.3.2 and V.3.5]{ser04}}]\label{serre corollary}
        For all $j\geq 0$, the norm map in \Cref{norm equa} induces
        \[N_j\colon\left(1+\mathfrak{P}^{j}\right)/\left(1+\mathfrak{P}^{j+1}\right) \lra \left(1+\mathfrak{p}^{\phi_p(j)}\right)/\left(1+\mathfrak{p}^{\phi_p(j)+1}\right),\]
        which is bijective except if $j = e-1$,
            in which case there is an exact sequence
            \[1 \lra C_2\lra \left(1+\mathfrak{P}^{e-1}\right)/\left(1+\mathfrak{P}^{e}\right) \lra \left(1+\mathfrak{p}^{e-1}\right)/\left(1+\mathfrak{p}^{e}\right)\lra C_2\lra 1.\]
        Note that we define $1+\mathfrak{P}^0 \coloneqq  \mathcal{O}_{E_\mathfrak{p}}^\times$, and similarly $1+\mathfrak{p}^0 \coloneqq  \mathcal{O}_{K_\mathfrak{p}}^\times$.
    \end{lem}

\begin{rem}[\textit{cf.} {\cite[Proposition IV.2.6]{ser04}}]\label{rem serre}
    As abelian groups, $\mathcal{O}^\times_{E_\mathfrak{p}}/(1+\mathfrak{P})\cong \mathbb{F}_p^\times$,
    and for all $j\geq 1$,
    $(1+\mathfrak{P}^j)/(1+\mathfrak{P}^{j+1})\cong \mathbb{F}_p.$
\end{rem}

The upshot now is that $Q_\mathfrak{p}$ is trivial if and only if the sequence of abelian groups
\begin{equation}\label{eq8}
    \mathcal{F}(E) \lra \mathcal{O}^\times_{E_\mathfrak{p}}/\left(1+\mathfrak{P}^{j}\right) \lra \mathcal{O}^\times_{{K_\mathfrak{p}}}/\left(1+\mathfrak{p}^{\phi_p(j )}\right)
\end{equation}
is exact for $j=j_0$. 
Therefore, if we can show that the sequence \eqref{eq8} is exact for given values of $m = p^k$ and $j=j_0$, then we will have proven \Cref{strong approx th}.
    
Since we work with finite groups, our plan is to show that the sequence~\eqref{eq8} is exact for given $m$ and $j$ by showing that the image of the first map and the kernel of the second map have the same order. 
Note that for all values of $m$ and $j$, the kernel $\mathcal{F}_{j}(E)$ of the first map is known thanks to \Cref{lem cyclo unit}, and so we know the size of its image. 
Now, in  order to continue, we need to determine from the assumptions of \Cref{strong approx th} which values for $j_0$ we have to consider.

According to \cite[Theorem 6.25]{bh23}, the value of $j_0\geq 0$ for which $\mathcal{F}^\#(L_{\mathfrak{p}}) = \mathcal{F}_{j_0}(E_{\mathfrak{p}})$ is given by $j_0 = 2n+\alpha$, where $\alpha \coloneqq  \val_{\mathfrak{P}}(\mathfrak{D}_{E/\mathbb{Q}}) = p^{k-1}(pk-k-1)$ and $n \coloneqq  \val_{\mathfrak{p}}(\mathfrak{n}(L, h))$. 
Since by assumption the trace form $(L, b)$ of $(L, h)$ is even, we have that $i \leq 2n\leq i+e$, where $i \coloneqq  \val_{\mathfrak{P}}(\mathfrak{s}(L))$ and $e = \val_{\mathfrak{P}}(\mathfrak{D}_{E/K})$; see \cite[Section~6.6]{bh23}.
According to \Cref{propo: decomp block val}, since we assumed that $D_f$ has order at most 2, we have that $-\alpha\leq i\leq \gcd(2,p)-\alpha$. 
In particular, in the setting of \Cref{strong approx th}, we have that either  $p$ is odd and $0\leq j_0\leq 2$, or $p=2$ and $0\leq j_0\leq 4$.

\begin{lem}\label{lem intermediaire}
    Recall that $m = p^k$ is a prime power, and let $j\geq 0$ be such that
    \begin{enumerate}
        \item\label{li-1} $0\leq j\leq 2$ if $p$ is odd, 
        \item\label{li-2} $0\leq j\leq 4$ if $p = 2$.
    \end{enumerate}
    Then the sequence \eqref{eq8} is exact.
\end{lem}

\begin{proof}
   First of all, note that if $j=0$, then $\mathcal{O}^\times_{E_\mathfrak{p}}/(1+\mathfrak{P}^j)$ is trivial, so there is nothing to prove.

(\ref{li-1})~   Now suppose that $p$ is odd, and let $1\leq j\leq 2$. By \Cref{lem cyclo unit}, we know that the image of $\mathcal{F}(E) \to \mathcal{O}^\times_{E_\mathfrak{p}}/(1+\mathfrak{P}^j)$ has order $2$ for $j=1$ and $2p$ for $j=2$. 
        \begin{enumerate}[label=(\alph*), ref=\alph*]
            \item\label{li-1a} For $j=1$, we have $\phi_p(1) = 1$, and according to \Cref{serre corollary}, the kernel of 
    \[ N_0\colon\mathcal{O}^\times_{E_\mathfrak{p}}/(1+\mathfrak{P}) \lra \mathcal{O}^\times_{{K_\mathfrak{p}}}/(1+\mathfrak{p})\] has order 2, and thus the sequence~\eqref{eq8} is exact. 
    \item For $j=2$, remark that $\phi_p(2) = 1$ too, and that the norm map 
    \(\mathcal{O}^\times_{E_\mathfrak{p}}/(1+\mathfrak{P}^2) \to\mathcal{O}^\times_{{K_\mathfrak{p}}}/(1+\mathfrak{p})\)
    factors as
    \[\mathcal{O}^\times_{E_\mathfrak{p}}/\left(1+\mathfrak{P}^2\right) \lra \mathcal{O}^\times_{E_\mathfrak{p}}/(1+\mathfrak{P}) \xrightarrow{N_0} \mathcal{O}^\times_{{K_\mathfrak{p}}}/(1+\mathfrak{p}),\]
    where the first arrow is surjective and with kernel $(1+\mathfrak{P})/(1+\mathfrak{P}^2)$. 
    Hence, by \Cref{serre corollary} and \Cref{rem serre}, we have that the kernel of  $\mathcal{O}^\times_{E_\mathfrak{p}}/(1+\mathfrak{P}^2) \to \mathcal{O}^\times_{{K_\mathfrak{p}}}/(1+\mathfrak{p})$
    has order $ 2p$.
        \end{enumerate}
    
 (\ref{li-2})~  Let us now assume  that $p=2$, and let $1\leq j\leq 4$.
    Again, by \Cref{lem cyclo unit}, we know that the image of $\mathcal{F}(E) \to \mathcal{O}^\times_{E_\mathfrak{p}}/(1+\mathfrak{P}^j)$ is trivial for $j=1$, it has order $2$ for $j=2$, and it has order $4$ for $j=3,4$. 
    
    \begin{enumerate}[label=(\alph*), ref=\alph*]
        \item For $j=1$, we have that $\phi_2(1) = 1$, and the map $N_0\colon\mathcal{O}^\times_{E_\mathfrak{p}}/(1+\mathfrak{P}) \to \mathcal{O}^\times_{{K_\mathfrak{p}}}/(1+\mathfrak{p})$ is an isomorphism whenever $p=2$ (see \Cref{serre corollary}).
    
        \item For $j=2$, we have that $\phi_2(2) = 2$, and there is a commutative diagram with exact rows
    \[\begin{tikzcd}
          1\arrow[r]&(1+\mathfrak{P})/\left(1+\mathfrak{P}^2\right)\arrow[d, "N_1"]\arrow[r]&\mathcal{O}^\times_{E_\mathfrak{p}}/\left(1+\mathfrak{P}^2\right)\arrow[r]\arrow[d]&  \mathcal{O}^\times_{E_\mathfrak{p}}/(1+\mathfrak{P})\arrow[d, "N_0", "\cong"']\arrow[r]&1\\
         1\arrow[r]&(1+\mathfrak{p})/\left(1+\mathfrak{p}^2\right)\arrow[r]&
        \mathcal{O}^\times_{{K_\mathfrak{p}}}/\left(1+\mathfrak{p}^2\right)\arrow[r,]&
         \mathcal{O}^\times_{{K_\mathfrak{p}}}/(1+\mathfrak{p})\arrow[r]&1\rlap{.}
    \end{tikzcd}\]
    According to \Cref{serre corollary}, the vertical map $N_0$ is an isomorphism, and $\ker N_1$ has order 2. 
    Therefore, the kernel of the middle vertical map is isomorphic to $\ker N_1$, and it has order $2$. 
    \item For the case $j=3$, we again have  that $\phi_2(3) = 2$, and following the idea for the odd prime-order case~(\ref{li-1})(\ref{li-1a}), we have that the norm map
    \(\mathcal{O}^\times_{E_\mathfrak{p}}/(1+\mathfrak{P}^3) \to\mathcal{O}^\times_{{K_\mathfrak{p}}}/(1+\mathfrak{p}^2)\)
    factors as
    \[\mathcal{O}^\times_{E_\mathfrak{p}}/\left(1+\mathfrak{P}^3\right) \lra \mathcal{O}^\times_{E_\mathfrak{p}}/\left(1+\mathfrak{P}^2\right) \lra \mathcal{O}^\times_{{K_\mathfrak{p}}}/\left(1+\mathfrak{p}^2\right).\]
    By similar arguments, we deduce that the kernel of $\mathcal{O}^\times_{E_\mathfrak{p}}/(1+\mathfrak{P}^3) \to \mathcal{O}^\times_{{K_\mathfrak{p}}}/(1+\mathfrak{p}^2)$ has order $4$. 
    \item Finally, when $j=4$, we have that $\phi_2(4) = 3$, and there is a commutative diagram with exact rows
    \[\begin{tikzcd}
          1\arrow[r]&\left(1+\mathfrak{P}^3\right)/\left(1+\mathfrak{P}^4\right)\arrow[r]\arrow[d, "N_3", "\cong"']&\mathcal{O}^\times_{E_\mathfrak{p}}/\left(1+\mathfrak{P}^4\right)\arrow[r]\arrow[d]&  \mathcal{O}^\times_{E_\mathfrak{p}}/\left(1+\mathfrak{P}^3\right)\arrow[r]\arrow[d]&1\\
         1\arrow[r]&\left(1+\mathfrak{p}^2\right)/\left(1+\mathfrak{p}^3\right)\arrow[r]&\mathcal{O}^\times_{{K_\mathfrak{p}}}/\left(1+\mathfrak{p}^3\right)\arrow[r]&
        \mathcal{O}^\times_{{K_\mathfrak{p}}}/\left(1+\mathfrak{p}^2\right)\arrow[r]&1\rlap{.}
    \end{tikzcd}\]
    According to \Cref{serre corollary}, the vertical map $N_3$ is an isomorphism, and by the snake lemma, we have that the kernels of the two other vertical maps are isomorphic.
    Hence, the kernel of $\mathcal{O}^\times_{E_\mathfrak{p}}/(1+\mathfrak{P}^4) \to \mathcal{O}^\times_{{K_\mathfrak{p}}}/(1+\mathfrak{p}^3)$ has order 4. 
    \end{enumerate}
    Therefore, for $p=2$ and $1\leq j\leq 4$, the sequence in \Cref{eq8} is exact, and that concludes the proof.
\end{proof}

Hence, \Cref{lem intermediaire} together with the prior discussions proves \Cref{strong approx th}.

\section{Tables of results}\label{app:C}

\subsection{Trivial discriminant action}
For the Mukai lattice $M=U^{\oplus 4}$ (see \Cref{tab u4}), we specify the smallest integer $k\geq 1$ for which the action can be realized as induced on a moduli space of $k$-twisted sheaves on an abelian surface (see \Cref{subsec induced descr}). If $k=1$, then the action can also be realized as induced on an OG6-type IHS manifold, and we set $k=0$ if the action cannot be induced, in the sense of \Cref{subsec induced descr}. We treat the case $M=U^{\oplus 4}\oplus E_8^{\oplus 2}$ similarly (see \Cref{tab k3n}). Note that for these two tables, whenever $k=1$,  the action can also be realized by \emph{natural automorphisms} (see \cite[D\'efinition 1]{boissiere1} and \cite[Section~3.1]{bnws11}). Finally, in the cases of order 4 for which $k=2$ in \Cref{tab k3n}, the first in the list is realizable for $\textnormal{K3}^{[n]}$-type IHS manifolds if and only if $n$ is odd, while the other one is realizable for all $n\geq 2$.

\renewcommand{\arraystretch}{1}
\begin{table}[!ht]
\caption{Trivial discriminant action---composite orders}
\begin{subtable}[t]{0.5\textwidth} 

    \caption{$\T = \text{Kum}_n\,(n\geq 2);\; M = U^{\oplus4}$}\label{tab u4}
\centering
\resizebox{.95\textwidth}{!}{
    \begin{tabular}{cccc}
        $m$&$M^g$&$M_g$&Induced\\
        \hline
        \rowcolor{lightgray!40!white}$4,6,12$& $\II_{(2,2)}$& $\II_{(2,2)}$&$1$\\
        4& $\II_{(2,2)}2^2$& $\II_{(2,2)}2^2$&$1$\\
        \rowcolor{lightgray!40!white}9& $\II_{(2,0)}3^{-1}$ & $\II_{(2,4)} 3^1$&0\\
        \hline
    \end{tabular}}
\end{subtable}
\hfill
\begin{subtable}[t]{0.4\textwidth} 
    \caption{$\T = \textnormal{OG6};\;M = U^{\oplus5}$}\label{tab u5}
\centering
\resizebox{.95\textwidth}{!}{
    \begin{tabular}{ccc}
        $m$&$M^g$&$M_g$\\
        \hline
        \rowcolor{lightgray!40!white}$4,6,12$& $\II_{(3,3)}$& $\II_{(2,2)}$\\
        4& $\II_{(3,3)}2^2$& $\II_{(2,2)}2^2$\\
        \rowcolor{lightgray!40!white}9& $\II_{(3,1)}3^{-1}$ & $\II_{(2,4)} 3^1$\\
        \hline
    \end{tabular}}
    \vspace*{1cm}
\end{subtable}
\bigskip   
\renewcommand{\arraystretch}{1.3}
\begin{subtable}[t]
{0.5\textwidth} 

    \caption{$\T = \text{K3}^{[n]}\,(n\geq 2);\;M=U^{\oplus4}\oplus E_8^{\oplus 2}$}\label{tab k3n}
\centering
    \resizebox{.95\textwidth}{!}{
    \begin{tabular}{cccc}
    
        $m$&$M^g$&$M_g$&Induced\\
        \hline
        \rowcolor{lightgray!40!white}$\substack{4,6,12,22\\33,44,66}$& $\II_{(2,2)}$& $\II_{(2,18)}$&$1$\\
        4& $\II_{(2,2)}2^2$& $\II_{(2,18)}2^2$&$1$\\
        \rowcolor{lightgray!40!white}4& $\II_{(2,2)}2^4$& $\II_{(2,18)}2^4$&$2$\\
        $4, 8, 16$& $\II_{(2,6)}2^{-2}$& $\II_{(2,14)}2^{-2}$&$1$\\
        \rowcolor{lightgray!40!white}$4, 8$& $\II_{(2,6)}2^{-4}$& $\II_{(2,14)}2^{-4}$&$1$\\
        4& $\II_{(2,6)}2^{-6}$& $\II_{(2,14)}2^{-6}$&$2$\\
        \rowcolor{lightgray!40!white}$\substack{4,6,9,12\\14,18,21\\28,36,42}$& $\II_{(2,10)}$& $\II_{(2,10)}$&$1$\\
        4& $\II_{(2,10)}2^2$& $\II_{(2,10)}2^2$&$1$\\
        \rowcolor{lightgray!40!white}4& $\II_{(2,10)}2^4$& $\II_{(2,10)}2^4$&$1$\\
        4& $\II_{(2,10)}2^6$& $\II_{(2,10)}2^6$&$1$\\
        \rowcolor{lightgray!40!white}$4, 8$& $\II_{(2,14)}2^{-2}$& $\II_{(2,6)}2^{-2}$&$1$\\
        4& $\II_{(2,14)}2^{-4}$& $\II_{(2,6)}2^{-4}$&$1$\\
        \rowcolor{lightgray!40!white}$4,6,12$& $\II_{(2,18)}$& $\II_{(2,2)}$&$1$\\
        4& $\II_{(2,18)}2^{2}$& $\II_{(2,2)}2^{2}$&$1$\\
        \rowcolor{lightgray!40!white}$9, 27$& $\II_{(2,4)}3^{1}$ & $\II_{(2,16)} 3^{-1}$&1\\
        9& $\II_{(2,4)}3^{-3}$ & $\II_{(2,16)} 3^3$&1\\
        \rowcolor{lightgray!40!white}9& $\II_{(2,10)}3^{-2}$ & $\II_{(2,10)} 3^{-2}$&1\\
        9& $\II_{(2,16)}3^{-1}$ & $\II_{(2,4)} 3^1$&1\\
        \rowcolor{lightgray!40!white}25& $\II_{(2,2)}5^{-1}$& $\II_{(2,18)}5^{-1}$&$1$\\
        \hline
    \end{tabular}}
    
\end{subtable}
\hfill
\begin{subtable}[t]{0.4\textwidth} 

    \caption{$\T = \textnormal{OG10};\;M=U^{\oplus5}\oplus E_8^{\oplus 2}$}\label{tab og10}
\centering
\resizebox{.95\textwidth}{!}{
  \begin{tabular}{ccc}
        $m$&$M^g$&$M_g$\\
        \hline
        \rowcolor{lightgray!40!white}$\substack{4,6,12,22\\33,44,66}$& $\II_{(3,3)}$& $\II_{(2,18)}$\\
        4& $\II_{(3,3)}2^2$& $\II_{(2,18)}2^2$\\
        \rowcolor{lightgray!40!white}4& $\II_{(3,3)}2^4$& $\II_{(2,18)}2^4$\\
        4& $\II_{(3,3)}2^6$& $\II_{(2,18)}2^6$\\
        \rowcolor{lightgray!40!white}$4,8, 16$& $\II_{(3,7)}2^{-2}$& $\II_{(2,14)}2^{-2}$\\
        $4,8$& $\II_{(3,7)}2^{-4}$& $\II_{(2,14)}2^{-4}$\\
        \rowcolor{lightgray!40!white}4& $\II_{(3,7)}2^{-6}$& $\II_{(2,14)}2^{-6}$\\
        4& $\II_{(3,7)}2^{-8}$& $\II_{(2,14)}2^{-8}$\\
        \rowcolor{lightgray!40!white}$\substack{4,6,9,12,14,18\\21,28,36,42}$& $\II_{(3,11)}$& $\II_{(2,10)}$\\
        4& $\II_{(3,11)}2^2$& $\II_{(2,10)}2^2$\\
        \rowcolor{lightgray!40!white}4& $\II_{(3,11)}2^4$& $\II_{(2,10)}2^4$\\
        4& $\II_{(3,11)}2^6$& $\II_{(2,10)}2^6$\\
        \rowcolor{lightgray!40!white}$4,8$& $\II_{(3,15)}2^{-2}$& $\II_{(2,6)}2^{-2}$\\
        4& $\II_{(3,15)}2^{-4}$& $\II_{(2,6)}2^{-4}$\\
        \rowcolor{lightgray!40!white}$4,6,12$& $\II_{(3,19)}$& $\II_{(2,2)}$\\
        4& $\II_{(3,19)}2^{2}$& $\II_{(2,2)}2^{2}$\\
        \rowcolor{lightgray!40!white}$9,27$& $\II_{(3,5)}3^{1}$ & $\II_{(2,16)} 3^{-1}$\\
        9& $\II_{(3,5)}3^{-3}$ & $\II_{(2,16)} 3^3$\\
        \rowcolor{lightgray!40!white}9& $\II_{(3,11)}3^{-2}$ & $\II_{(2,10)} 3^{-2}$\\
        9& $\II_{(3,17)}3^{-1}$ & $\II_{(2,4)} 3^1$\\
        \rowcolor{lightgray!40!white}25& $\II_{(3,3)}5^{-1}$& $\II_{(2,18)}5^{-1}$\\
        \hline
    \end{tabular}}
\end{subtable}
\end{table}

\subsection{Nontrivial discriminant action}

For each entry in \Cref{disc non triv}, there is a unique $\Mon^2(\Lambda)$-conjugacy class of subgroups $H := \langle h\rangle$ with the given lattice data except in the case $(\textnormal{K3}^{[24]}, 46)$, where there are three classes (see \Cref{class 3 div} and \Cref{propo class gen herm}).

\vfill
\pagebreak
\renewcommand{\arraystretch}{1.3}
\begin{table}[!ht]    
    \caption{Nontrivial discriminant action---twice prime powers}\label{disc non triv}\vspace*{0.3cm}
\centering
\rowcolors{1}{white}{lightgray!40!white}
\begin{adjustbox}{totalheight=\textheight-3\baselineskip}
%\resizebox*{!}{\dimexpr\textheight-2\baselineskip\relax}{%
\begin{tabular}{ccccc}
        \rowcolor{white}$\T$& $m$&$M^g$&$\Lambda^H$&$\Lambda_H$\\
        \hline
        \cellcolor{white}&4&$\II_{(3,1)}2^{-2}_6$&$U$&$U^{\oplus2}\oplus \langle -2\rangle^{\oplus2}$\\
        \cellcolor{white}&4&$\II_{(3,1)}2^{-2}_6$&$U(2)$&$U^{\oplus2}\oplus \langle -2\rangle^{\oplus2}$\\
        \cellcolor{white}&4&$\II_{(3,1)}2^{4}_2$&$U(2)$&$U\oplus U(2)\oplus \langle-2\rangle^{\oplus 2}$\\
        \cellcolor{white}&4&$\II_{(3,5)}2^2_6$&$U\oplus D_4$&$\langle2\rangle^{\oplus 2}$\\
        \cellcolor{white}\multirow{-5}{*}{OG6}&8&$\II_{(3,3)}2^{2}$&$U\oplus \langle-2\rangle^{\oplus2}$&$U\oplus U(2)$\\
        \hline
        &4&$\II_{(2,0)}2^2_2$&$\langle4\rangle$&$U^{\oplus2}\oplus E_8^{\oplus2}\oplus \langle-2\rangle^{\oplus2}$\\
        \cellcolor{white}&4&$\II_{(2,4)}2^2_6$&$U\oplus A_3$&$U^{\oplus2}\oplus E_7^{\oplus2}$\\
        &4&$\II_{(2,4)}2^{-4}_2$&$U(2)\oplus A_3$&$U\oplus U(2)\oplus E_7^{\oplus2}$\\
        \cellcolor{white}&4&$\II_{(2,8)}2^2_2$&$U\oplus D_7$&$U^{\oplus2}\oplus E_8\oplus \langle-2\rangle^{\oplus2}$\\
        &4&$\II_{(2,8)}2^4_2$&$U(2)\oplus D_7$&$U^{\oplus2}\oplus D_8\oplus \langle-2\rangle^{\oplus2}$\\
        \cellcolor{white}&4&$\II_{(2,8)}2^6_2$&$U(2)\oplus A_3\oplus D_4$&$U^{\oplus2}\oplus D_4^{\oplus2}\oplus \langle-2\rangle^{\oplus2}$\\
        &4&$\II_{(2,8)}2^{-8}_6$&$U\oplus E_7(2)$&$U\oplus U(2)\oplus D_4^{\oplus2}\oplus \langle-2\rangle^{\oplus2}$\\
        \cellcolor{white}&4&$\II_{(2,12)}2^2_6$&$U\oplus A_3\oplus E_8$&$U^{\oplus2}\oplus D_6$\\
        &4&$\II_{(2,12)}2^4_6$&$U\oplus A_3\oplus D_8$&$U^{\oplus2}\oplus D_4\oplus \langle-2\rangle^{\oplus2}$\\
        \cellcolor{white}&4&$\II_{(2,12)}2^6_6$&$U(2)\oplus A_3\oplus D_8$&$U^{\oplus 2}\oplus \langle-2\rangle^{\oplus6}$\\
        &4&$\II_{(2,16)}2^2_2$&$U\oplus D_7\oplus E_8$&$U^{\oplus2}\oplus \langle-2\rangle^{\oplus2}$\\
        \cellcolor{white}&4&$\II_{(2,16)}2^4_2$&$U\oplus A_3\oplus D_4\oplus E_8$&$U\oplus U(2)\oplus \langle-2\rangle^{\oplus2}$\\
        &4&$\II_{(2,20)}2^2_6$&$U\oplus A_3\oplus E_8^{\oplus2}$&$\langle2\rangle^{\oplus2}$\\
        \cellcolor{white}&8&$\II_{(2,2)}2^2$&$U\oplus \langle-4\rangle$&$U^{\oplus2}\oplus D_8\oplus E_8$\\
        &8&$\II_{(2,2)}2^4$&$U(2)\oplus \langle-4\rangle$&$U^{\oplus2}\oplus D_4^{\oplus2}\oplus E_8$\\
        \cellcolor{white}&8&$\II_{(2,10)}2^2$&$U\oplus E_8\oplus \langle-4\rangle$&$U^{\oplus2}\oplus D_8$\\
        &8&$\II_{(2,10)}2^4$&$U\oplus D_8\oplus \langle-4\rangle$&$U^{\oplus2}\oplus D_4^{\oplus2}$\\
        \cellcolor{white}&8&$\II_{(2,18)}2^2$&$U\oplus E_8^{\oplus2}\oplus\langle-4\rangle$&$U\oplus U(2)$\\
        &16&$\II_{(2,14)}2^{-2}$&$U\oplus D_5\oplus E_8$&$U^{\oplus2}\oplus D_4$\\
        \cellcolor{white}\multirow{-20}{*}{$\textnormal{K3}^{[3]}$}&32&$\II_{(2,6)}2^{-2}$&$U\oplus D_5$&$U^{\oplus2}\oplus D_4\oplus E_8$\\
        \hline
        \cellcolor{white}&6&$\II_{(3,1)}3^{-1}$&$U$&$U^{\oplus2}\oplus A_2\oplus E_8^{\oplus2}$\\
        
        \cellcolor{white}&6&$\II_{(3,9)}3^{-1}$&$U\oplus E_8$&$U^{\oplus2}\oplus A_2\oplus E_8$\\
        
        \multirow{-3}{*}{OG10}&$6,18$&$\II_{(3,17)}3^{-1}$&$U\oplus E_8^{\oplus 2}$&$U^{\oplus 2}\oplus A_2$\\
        \hline
        \cellcolor{white}&6&$\II_{(2,0)}3^{-1}$&
        $\langle2\rangle$&$U^{\oplus2}\oplus A_2\oplus E_8^{\oplus2}$\\
        \cellcolor{white}&6&$\II_{(2,8)}3^{-1}$&$E_8\oplus\langle2\rangle$&$U^{\oplus2}\oplus A_2\oplus E_8$\\
        \cellcolor{white}\multirow{-3}{*}{$\textnormal{K3}^{[4]}$}&$6, 18$&$\II_{(2,16)}3^{-1}$&$E_8^{\oplus2}\oplus\langle2\rangle$&$U^{\oplus2}\oplus A_2$\\
        \hline
        &$10,50$&$\II_{(2,2)}5^{-1}$&$U\oplus \langle-2\rangle$&$U^{\oplus2}\oplus E_8\oplus L_8^5$\\
        \cellcolor{white}&10&$\II_{(2,10)}5^{-1}$&$U\oplus E_8\oplus \langle-2\rangle$&$U^{\oplus2}\oplus L_8^5$\\
        \multirow{-3}{*}{$\textnormal{K3}^{[6]}$}&10&$\II_{(2,18)}5^{-1}$&$U\oplus E_8^{\oplus2}\oplus\langle-2\rangle$&$U\oplus H_5$\\
        \hline
        \cellcolor{white}$\textnormal{K3}^{[8]}$&14&$\II_{(2,16)}7^1$&$U\oplus E_7\oplus E_8$&$U^{\oplus2}\oplus K_7$\\
        \hline
        $\textnormal{K3}^{[14]}$&26&$\II_{(2,10)}13^{-1}$&$U\oplus E_8\oplus \langle-2\rangle$&$U^{\oplus2}\oplus L_8^{13}$\\
        \hline
        \cellcolor{white}$\textnormal{K3}^{[24]}$&46&$\II_{(2,0)}23^1$&$\langle2\rangle$&$U^{\oplus2}\oplus E_8^{\oplus2}\oplus K_{23}$\\
        \hline
\end{tabular}%}
\end{adjustbox}
\end{table}

\clearpage

%%%%%%%%%%%%%%%%%%%%%
% References
%%%%%%%%%%%%%%%%%%%%%

\newcommand{\etalchar}[1]{$^{#1}$}

\end{document}